\documentclass[final,3p,times]{elsarticle}
\usepackage{amsfonts}
\usepackage{amssymb,amsthm,epsfig}
\usepackage{amsmath}
\usepackage{graphicx,color}
\usepackage{dsfont,subfigure}

\newtheorem{thm}{Theorem}[section]
\newtheorem{lem}[thm]{Lemma}
\newtheorem{cor}[thm]{Corollary}
\newtheorem{Def}[thm]{Definition}
\newtheorem{prop}[thm]{Proposition}
\newtheorem{examp}[thm]{Example}

\begin{document}

\title{Local Kernels and the Geometric Structure of Data}
\author{Tyrus Berry}
\address{Dept. of Mathematics, Pennsylvania State University, University Park, PA 16802}
\author{Timothy Sauer}
\address{Dept. of Mathematical Sciences, George Mason University, Fairfax, VA 22030}

\begin{abstract}
%We develop the geometric approach to nonparametric modeling, in which observed data is assumed to lie near a smooth manifold.  The key development is the
We introduce a theory of \emph{local} kernels, which generalize the kernels used in the standard diffusion maps construction of nonparametric modeling.  
We prove that evaluating a local kernel on a data set gives a discrete representation of the generator of a continuous Markov process, which converges in the limit of large data. We explicitly connect the drift and diffusion coefficients of the process to the moments of the kernel. Moreover, when the kernel is symmetric, the generator is the Laplace-Beltrami operator with respect to a geometry which is influenced by the embedding geometry and the properties of the kernel. In particular, this allows us to generate any Riemannian geometry by an appropriate choice of local kernel. In this way, we continue a program of Belkin, Niyogi, Coifman and others to reinterpret the current diverse collection of kernel-based data analysis methods and place them in a geometric framework. We show how to use this framework to design local kernels invariant to various features of data.  These data-driven local kernels can be used to construct conformally invariant embeddings and reconstruct global diffeomorphisms.

%We show that, in the limit of large data, symmetric local kernels are equivalent to smooth geometries, in the sense that every local kernel defines a Riemannian geometry on the data manifold and moreover, every Riemannian geometry can be represented by a local kernel.  A central challenge in this approach to data science is to design a kernel which preserves desired features, which are called intrinsic, and is invariant to the remaining features, which are called extrinsic. 
\end{abstract}


\maketitle

%\chapter[Nonparametric modeling and the implicit geometry of kernel functions]

\section{Introduction}

The need to analyze massive data sets in Euclidean space has led to a proliferation of research activity, including methods of dimension reduction and manifold learning. In general, understanding large data means identifying intrinsic characteristics of the data and developing techniques to isolate them.

Various attempts have been made to generalize principal component analysis (PCA) for this purpose. For example, the method of Kernel PCA \cite{kpca1,kpca2} has led to large classes of kernels, which specify the degree of affinity between pairs of points. For a data set consisting of $N$ points, Kernel PCA constructs a symmetric positive-definite $N\times N$ matrix $K$ of inner products, and considers the eigenvectors as coordinates. The perspective taken by kernel PCA is that the distance defined by the inner product will be represented by Euclidean distance in $\mathbb{R}^N$, and taking only the first $M<N$ eigenvectors as coordinates, will optimally approximate these distances in $\mathbb{R}^M$. This will be successful for flat manifolds, but geodesic distances on general curved manifolds will not be preserved.  For example, a sphere cannot be mapped into a finite-dimensional Euclidean space in a way that translates geodesic distances into Euclidean distances.

While kernel PCA tries to understand the data by mapping it to another, usually high-dimensional feature space, an alternative approach attempts to encode structure through differential operators by assuming the data lies on a manifold. There has been a movement to reinterpret kernel PCA methods geometrically, as a form of manifold learning, for a particular class of kernels. 
Belkin and Niyogi \cite{BN} and Coifman and collaborators \cite{diffusion} focused on kernels that depend only on the distance between points in ambient space, and that have exponential decay with distance. They used these kernels to estimate the Laplacian on the manifold described by the data.  The Laplacian encodes all of the geometric information contained in the data. This differs from the interpretation of kernel PCA in two important ways: (1) the matrix $K$ is viewed as an approximation of a differential operator, and (2) the eigenvectors are approximations to the eigenfunctions of the operator, evaluated on the data set.

The goal of this article is to extend the geometric perspective to a wider class of kernels. In fact, we show that all kernels with exponential decay can be interpreted as defining a Laplacian with respect to some Riemannian geometry.  We refer to this wider class as {\it local} kernels, because all information must flow through local interactions due to the strong decay.  In particular, local kernels include any kernel with compact support.  The kernels of \cite{BN,diffusion} are local, but because they are radially symmetric and independent of location on the manifold, can only access the geometry inherited from the ambient space. Later work of Coifman and Singer et al. \cite{Singer2008,Singer2009,dsilva} considered kernels that were not radially-symmetric from a non-geometric standpoint, and these kernels are closely related to the prototypical local kernels introduced in Section \ref{localkernels}. Local kernels extend the results of \cite{Singer2008} to a much larger class of kernels and naturally give rise to an intrinsic geometry on the data. In particular, Theorems \ref{maintheorem} and \ref{mainconverse} show that every symmetric local kernel corresponds to a Riemannian geometry and conversely, any Riemannian geometry can be represented with an appropriate local kernel. This opens up all kernels with exponential decay to exploitation by the whole range of geometric tools.

Moreover, when the local kernel is not symmetric, we show the kernel approximates the generator of a Markov process on the manifold defined by the data. From this perspective we can view the local kernel as defining transition probabilities between points on the manifold. We will show that in the limit of large data, an appropriate local kernel can be used to recover the generator of an arbitrary It\^o process. This generalizes the views of \cite{diffdist,diffcoords,diffusionslowmanifold,diffusionreductioncoords,adaptedDiffusion,Ting2010} which connected the diffusion maps construction to the generator of a Markov process in the case of a gradient flow.  In Section \ref{localkernels}, we connect this theory to the theory of nonlinear independent components of It\^o processes, which was introduced in \cite{Singer2008} and applied in \cite{Singer2009,dsilva}.  

One promising application of local kernels is geometric regularization. Properties of embedded data that are considered extrinsic for a particular purpose can be removed. Reducing to intrinsic properties allows comparison and classification of different data sets.
In Section \ref{mainresultex}, we show how to construct local kernels that result in geometries that are invariant under conformal isometries. We then show how to reconstruct a global diffeomorphism using a correspondence between the data sets. One application of this technique is to the problem of merging multiple observations with different modalities.

In Section \ref{dmbackground} we summarize the relevant developments and techniques related to diffusion maps as found in \cite{BN,diffusion,SingerEstimate,Hein05,singerWu,BH14}.  In Section \ref{localkernels} we generalize the diffusion maps construction to a large class of kernels called local kernels and in Section \ref{kernelgeometry} we show that symmetric local kernels are equivalent to Riemannian metrics in the limit of large data. Section \ref{mainresultex} contains applications of local kernels.

\section{The Geometric Prior and Diffusion Maps}\label{dmbackground}

Our typical assumption is that we are presented with a finite set of points on or near a manifold embedded in a high-dimensional Euclidean space, but with no {\it a priori} knowledge of the underlying manifold.  We will assume the manifold to be a compact $d$-dimensional differentiable manifold $\mathcal{M} \subset \mathbb{R}^n$. This is a nonparametric model for our data, since we assume that the manifold exists but we do not assume any parametric form.  We think of this assumption as a \emph{geometric prior}.  Given the geometric prior, our goal is to learn the geometric structure of the data and exploit this structure to simplify and understand the data.   

A diffusion map to a lower-dimensional space is a method of representing the geometry of the data.  In rough analogy to the principal components from a singular value decomposition, the components of a diffusion map \cite{diffusion,diffcoords} are eigenvectors of a transition matrix for a random walk on the data set.  Under appropriate normalizations, the transition matrix is a discrete approximation to the Laplace-Beltrami operator, which encodes all the geometric features of the manifold inherited from the embedding \cite{laplacianBook}. 
% Thus the components of the diffusion map will be approximations to eigenfunctions of this operator and, as was shown in \cite{MaggioniThesis}, a diffusion map minimizes an energy functional that measures the distortion of the manifold's geometry (see also \cite{BerryThesis}).  

The transition matrix is constructed by evaluating a kernel $K(x,y)$ on all pairs from a data set.  This yields a square $N\times N$ matrix, where $N$ is the number of data points, which is a discrete representation of a continuous operator.  The goals of these kernel based techniques are threefold: (1) to describe the operator limit based on the chosen kernel, (2) to give techniques to construct a desired operator in terms of the kernel, and (3) to describe the convergence of the discrete representation to the continuous operator in the limit of large data.  

Assuming a kernel of the form $K_{\epsilon}(x,y) = h(||x-y||^2/\epsilon)$, where $h$ has exponential decay, the first two goals were achieved definitively in the work of Coifman and Lafon \cite{diffusion} and the final goal was achieved by Singer \cite{SingerEstimate}.  In particular, this theory can be used to approximate the Laplace-Beltrami operator for data sampled from a Riemannian manifold, with arbitrary sampling distribution.  The remaining restriction of this theory is the special form of the kernel $K_{\epsilon}$ and in Sections \ref{localkernels} and \ref{kernelgeometry} we give a far-reaching generalization of the existing theory.  

To begin, we briefly summarize the relevant results of \cite{diffusion,SingerEstimate}.
Given a data set $\{x_i\}_{i=1}^N \subset \mathbb{R}^n$ sampled from a d-dimensional Riemannian manifold $\mathcal{M} \subset \mathbb{R}^n$ with sampling density $q$, the diffusion maps algorithm produces a $N \times N$ matrix which approximates the Kolmogorov operator
\[ \mathcal{L}f = \Delta f + (2-2\alpha)\nabla f \cdot \frac{\nabla q}{q} \]
where $\alpha$ is a constant which can be chosen in the diffusion maps construction.  Note that $\Delta$ is the Laplacian operator (with negative eigenvalues) and $\nabla$ is the gradient operator, and each are taken with respect to the Riemannian metric inherited from the ambient space $\mathbb{R}^n$.  The key to understanding diffusion maps is that continuous notions such as functions and operators are made discrete by writing them in the basis of the data set itself.  Thus, a function $f$ is represented by a vector $[f] = (f(x_1),f(x_2),...,f(x_N))^\top$ and an operator $\mathcal{A}$ is represented by a $N\times N$ matrix $A$ such that $(A[f])_i = \mathcal{A}(f)(x_i)$.  With this intuition in mind, we construct a matrix $J_{\epsilon}$ which represents a Markov chain on the data set with transition probabilities using the definitions
\begin{align}
J_{\epsilon}(x_i,x_j) &= \exp\left\{-\frac{||x_i-x_j||^2}{4\epsilon}\right\} &\hspace{10pt} q_{\epsilon}(x_i) &= \sum_{j=1}^N J_{\epsilon}(x_i,x_j)   \nonumber \\
J_{\epsilon,\alpha}(x_i,x_j) &= \frac{J_{\epsilon}(x_i,x_j)}{q_{\epsilon}(x_j)^{\alpha}} &\hspace{10pt} q_{\epsilon,\alpha}(x_i) &= \sum_{j=1}^N J_{\epsilon,\alpha}(x_i,x_j) \nonumber \\ 
\hat J_{\epsilon,\alpha}(x_i,x_j) &= \frac{J_{\epsilon,\alpha}(x_i,x_j)}{ q_{\epsilon,\alpha}(x_i)} &\hspace{10pt} L_{\epsilon,\alpha} &= \frac{\hat J_{\epsilon,\alpha}-I}{\epsilon} \nonumber
\end{align}
The crucial theoretical result of diffusion maps \cite{diffusion} is that in the limit as $N\to \infty$ and $\epsilon \to 0$ we have $L_{\epsilon,\alpha} \to \mathcal{L}$ and $\hat J_{\epsilon,\alpha}^{t/\epsilon} \to e^{t\mathcal{L}}$, in the sense that for any sufficiently smooth function $f$ at any point $x_k$ in the data set we have $(L_{\epsilon,\alpha}[f])_k \to \mathcal{L}f(x_k)$ and $(\hat J_{\epsilon,\alpha}^{t/\epsilon}[f])_k \to e^{t\mathcal{L}}f(x_k)$.  Moreover, when $q=1$ is uniform, Singer \cite{SingerEstimate} shows that
\[ L_{\epsilon,0}f(x) = \mathcal{L}f(x) + \mathcal{O}\left(\epsilon,\frac{||\nabla f(x)||}{\sqrt{N} \epsilon^{1/2+d/4}} \right) \]
with high probability.  

Since the data points $\{x_i\}$ are sampled independently from the density $q$,  $q_{\epsilon}(x_i) \propto q(x_i) + \mathcal{O}(\epsilon)$, meaning that $q_{\epsilon}$ is a kernel density estimate of the invariant measure. In fact the diffusion maps theory is much more general, and allows any kernel $J_{\epsilon}(x,y) = h({||x-y||^2}/{\epsilon})$ such that the shape function $h:[0,\infty)\to[0,\infty)$ has exponential decay at infinity and finite $m \equiv \frac{1}{2}\int_{\mathbb{R}^d} z_1^2 h(||z||^2)dz/\int_{\mathbb{R}^d} h(||z||^2) dz$.  The constant $m$ is related to the moments of the shape function, and the only modification required to the above construction is that $\frac{1}{m}L_{\epsilon,\alpha} \to \mathcal{L}$.  Note that for the exponential kernel above we find $m=1$ because the exponential was chosen to have variance $2$.

The diffusion maps algorithm essentially evaluates the kernel $J_{\epsilon}$ on all pairs from the data set and then applies two normalizations.  The first normalization divides the columns of the $J_{\epsilon}$ matrix by the column sums, $q_{\epsilon}(x_j)$, to the power $\alpha$.  We will refer to this as a \emph{right-normalization} since it is equivalent to multiplying the matrix $J_{\epsilon}$ on the right with a diagonal matrix with diagonal entries $q_{\epsilon}(x_j)^{-\alpha}$.  Note that in \cite{diffusion} both the rows and columns are divided by the column sums in this step, however this is numerical trick which tends to obfuscate the theoretical function of the right-normalization.  The second normalization takes the right-normalized matrix and divides the rows by the row sums, making $\hat J_{\epsilon,\alpha}$ a row-stochastic matrix.  We will refer to this normalization as a \emph{left-normalization}.

Intuitively, the right-normalization should be understood as a de-biasing which accounts for the fact that the discrete operator will be applied to functions which are evaluated on a data set that is sampled according to the density $q$.  The parameter $\alpha$ controls the degree to which the sampling distribution is allowed to bias the operator, and a key result of \cite{diffusion} is that setting $\alpha = 1$ removes the bias entirely and recovers the Laplace-Beltrami operator independent of the sampling density $q$.  The left-normalization has a more delicate theoretical explanation.  From the discrete perspective, the left-normalization makes the matrix into a row-stochastic (or Markovian) matrix.  In the continuous limit, the effect of the left-normalization is to eliminate a complicated curvature dependent term which appears in the expansion of $J_{\epsilon}$ (see Lemma \ref{diffmaplemma} below).  We note the fascinating correspondence between the Markovian normalization from the discrete perspective, and the isolation of the generator of a reversible stochastic process from the continuous perspective.

The foundation of the data-driven manifold learning approach is the assumption that the data is given by sampling data points on a manifold.  For this approach to be practical we must require the manifold to have non-vanishing sampling density.  In this sense, the manifold of interest is by definition the set of points where the sampling density is strictly positive.  For this set to be compact requires that the density function is bounded away from zero.  Recently it was shown in \cite{BH14} that the assumption of a compact manifold could be relaxed, allowing densities that decay to zero, by using a variable bandwidth kernel, analogous to those used in kernel density estimation.  

In order to allow the sampling density to be arbitrarily close to zero, the bandwidth function must be large in areas of small sampling and small in areas of large sampling. It was shown in \cite{BH14} that the sampling density could be estimated from the data set with sufficient accuracy to form an appropriate bandwidth function assuming that the dimension of the manifold was known. While the theory developed here will apply to variable bandwidth kernels, many of the large class of kernels that will be studied in Sections \ref{localkernels} and \ref{kernelgeometry} will not satisfy the constraints required to be applicable to non-compact manifolds.  In fact, the expansions in Sections \ref{localkernels} and \ref{kernelgeometry} do generalize to non-compact manifolds, simply by assuming the operators are only applied to functions that are square integrable with respect to the sampling measure.  The difficulty comes in using a discrete data set to approximate the integral operators as Monte Carlo integrals.  For many kernels the pointwise error bounds on these Monte Carlo integrals go to infinity as the sampling density goes to zero \cite{BH14}. Since we are interested in operators which can be approximated by discrete sampling, throughout this paper we will restrict our attention to compact manifolds.

%The parameter $\alpha$ will determine how much influence the sampling density will have on the operator $\mathcal{L}$.
%In \cite{DMDC} the parameter $\alpha$ was shown to correspond to a conformal change of metric $g_{\alpha} = e^{4(1-\alpha) U/(d-2)} g$ since
% \[ \Delta_{g_{\alpha}} \varphi = e^{-4(1-\alpha) U/(d-2)} (\Delta_g \varphi - 2(1-\alpha) \nabla \varphi \cdot \nabla U) =  e^{-4(1-\alpha) U/(d-2)} \mathcal{L}\varphi.  \]
% As shown in \cite{diffusion}, by taking $\alpha = 1$ we can remove this influence entirely, and recover the Laplace-Beltrami operator $\mathcal{L} =\Delta_g$.  

\section{Generalization of diffusion maps to local kernels}\label{localkernels}

In this section we define \emph{local} kernels and show that under the geometric prior, each local kernel defines a geometry on the embedded manifold in the limit of large data.  Section \ref{mainresult} introduces the formal definition of a local kernel and develops a natural generalization of the results of diffusion maps in \cite{diffusion}.  In Section \ref{mainresultex} we give practical examples of how local kernels can be used to regularize the geometry on an embedded manifold.  For convenience and clarity we restrict our construction in this section to manifolds without boundary; we conjecture that the results could be extended to manifolds with boundary following the technique of \cite{diffusion}.

As we saw in Section \ref{dmbackground}, the standard diffusion maps construction starts with a kernel which can be written as a scalar function of the Euclidean distance, namely $J_{\epsilon}(x,y) = h(||x-y||^2/\epsilon)$.  Such a kernel is sometimes called a \emph{radial kernel}.  Our primary goal is to generalize the results of \cite{diffusion} to kernels of the form $K(\epsilon,x,y)$ that may be nonhomogeneous in $x$ and $y$, and may depend on norms other than the Euclidean norm used in the radial kernels.  Our primary assumption will be that the kernel is bounded above by a radial kernel, so that intuitively as $\epsilon\to 0$ the kernel strongly localizes the interactions between points (since $K$ is very close to zero when $x$ and $y$ are not close).  However, local kernels will not have to be homogeneous in $x$ and $y$ and will not need to decay at the same rate in all directions.  

The key property of $K$, that the kernel strongly localizes as $\epsilon\to 0$, motivates the name \emph{local kernels}.  It turns out that the decay rate does not need to be entirely independent of $\epsilon$.  If we think of $K(\epsilon,x,y)$ as defining a transition probability, a \emph{drift-free} kernel would be centered so that the maximum is at $y=x$.  In our definition, a local kernel does not have to be drift-free, so we will allow the maximum of the transition probability to be at $y = x + \sqrt{\epsilon}b(x)$.  While a local kernel does not have to be centered, the maximum must approach $y=x$ at a rate no slower than $\sqrt{\epsilon}$.  If the maximum approaches faster than $\sqrt{\epsilon}$ then the kernel will have the same limit as the associated centered kernel, but when the rate is precisely $\sqrt{\epsilon}$, the limiting operator contains a drift based on the vector field $b$.

\subsection{Local kernels and their associated Markov processes}\label{mainresult}

We now define {local kernels} and show how they generalize the radial kernels of \cite{diffusion}.  The key result will be that in the limit as $\epsilon\to 0$ the integral operator associated to a local kernel approximates the generator of a Markov process on the manifold $\cal M$, a $d$-dimensional smooth manifold.  The drift and diffusion coefficients of this Markov process depend on the moments of the local kernel computed on the tangent bundle of $\mathcal{M}$.

\begin{Def}[Local kernel] \rm A nonzero function $K : \mathbb{R} \times \mathbb{R}^n \times \mathbb{R}^n \to \mathbb{R}$ is called a \emph{local kernel} if there exists constants $c,\sigma>0$ and a vector field $b:\mathbb{R}^n\to\mathbb{R}^n$ independent of $\epsilon$ such that
\[0 \leq K(\epsilon,x,x + \sqrt{\epsilon}z) \leq ce^{-\sigma ||z-\sqrt{\epsilon}b(x)||^2} \]
for all $x,z \in \mathbb{R}^n$. 
\end{Def} 

The limiting continuous operator constructed via a local kernel is determined by the moments defined below.  Throughout this section and Section \ref{kernelgeometry} we fix a basis $\{\partial_{i} = \frac{\partial}{\partial x^i}\}_{i=1}^d$ for the tangent space $T_x\cal M$ at an arbitrary point $x\in \mathcal{M}$.  For convenience and without loss of generality we assume that the tangent space is aligned in the ambient space so that $z \in T_x\mathcal{M}\subset \mathbb{R}^n$ has coordinates $(z,0)^\top \in \mathbb{R}^n$.  Notice that local kernels include any function where $K(\epsilon,x,x+\sqrt{\epsilon}z)$ has compact support in $z$ for all $x$.  For example, $K(\epsilon,x,y) = \textup{max}\{1-||x-y||^2/\epsilon,0\}$ is a local kernel.

\begin{Def}[Moments of a local kernel] \rm For a local kernel $K$ define the zeroth, first, and second moment functions
\begin{align}\label{kerneldefs} m(x) &\equiv \lim_{\epsilon\to 0} \int_{T_x\mathcal{M}}K(\epsilon,x,x+\sqrt{\epsilon}\hat z)\,dz \nonumber \\
\mu_i(x)  &\equiv \lim_{\epsilon\to 0} \frac{1}{\sqrt{\epsilon}}\int_{T_x\mathcal{M}} z_i K(\epsilon,x,x+\sqrt{\epsilon}\hat z)\, dz \nonumber \\
C_{ij}(x) &\equiv \lim_{\epsilon\to 0} \int_{T_x\mathcal{M}} z_i z_j K(\epsilon,x,x+\sqrt{\epsilon}\hat z)\, dz
\end{align}
respectively, where $\hat z \in \mathbb{R}^n$ is equal to $z$ on $T_x\mathcal{M} \subset \mathbb{R}^n$ and zero in all orthogonal directions. 
\end{Def} 

Note that $\mu(x)$ is a $d$-dimensional vector-valued function on $\mathcal{M}$ and $C(x)$ is a $d\times d$ matrix-valued function on $\mathcal{M}$ based on the coordinates $dx_i$.  While we work in the basis $\{ \partial_i\}$, the vector $\mu$ and matrix $C$ transform appropriately as tensors so we will sometimes neglect the indices.  While the definition of the moments may seem impractical due to the need to integrate over each tangent space, we will see examples where these definitions simplify (such as the isotropic kernels defined below) and other examples where they are natural for data-driven algorithms.  For example, if we define a norm $||\cdot ||_{C(x)}$ where $C(x)$ is the correlation matrix based on the nearest neighbors of $x$ in the ambient space, in the limit of large data the correlation matrix will be rank $d$ and will only be a norm on the tangent space $T_x\cal M$.  

As we will see below, the standard radial kernel $J_{\epsilon}$ has $\mu=0$, so we introduce the following definition for this special class of kernels,
\begin{Def}[Drift-free kernel] \rm A local kernel is called a \emph{drift-free kernel} if the first moment $\mu$ is identically zero.
\end{Def}
The second property of $J$ in the diffusion maps construction is that the kernel is isotropic.
\begin{Def}[Isotropic local kernel] \rm A local kernel is called \emph{isotropic} if the second moment is a multiple of an orthogonal transformation.  Namely for some scalar function $\rho:\mathcal{M}\to\mathbb{R}$, the second moment matrix $C(x)$ satisfies $C(x)^TC(x) = \rho(x)\textup{Id}_{d\times d}$.
\end{Def}
Finally, the kernel $J$ is also homogeneous in the following sense.
\begin{Def}[Homogeneous local kernel] \rm A local kernel is called \emph{homogeneous} with respect to a moment if the moment is independent of $x$.
\end{Def}

We now show that any radial kernel is a local kernel which is drift-free, isotropic, and homogeneous in all moments.
\begin{prop} Assume a kernel $J$ can be written in the form $J(\epsilon,x,y) = h(||x-y||^2/\epsilon)$ where $|h(u)|< ce^{-u/\sigma}$ for some $c,\sigma$. Then $J$ is a local kernel which is drift-free, isotropic and homogeneous in all moments.
\end{prop}
\begin{proof}
Since $h$ has fast decay $J$ is a local kernel.  Note that $J(\epsilon,x,x+\sqrt{\epsilon}\hat z) = h(||\hat z||^2) = h(||z||^2)$, therefore $\mu=0$ and
\[ C_{ij}(x) = \int_{T_x\mathcal{M}} z_i z_j h(||z||^2) \, dz = \delta_{ij} \int_{T_x\mathcal{M}} z_1^2 h(||z||^2) \, dz, \]
where the integral vanishes when $i\neq j$ since the integrand is odd.  Thus for $\rho(x) = \rho_0 = \int_{T_x\mathcal{M}} z_1^2 h(||z||^2)\, dz$ we have $C(x) = \rho_0 \textup{Id}_{d\times d}$, implying that $J$ is isotropic and homogeneous.
\end{proof}

While $J$ is homogeneous and isotropic, the right-normalized diffusion maps kernel $J_{\epsilon,\alpha}$ has a very special type of non-homogenous anisotropy that is determined by the $\alpha$ parameter.  As noted in Section \ref{dmbackground}, this anisotropy allows the diffusion maps construction to access different geometries which are conformally equivalent to the geometry induced by the ambient space.  However, this normalization is best understood as accounting for the sampling measure and we will return to this normalization in Section \ref{nonuniform}.  Our goal is to allow any type of non-homogeneous and anisotropic kernel and find the operators which can be approximated in the limit of $\epsilon \to 0$ using local kernels.  The following example is the prototype of a local kernel which can be used to define a geometry.

\begin{examp}[Prototypical local kernels]\label{prototype} \rm Let $A(x)$ be a matrix valued function on the manifold $\mathcal{M}$ such that each $A(x)$ is a symmetric positive definite $n\times n$ matrix and let $b(x)$ be a vector valued function.  Define the prototypical kernel with covariance $A$ and drift $b$ by
\[ K(\epsilon,x,y) = \exp\left(-\frac{(x-y-\epsilon b(x))^T A(x)^{-1}(x-y-\epsilon b(x))}{2\epsilon}\right). \]
We note that $K$ can be rewritten as
\[ K(\epsilon,x,y) =  \exp\left(-\frac{(x-y)^T A(x)^{-1}(x-y)}{2\epsilon} + (x-y)^\top A(x)^{-1}b(x) - \frac{\epsilon}{2} b(x)^{\top}A(x)^{-1}b(x) \right), \]
and that if we omit the term $\epsilon b^\top A^{-1}b$, the moments will not be affected because this term is higher order.  To define the moments we need to restrict the $n\times n$ matrix $A$ to the tangent space $T_x\cal M$, thus we define $\mathcal{I} = \mathcal{I}(x) : \mathbb{R}^n \to T_x\cal M$ to be the restriction of the ambient space to the tangent space (written in the basis $\{\partial_i\}$) so that $\mathcal{I}(x)$ is a $d\times n$ matrix.  The lower moments of the prototypical kernel are, $m(x) = (2\pi)^{d/2}\textup{det}(\mathcal{I}(x)A(x)\mathcal{I}(x)^\top)^{1/2}$, $\mu(x) = m(x)\mathcal{I}(x)b(x)$, and $C(x) =m(x)\mathcal{I}(x)A(x)\mathcal{I}(x)^\top$.  
\end{examp}

Notice that a prototypical kernel is simply an unnormalized multivariate Gaussian in the ambient space.  While a normalized Gaussian would have some advantages which we will remark on below, the normalization factor $m(x)$ is very difficult to determine.  This is because finding $m(x)$ requires computing the determinant of $A(x)$ restricted to each tangent space $T_x\cal M$, and since we are trying to learn the structure of the manifold from the data we do not want to assume that $\cal I$ is known.  Rather than explicitly estimating $m(x)$ in the construction of the kernel, we will instead show that a normalization trick, motivated by the left-normalization first introduced in \cite{diffusion}, allows us to eliminate the influence of $m(x)$.  In fact, we will see that this approach uses the kernel to determine an estimate of $m(x)$, and normalizing by this factor simultaneously removes the influence of $m(x)$ as well as another unwanted term which is higher order.

In order to understand the limiting behavior of local kernels, we first need to generalize the following lemma from \cite{diffusion} which allows the approximation of the integral operator $G_{\epsilon}$ for radial kernels.

\begin{lem}[Expansion of radial kernels, Coifman and Lafon \cite{diffusion}]\label{diffmaplemma} Let $f$ be a smooth real-valued function on an embedded $d$-dimensional manifold $\mathcal{M} \subset \mathbb{R}^n$ and let $h:\mathbb{R}\to\mathbb{R}$ have fast decay, meaning that there exist constants $c,\sigma$ such that $h(a)\leq ce^{-a/\sigma}$ for all $a$. Then
\[ G_{\epsilon}f(x) \equiv \epsilon^{-d/2}\int_{\mathcal{M}} h\left(\frac{||x-y||^2}{\epsilon}\right)f(y)\ dy = m_0 f(x) + \epsilon \frac{m_2}{2}(\omega(x)f(x) + \Delta f(x)) + \mathcal{O}(\epsilon^2) \]
where $m_0 = \int_{\mathbb{R}^d} h(||x||^2)\ dx$ and $m_2  = \int_{\mathbb{R}^d} x_1^2 h(||x||^2)\ dx$ are constants determined by $h$, and $\omega(x)$ depends on the induced geometry of $\mathcal{M}$.  The operator $\Delta$ is the Laplacian operator for $\mathcal{M}$ with the metric induced from the ambient space. 
\end{lem}

The next lemma generalizes this result to local kernels.  We introduce the standard notation $\textup{div}$ and $\nabla$ to refer to the intrinsic divergence and gradient operators on the embedded manifold such that $\Delta = \textup{div}\circ \nabla$ is the (negative definite) Laplacian for $\mathcal{M}$ with the induced metric.  Consider a stochastic process on $\cal M$ with drift $\mu$ and diffusion matrix $\sqrt{C}$ written in It\^o form as
\begin{align}\label{SDE} dx = \mu(x)dt + \sqrt{C(x)} dW_t, \end{align}
where $W_t$ is d-dimensional Brownian motion on $\cal M$.  The generator $\cal L$ for \eqref{SDE}, also known as the backward Kolmogorov operator, and its adjoint $\cal L^*$, the Fokker-Planck operator, are given by
\begin{equation} \label{eqScriptL}
 \mathcal{L}f = \mu \cdot \nabla f + \frac{1}{2}C_{ij}\nabla_i \nabla_j f  \hspace{40pt} \mathcal{L}^*f = -\textup{div}(\mu f) + \frac{1}{2} \nabla_j \nabla_i(C_{ij} f), 
 \end{equation}
where $\nabla_i$ is the covariant derivative in the $i$-th direction. The Hessian matrix $\nabla_i\nabla_j f$ and the dot product of the vector fields $\mu$ and $\nabla f$ (where $\nabla$ without subscripts refers to the gradient operator) are taken with respect to the Riemannian metric on $\mathcal{M}$ inherited from the ambient space.  

Later we will be applying local kernels to analyze data sets.  We will not assume that the data are sampled from the system \eqref{SDE}. In fact, there is no requirement that the data are generated by a dynamical system at all.  The system \eqref{SDE} is a Markov process which is implicit to the local kernel construction in the sense that any local kernel with moments $\mu$ and $C$ can be used to construct the operators $\mathcal{L}$ and $\mathcal{L}^*$ that correspond to \eqref{SDE}.  If the data set were generated by the system \eqref{SDE} and the moments $\mu$ and $C$ could be estimated from the data set, then a local kernel could be constructed with these moments to approximate the generator of the data set.  Such an approach was developed recently in \cite{probabilisticEstimates} where the moments are estimated from the data assuming a slow evolution on the manifold.  One application of the theory of local kernels is to show that a large class of kernels can be used to construct the desired operator instead of the standard exponential kernel for which the theory was developed in \cite{Singer2008}. This generalization also applies to related work such as \cite{Kushnir2012,talmon2012,Singer22092009}. The real power of the local kernel construction is the ability to construct the operators $\mathcal{L}$ and $\mathcal{L}^*$ for any system of the form \eqref{SDE} on the manifold defined by the data regardless of how the data is generated, by choosing an appropriate local kernel.

The following lemma connects the asymptotic expansion of the integral operator associated to a local kernel with the generator $\mathcal{L}$.

\begin{lem}[Expansion of local kernels] \label{mainlemma}
Let $f$ be a smooth real-valued function on an embedded $d$-dimensional manifold $\mathcal{M} \subset \mathbb{R}^n$ and let $K(\epsilon,x,y)$ be a local kernel.  Let $m$ denote the zeroth moment of $K$ from (\ref{kerneldefs}), and let $\mathcal{L}$ be defined using the first and second moments of $K$ as in (\ref{eqScriptL}). Then the expansion
\begin{align} G_{\epsilon}f(x) &\equiv \epsilon^{-d/2}\int_{\mathcal{M}} K(\epsilon,x,y)f(y)\ dy \nonumber \\ 
&= m(x) f(x) + \epsilon \left(\omega(x)f(x) + \mathcal{L}f(x) \right) + \Omega(x)\epsilon^{3/2}+ \mathcal{O}(\epsilon^2)
\end{align}
holds, where $\omega(x)$ and $\Omega(x)$ depend on the kernel and the induced metric $g$.
\end{lem}
\begin{proof} 
Let $x\in \mathcal{M}$, for $0<\gamma<1/2$ and $\epsilon$ sufficiently small, the neighborhood $N_{\epsilon^{\gamma}}(x)$ of radius $\epsilon^{\gamma}$ about $x$ is diffeomorphic to a neighborhood of zero in the tangent space $T_x\mathcal{M}$.  Thus, for any $y \in N_{\epsilon^{\gamma}}(x)$ we can write $y-x = (u,g(u))$ where $u\in T_x\mathcal{M}$ is the orthogonal projection of $y-x$ into $T_x\mathcal{M}$.  Note that setting $u=0$ we have $0=(0,g(0))$ and so $g(0)=0$, and moreover $Dg(0)=0$ since $g$ is tangent to $\mathcal{M}$ at $u=0$.  Thus we have the Taylor expansion $g(u) = p_{x,2}(u) + p_{x,3}(u) + \mathcal{O}(||u||^4)$.  Since $K$ is a local kernel, we can expand the kernel about $\hat u = (u,0)$ as 
\begin{align}\label{kernelexpansion} K(\epsilon,x,y) &= K(\epsilon,x,x+\hat u + (0,g(u)))  \nonumber \\
&= K(\epsilon,x,x+\hat u) + D_yK(\epsilon,x,x+\hat u)^\top(0,g(u))^\top + \left|H_sK(\epsilon,x,x+\hat u)\right|\mathcal{O}(||g(u)||^2) \nonumber \\
&= K(\epsilon,x,x+\hat u) + D_yK(\epsilon,x,x+\hat u)^\top(0,p_{x,2}(u) + p_{x,3}(u))^\top + \left|H_sK(\epsilon,x,x+\hat u)\right|\mathcal{O}(||u||^4) \nonumber \\
&= K(\epsilon,x,x+\hat u) + (\Pi_{u^{\perp}}D_yK(\epsilon,x,x+\hat u))^\top (p_{x,2}(u) + p_{x,3}(u)) + \left|H_sK(\epsilon,x,x+\hat u)\right|\mathcal{O}(||u||^4). 
\end{align}
Following \cite{diffusion} we can expand $f(y) = f(\exp_x(s)) = \tilde f(s)$ when $y\in N_{\epsilon^{\gamma}}(x)$ as
\begin{align}\label{functionexpansion} f(y) = \tilde f(0) +  u^\top D_s \tilde f(0) + \frac{1}{2} u^\top H_s \tilde f(0)  u + p_{x,3}(u) + \mathcal{O}(||u||^4).\end{align} 

Combining \eqref{kernelexpansion} and \eqref{functionexpansion} we have the following expansion for the product:
\begin{align}\label{prodexp} K(\epsilon,x,y)f(y) &= \tilde f(0)\left(K(\epsilon,x,x+\hat u) + \Pi_{u^{\perp}}D_yK(\epsilon,x,x+\hat u)^\top p_{x,2,3}(u) \right) \nonumber \\
&\hspace{10pt}+ K(\epsilon,x,x+\hat u)\left[ u^\top D_s\tilde f(0) + \frac{1}{2} u^\top H_s\tilde f(0)  u  + p_{x,3}(u) \right] \nonumber\\
&\hspace{10pt}+ \left(K(\epsilon,x,x+\hat u) + \left|H_sK(\epsilon,x,x+\hat u)\right|\right)\mathcal{O}(||u||^4)
\end{align}
where all homogeneous polynomials of degree 2 and 3 in the variable $u$ are combined into the single term $p_{x,2,3}(u)$.  We want to use this expansion inside the integral operator $G_{\epsilon}f(x) = \epsilon^{-d/2}\int_{\mathcal{M}} K(\epsilon,x,y)f(y)\ dy$, so we localize this integral to $y\in N_{\epsilon^{\gamma}}(x)$.  The residual integral is therefore
\[ \left|\epsilon^{-d/2}\int_{||y-x||>\epsilon^{\gamma}}K(\epsilon,x,y)f(y)\ dy\right| = \left|\int_{||\tilde y-x||>\epsilon^{\gamma-1/2}}K(\epsilon,x,\sqrt{\epsilon}(\tilde y-x)+x)f(\sqrt{\epsilon}(\tilde y-x)+x) d\tilde y \right| \leq ||f||_{\infty}\mathcal{O}(\epsilon^{2}), \]
where we have changed variables to $y = \sqrt{\epsilon}(\tilde y-x)+x$ and used the exponential decay in the tails of $K(\epsilon,x,\sqrt{\epsilon}(\tilde y-x)+x)$.  Note that $\gamma<1/2$ meaning that $\epsilon^{\gamma-1/2}\to\infty$ as $\epsilon\to 0$ and therefore the integral is only over the tail of the kernel.  In fact, the integral of the exponential tail shrinks faster than any polynomial in $\epsilon$, so in particular it is less than $\mathcal{O}(\epsilon^2)$.

Thus we have the following expansion for the integral operator:
\begin{align} G_{\epsilon}f(x) &= \epsilon^{-d/2}\int_{\mathcal{M}} K(\epsilon,x,y)f(y) \ dy = \epsilon^{-d/2}\int_{||y-x||<\epsilon^{\gamma}} K(\epsilon,x,y)f(y) \ dy \nonumber \\
&= \epsilon^{-d/2}\int_{||u||<\epsilon^{\gamma}} \tilde f(0)\left(K(\epsilon,x,x+\hat u) + \Pi_{u^{\perp}}D_yK(\epsilon,x,x+\hat u)^\top p_{x,2,3}(u) \right) (1+p_{x,2,3}(u) + \mathcal{O}(\epsilon^2))du \nonumber \\
&\hspace{10pt}+\epsilon^{-d/2}\int_{||u||<\epsilon^{\gamma}} K(\epsilon,x,x+\hat u)\left[ u^\top D_s\tilde f(0) + \frac{1}{2} u^\top H_s\tilde f(0)  u  + p_{x,3}(u) \right] (1+p_{x,2,3}(u) + \mathcal{O}(\epsilon^2))du \nonumber \\
&\hspace{10pt}+\epsilon^{-d/2}\int_{||u||<\epsilon^{\gamma}} \left(K(\epsilon,x,x+\hat u) + \left|H_sK(\epsilon,x,x+\hat u)\right|\right)\mathcal{O}(||u||^4)(1+p_{x,2,3}(u) + \mathcal{O}(\epsilon^2))du \nonumber \\
&= \int_{||z||<\epsilon^{\gamma-1/2}} \tilde f(0)\left(K(\epsilon,x,x+\sqrt{\epsilon}\hat z) + \Pi_{u^{\perp}}D_yK(\epsilon,x,x+\sqrt{\epsilon}\hat z)^\top (\epsilon p_{x,2}(z)+\epsilon^{3/2}p_{x,3}(z)) \right) (1+\epsilon p_{x,2}(z)+\epsilon^{3/2}p_{x,3}(z) + \mathcal{O}(\epsilon^2))dz \nonumber \\
&\hspace{10pt}+\int_{||u||<\epsilon^{\gamma}} K(\epsilon,x,x+\sqrt{\epsilon}z)\left[\sqrt{\epsilon} z^\top D_s\tilde f(0) + \frac{\epsilon}{2} z^\top H_s\tilde f(0)z  +\epsilon^{3/2} p_{x,3}(z) \right] (1+\epsilon p_{x,2}(z)+\epsilon^{3/2} p_{x,3}(z) + \mathcal{O}(\epsilon^2))dz \nonumber \\
&\hspace{10pt}+\epsilon^2 \int_{||z||<\epsilon^{\gamma-1/2}} \left(K(\epsilon,x,x+\sqrt{\epsilon}\hat z) + \left|H_sK(\epsilon,x,\sqrt{\epsilon}\hat z)\right|\right)\mathcal{O}(||z||^4)(1+\epsilon p_{x,2}(z)+\epsilon^{3/2}p_{x,3}(z) + \mathcal{O}(\epsilon^2))dz \nonumber
\end{align}
where we use the fact \cite{diffusion} that $\textup{det}\left(\frac{dy}{du}\right) = 1 + p_{x,2,3}(u) + \mathcal{O}(\epsilon^2)$ to change variables from $y$ to $u$; and then we change to $z = \epsilon^{-1/2}u$ so that $\textup{det}\left(\frac{du}{dz}\right) = \epsilon^{d/2}$ and we set $\hat z = (0,z)^\top$.  We now use the exponential decay of the kernel and its first two derivatives to extend the integrals to the entire tangent space.  Note that any polynomial integrated against the kernels will be a constant, yielding
\begin{align} G_{\epsilon}f(x) &= \int_{T_x\mathcal{M}} \tilde f(0)\left(K(\epsilon,x,x+\sqrt{\epsilon}\hat z)(1+\epsilon p_{x,2}(z)+\epsilon^{3/2}p_{x,3}(z)) + \Pi_{u^{\perp}}D_yK(\epsilon,x,x+\sqrt{\epsilon}\hat z)^\top (\epsilon p_{x,2}(z)+\epsilon^{3/2}p_{x,3}(z)) \right)dz \nonumber \\
&\hspace{10pt}+\int_{T_x\mathcal{M}} K(\epsilon,x,x+\sqrt{\epsilon}\hat z)\left[\sqrt{\epsilon} z^\top D_s\tilde f(0) + \frac{\epsilon}{2} z^\top H_s\tilde f(0)z  +\epsilon^{3/2} p_{x,3}(z) \right] (1+\epsilon p_{x,2}(z) )dz + \mathcal{O}(\epsilon^2) \nonumber \\
&= f(x) \int_{T_x\mathcal{M}}K(\epsilon,x,x+\sqrt{\epsilon}\hat z)dz  +\sqrt{\epsilon} \left(\sum_{i=1}^d \frac{\partial \tilde f}{\partial s_i}(0) \int_{T_x\mathcal{M}} z_i K(\epsilon,x,x+\sqrt{\epsilon}\hat z) dz \right) \nonumber \\
&\hspace{10pt}+\epsilon \left(\sum_{i,j=1}^d \frac{\partial^2 \tilde f}{\partial s_j \partial s_i}(0)\int_{T_x\mathcal{M}} \frac{1}{2}z_i z_j K(\epsilon,x,x+\sqrt{\epsilon}\hat z) dz  +  f(x) \int_{T_x\mathcal{M}} K(\epsilon,x,x+\sqrt{\epsilon}\hat z) p_{x,2}(z) +\Pi_{u^{\perp}}D_yK(\epsilon,x,x+\sqrt{\epsilon}\hat z)^\top p_{x,2}(z) dz \right)\nonumber \\
&\hspace{10pt}+\epsilon^{3/2} \left( \int_{T_x\mathcal{M}} K(\epsilon,x,x+\sqrt{\epsilon}\hat z)p_{x,3}(z)+\Pi_{u^{\perp}}D_yK(\epsilon,x,x+\sqrt{\epsilon}\hat z)^\top p_{x,3}(z) dz \right)+ \mathcal{O}(\epsilon^2) \nonumber
\end{align}
We now define the terms
\begin{align}\label{kerneldefs2} 
\omega(x) &\equiv \lim_{\epsilon\to 0} \int_{T_x\mathcal{M}} K(\epsilon,x,x+\sqrt{\epsilon}\hat z) p_{x,2}(z) +\Pi_{u^{\perp}}D_yK(\epsilon,x,x+\sqrt{\epsilon}\hat z)^\top p_{x,2}(z)\, dz, \nonumber \\
\Omega(x) &\equiv \lim_{\epsilon\to 0} \int_{T_x\mathcal{M}} K(\epsilon,x,x+\sqrt{\epsilon}\hat z)p_{x,3}(z)+\Pi_{u^{\perp}}D_yK(\epsilon,x,x+\sqrt{\epsilon}\hat z)^\top p_{x,3}(z) \,dz.
\end{align}
Combining the definitions of \eqref{kerneldefs} and \eqref{kerneldefs2} with the expansion of $G_{\epsilon}$ yields
\begin{align} \int_{\mathcal{M}} K(\epsilon,x,y)f(y) \,dy = m(x) f(x) + \epsilon \left(\omega(x)f(x) +\sum_i \mu_i(x) \frac{\partial \tilde f}{\partial s_i}(0) + \frac{1}{2}\sum_{ij} C_{ij}(x) \frac{\partial^2 \tilde f}{\partial s_i \partial s_j}(0) \right) + \Omega(x)\epsilon^{3/2}+ \mathcal{O}(\epsilon^2).
\end{align}
Note that writing $f$ in geodesic coordinates based at the point $x$, the gradient operator at $x$ becomes $\nabla f(x) = g^{jl}\frac{\partial \tilde f}{\partial s_l}(0) dx_j$ so that the inner product becomes 
\[ \mu \cdot \nabla f = \sum_{ij}g_{ij}\mu_i (\nabla f)_j = \sum_{ij} g_{ij}\mu_i g^{jl}\frac{\partial \tilde f}{\partial s_l}(0) = \sum_i \mu_i\frac{\partial \tilde f}{\partial s_i}(0), \]
since $\sum_j g_{ij}g^{jl} = \delta_{il}$.  Since $\mathcal{L}$ is written in local coordinates as $\mathcal{L}f(x) = \sum_i \mu_i(x) \frac{\partial \tilde f}{\partial s_i}(0) + \frac{1}{2} \sum_{ij} C_{ij}(x) \frac{\partial^2 \tilde f}{\partial s_i \partial s_j}(0)$, we have shown
\[ \epsilon^{-d/2}\int_{\mathcal{M}} K(\epsilon,x,y)f(y) \,dy = m(x) f(x) + \epsilon \left(\omega(x)f(x) + \mathcal{L}f(x) \right) + \Omega(x)\epsilon^{3/2}+ \mathcal{O}(\epsilon^2), \]
as desired.  Notice that neglecting the $\Omega$ term, the expansion is of order $\epsilon^{3/2}$. However, if the kernel and its derivative have zero skewness then $\Omega=0$ and the expansion is of order $\epsilon^2$.  
\end{proof}

Notice that the polynomials in the definition of $\Omega$ in \eqref{kerneldefs2} involve $f$ and mixed third derivatives of $f$, so in general these terms will be difficult to cancel with any type of normalization.  We therefore introduce the following definition.
\begin{Def}[Skew-free local kernel] \rm A local kernel is called \emph{skew-free} if for any homogeneous polynomial of order-3 in the variable $z$ (with coefficients depending on $x$), we have $\lim_{\epsilon\to 0} \int_{T_x\mathcal{M}} p_{x,3}(z)K(\epsilon,x,x+\sqrt{\epsilon}z)dz = 0$ and $\lim_{\epsilon\to 0} \int_{T_x\mathcal{M}} p_{x,3}(z)D_yK(\epsilon,x,x+\sqrt{\epsilon}z)dz = 0$.
\end{Def}
For the remainder of the paper we will restrict our attention to skew-free local kernels so that $\Omega=0$ and the expansion in Lemma \ref{mainlemma} is order $\epsilon^2$.  The results which follow will still apply for local kernels which are not skew-free, however the expansions will only be valid up to order $\epsilon^{3/2}$ rather than order $\epsilon^2$.  Notice that any operator which can be recovered with a local kernel can be recovered with a prototypical local kernel, which is skew-free.  Thus, in the limit of large data, there is no reason to use a local kernel which is not skew-free.  

From Lemma \ref{mainlemma} we can easily derive the expansion for the adjoint of the kernel which we define by $K^*(\epsilon,x,y) = K(\epsilon,y,x)$ with associated operator $G^*_{\epsilon}f(x) = \epsilon^{-d/2} \int_{\mathcal{M}} K(\epsilon,y,x)f(y)dy$.

\begin{lem}[Expansion of adjoint of skew-free local kernel]\label{transposelemma} Let $K$ be a skew-free local kernel. Under the same assumptions as Lemma \ref{mainlemma},
\begin{align}\label{transposeexpansion} G^*_{\epsilon}f(x) &\equiv \epsilon^{-d/2}\int_{\mathcal{M}} K(\epsilon,y,x)f(y)dy = m(x) f(x) + \epsilon \left(\omega(x)f(x) + \mathcal{L}^*f(x) \right) + \mathcal{O}(\epsilon^2).
\end{align}

\end{lem}
\begin{proof} We describe the operator $G^*_{\epsilon}f(x)$ in the weak formulation by letting $h$ be an arbitrary smooth test function so that
\[ \left<h,G^*_{\epsilon}f\right>_{L^2(\mathcal{M})} = \int_{\mathcal{M}}\int_{\mathcal{M}} h(x)K(\epsilon,y,x)f(y)\ dy dx. \]
We will expand this inner product by changing the order of integration, and noting that $\int_{\mathcal{M}} K(\epsilon,y,x)h(x)\ dx = G_{\epsilon}h(y)$,
\begin{align} \left<h,G^*_{\epsilon}f\right>_{L^2(\mathcal{M})} &= \int_{\mathcal{M}} f(y)G_{\epsilon}h(y) dy \nonumber \\
&= \int_{\mathcal{M}} f(y)\left(m(y)h(y) + \epsilon(\omega(y)h(y) + \mathcal{L}h(y))\right)dy + \mathcal{O}(\epsilon^2) \nonumber \\
&= \int_{\mathcal{M}} m(y)h(y)f(y) + \epsilon(\omega(y)h(y)f(y) + f(y)\mathcal{L}h(y)) dy + \mathcal{O}(\epsilon^2) \nonumber \\
&= \left<h,f + \epsilon(\omega f + \mathcal{L}^*f)\right> + \mathcal{O}(\epsilon^2),
\end{align}
where we have used the fact that $\left<f,\mathcal{L}h\right> = \left<\mathcal{L}^*f,h\right>$ in order to factor out $g(y)$ from each term in the last equality.  The above computation shows that in the weak sense we have $G^*_{\epsilon}f = f + \epsilon(\omega f + \mathcal{L}^*f) + \mathcal{O}(\epsilon^2)$.
\end{proof}

Since we are typically interested in the operator $\mathcal{L}$, we note that $\mathcal{L}1 = 0$, which means that if we apply the kernel operator to the constant function, we find $(G_{\epsilon}1)(x) = m(x) + \epsilon \omega(x) + \mathcal{O}(\epsilon^2)$.  So $G_{\epsilon}1$ isolates all the unwanted terms in the expansion of $G_{\epsilon}$, including the aforementioned zeroth moment $m(x)$, which is now estimated by the kernel operator $G_{\epsilon}1$, so that it does not need to be known in order to define the kernel.  The following theorem normalizes the operator by dividing by $G_{\epsilon}1$ in order to isolate $\mathcal{L}$.

\begin{thm}\label{localkerneltheorem}
Let $K(\epsilon,x,y)$ be a local kernel and set
\begin{align}\label{approxL} L_{\epsilon}f = \frac{(G_{\epsilon}1)^{-1}G_\epsilon f - \textup{Id}(f)}{\epsilon}, \hspace{50pt} L^*_{\epsilon}f = \frac{(G_{\epsilon}1)^{-1}G^*_\epsilon f - \textup{Id}(f)}{\epsilon}. \end{align}
Then ${\displaystyle \lim_{\epsilon \to 0} L_{\epsilon} = \frac{1}{m} \mathcal{L}}$ and ${\displaystyle \lim_{\epsilon \to 0} L^*_{\epsilon} = \frac{1}{m} \mathcal{L}^*}$, where $\mathcal{L}$ and $\mathcal{L}^*$ are defined in (\ref{eqScriptL}).
\end{thm}

It is crucial in Theorem \ref{localkerneltheorem} that both the kernel and the adjoint are normalized by $G_{\epsilon}1$.  This is because $-\textup{div}(\mu f) = -f\textup{div}(\mu) - \mu \cdot \nabla f$ implies that $G_{\epsilon}^*1(x) = 1+\epsilon(\omega(x) - \textup{div}(\mu))$. Dividing by this term would introduce an unwanted term to the operator.  

The normalization \eqref{approxL} was first introduced in \cite{diffusion}, and it has the significant advantage that the zeroth moment $m(x)$ of the kernel does not need to be known when the kernel is defined.  For the prototypical kernel in Definition \ref{prototype}, for example, since we did not normalize the Gaussian, we have $\mu(x) = m(x)b(x)$ and $C(x) = m(x)A(x)$ and therefore
\[ \frac{1}{m}\mathcal{L}f = \frac{1}{m}\left( m b\cdot \nabla f + m A_{ij}\nabla_i \nabla_j f \right) =  b\cdot \nabla f + A_{ij}\nabla_i \nabla_j f. \]
Since the norm in the prototypical kernel is defined in terms of $A$ and $b$, we typically do not wish the normalization factor (which is difficult to estimate before the kernel is defined, since the determinant must be computed on the tangent space) to affect the operator.  The normalization \eqref{approxL} lets us avoid the normalization factor altogether.  However, for the prototypical kernel the formula for $L_{\epsilon}^*$ becomes more complicated.  In this case it is more natural to define
\[ \hat L_{\epsilon}^* = \frac{(G_{\epsilon}^*((G_{\epsilon}1)^{-1}f) - \textup{Id}}{\epsilon}, \]
which is equivalent to first normalizing the kernel matrix and then computing the transpose.  It is easy to verify that for a prototypical kernel we have
\[ \hat L_{\epsilon}^*f = -\textup{div}(b f) +\nabla_i \nabla_j ( A_{ij}f ) + \mathcal{O}(\epsilon), \]
and we will use this fact in Section \ref{markovexample}.

We note that another normalization option is to subtract the unwanted terms so that
\begin{align}\label{subtractionKernel} \frac{1}{\epsilon}\left(G_{\epsilon}f - f G_{\epsilon}1\right) = \mathcal{L}f + \mathcal{O}(\epsilon). \end{align}
The normalization \eqref{subtractionKernel} was used in \cite{BN} and has the advantage of discretizing as a weighted graph Laplacian, which is an unbiased estimator of the continuous operator. When the kernel is symmetric the estimator will also be a symmetric matrix.  However, this approach would require the kernel to be normalized by dividing the kernel by the zeroth moment $m(x)$.  In most applications this is not known a priori and the normalization by the empirical estimate $G_{\epsilon}1$ as in \eqref{approxL} is a more practical approach.

\subsection{Numerical example}\label{markovexample}

Due to the complexity of the previous derivations, and their importance in the subsequent sections, we will give a numerical example demonstrating and validating the theory developed so far.  Consider a flat torus isometrically embedded in $\mathbb{R}^4$.  This example allows easy computation of the covariant derivatives and adherence to the uniform sampling assumptions. Start with a uniform grid of $10000$ points $(\theta_i,\phi_i)$ in the flat torus $[0,2\pi]^2$, and map these points into $\mathbb{R}^4$ via the isometric embedding $x_i = (\sin(\theta_i),\cos(\theta_i),\sin(\phi_i),\cos(\phi_i))^\top$.  In order to verify the theory for a non-homogeneous, anisotropic kernel, we will design a prototypical local kernel with the moments
\[ \mu(\theta,\phi) = (2+\sin(\theta),0)^\top \hspace{50pt} C(\theta,\phi) = \left[\begin{array}{cc} 3+\sin(\phi) & 1 \\ 1 & 1 \end{array} \right]. \]
To build a kernel on the embedded torus, we must lift these two dimensional tensors into $\mathbb{R}^4$. Let $D\iota$ be the matrix with rows given by the tangent vectors
\[ D\iota(\theta,\phi) = \left[\begin{array}{cccc} \cos(\theta) & -\sin(\theta) & 0 & 0 \\ 0 & 0 & \cos(\phi) & -\sin(\phi) \end{array} \right]. \]
Abbreviating $D\iota_i = D\iota(\theta_i,\phi_i)$, $\mu_i = \mu(\theta_i,\phi_i)$, and $C_i = C(\theta_i,\phi_i)$, we can define the prototypical local kernel
\[ K(\epsilon,x_i,x_j) = \exp\left(-\frac{(x_j-x_i-\epsilon D\iota_i \mu_i)^\top D\iota_i C_i D\iota_i^\top (x_j-x_i-\epsilon D\iota_i \mu_i) }{2\epsilon} \right). \]
While this construction is quite artificial, it is only for the purposes of numerical verification.  Indeed, as we will show in Section \ref{mainresultex}, the real power of local kernels is the ability to build a data-driven kernel where the moments are naturally constructed from the data itself.  

 \begin{figure}[h]
  \begin{center}
\subfigure[]{  \includegraphics[height=.3\linewidth]{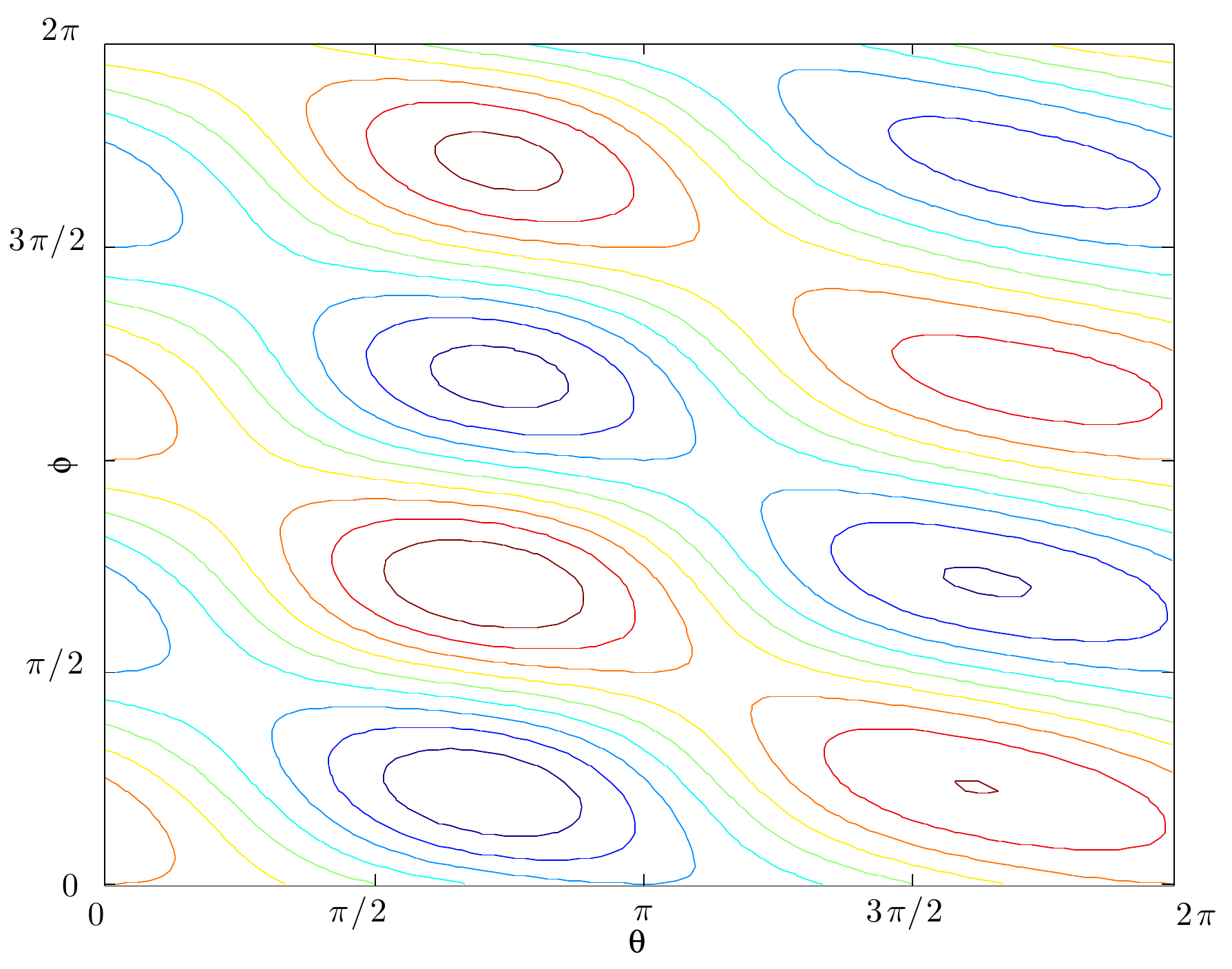}}
\subfigure[]{    \includegraphics[height=.3\linewidth]{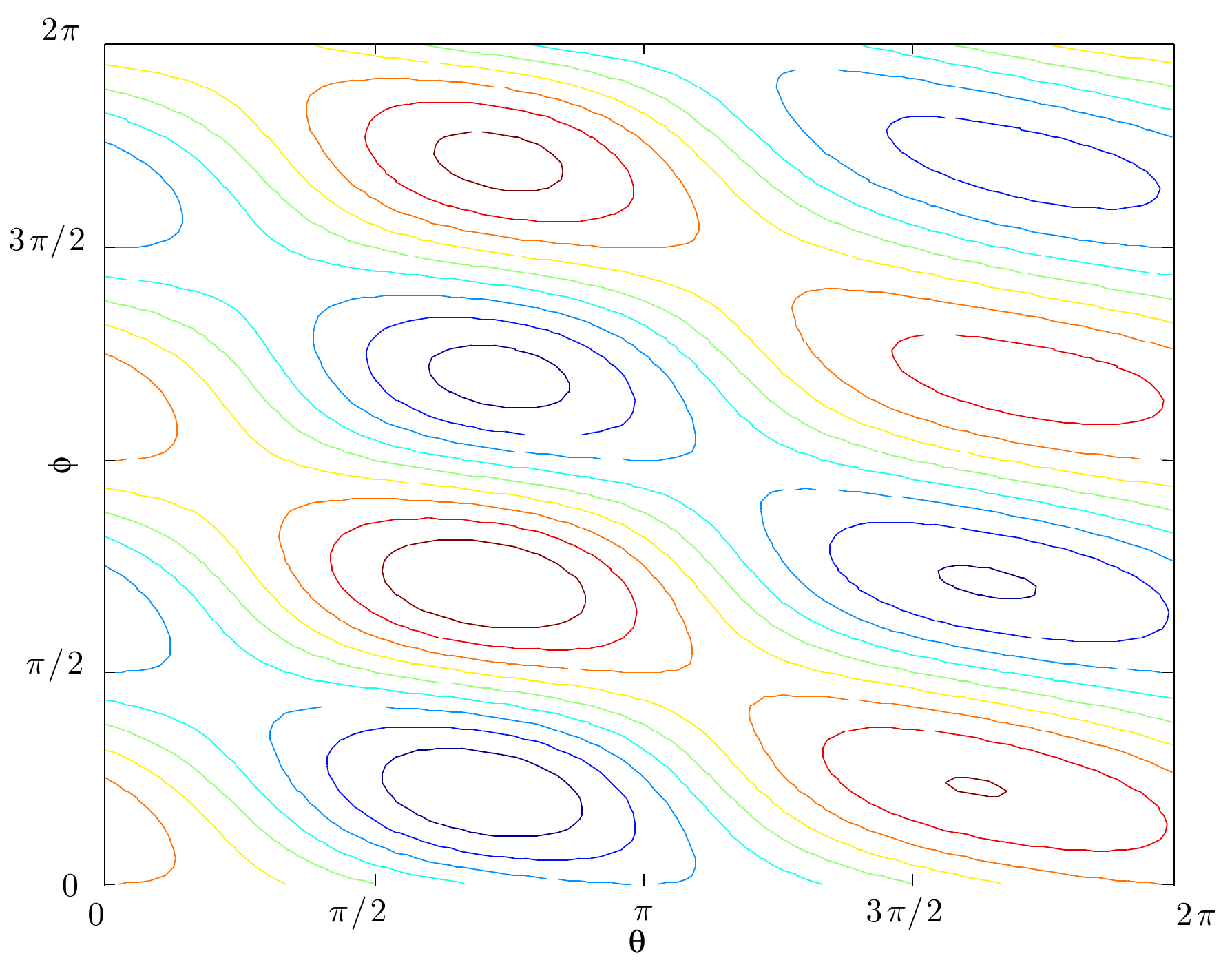}}\\
\subfigure[]{      \includegraphics[height=.3\linewidth]{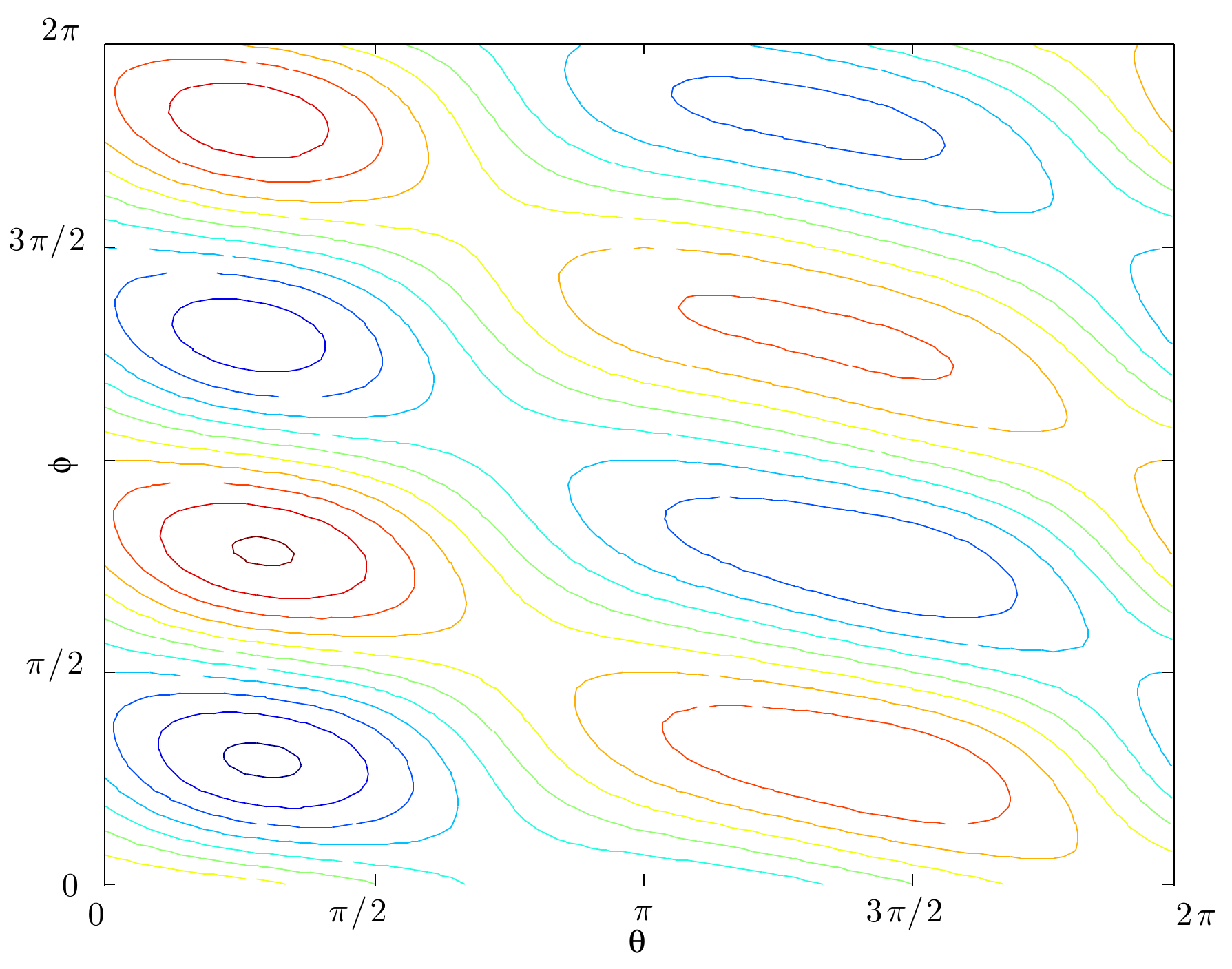}}
\subfigure[]{    \includegraphics[height=.3\linewidth]{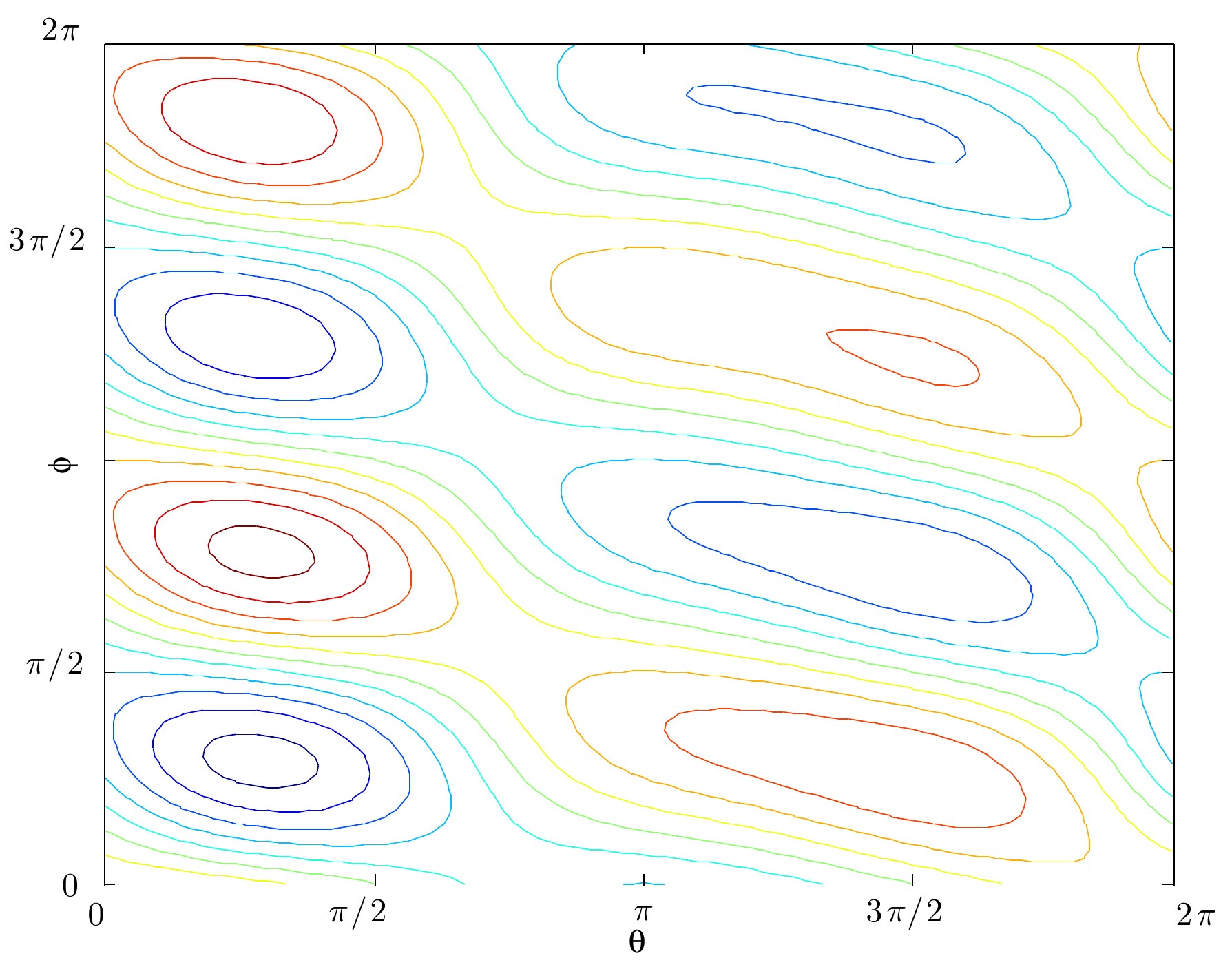}}
\caption{\label{FPBK} (a) The analytic $\mathcal{L}f$ shown as a contour plot as a function of $\theta$ and $\phi$. (b) Numerical estimate $L_{\epsilon}f$ using the local kernel evaluated on $10000$ points on a uniform grid on the flat torus in $\mathbb{R}^4$ with $\epsilon=0.001$ (right). (c) Analytic $\mathcal{L}^*f$ and (d)  $\hat{L}^*_{\epsilon}f$. Note that the color indicates the functional value and the left and right plots are drawn in the same scale.}
\end{center}
\end{figure}

We will validate Theorem \ref{localkerneltheorem} by constructing $L_{\epsilon}$ and $\hat L_{\epsilon}^*$ and applying them to the function $f(\theta,\phi) = \sin\theta\sin 2\phi$.  We first compute the analytic result $\mathcal{L}f$ and $\mathcal{L}^*f$.  Note that because the torus is flat and $x^1 = \theta$ and $x^2 = \phi$ give global coordinates, we can perform all operations with respect to these coordinates. In particular, the covariant derivatives are simply those with respect to $\theta$ and $\phi$ respectively.  Using these facts we compute
\begin{align} \mathcal{L}f &= \mu \cdot \nabla f + \sum_{l,r} C_{lr}\frac{d^2 f}{dx^l dx^r} \nonumber \\
&= (2+\sin(\theta),0)\left(\frac{\partial f}{\partial \theta},\frac{\partial f}{\partial \phi}\right)^\top + (3+\sin(\phi))\frac{\partial^2 f}{\partial \theta^2} + 2 \frac{\partial^2 f}{\partial \theta \partial \phi} + \frac{\partial^2 f}{\partial \phi^2}  \nonumber \\
&= (2+\sin(\theta))\cos(\theta)\sin(2\phi) -  (3+\sin(\phi))\sin(\theta)\sin(2\phi) + 4\cos(\theta)\cos(2\phi) - 4\sin(\theta)\sin(2\phi), \nonumber
\end{align}
and
\begin{align} \mathcal{L}^*f &= -\textup{div}(\mu f) + \sum_{l,r} \frac{d^2}{dx^l dx^r}(C_{lr} f) = -\frac{\partial}{\partial \theta}((2+\sin(\theta))f) + \frac{\partial^2}{\partial \theta^2}((3+\sin(\phi))f) + 2 \frac{\partial^2 f}{\partial \theta \partial \phi} + \frac{\partial^2 f}{\partial \phi^2}  \nonumber \\
&= -(2+\sin(\theta))\cos(\theta)\sin(2\phi)-\cos(\theta)\sin(\theta)\sin(2\phi) -  (3+\sin(\phi))\sin(\theta)\sin(2\phi) + 4\cos(\theta)\cos(2\phi) - 4\sin(\theta)\sin(2\phi). \nonumber
\end{align}
In Figure \ref{FPBK} we show that these analytic formulas compare closely to the discrete estimates $L_{\epsilon}f$ and $\hat L_{\epsilon}^*f$.

\subsection{Connection to nonlinear independent component analysis}

Nonlinear independent component analysis was studied for It\^o processes in \cite{Singer2008}.  The central assumption is that a stochastic process $x(t)$ is generated in an $n$-dimensional latent space, which is then observed by a nonlinear mapping $y(t) = F(x(t))$ into an $m$-dimensional observation space with $m\ge n$.  In the latent space, the process is assumed to have isotropic homogeneous stochastic forcing and drift determined by an arbitrary vector field $\mu(x)$.  Such a process can be described by the It\^o stochastic differential equation
\[ dx = \mu(x) \, dt + \textup{Id} \, dW_t, \]
where $dW_t$ is a Brownian process on the latent space and Id is the identity matrix.  The It\^o Lemma implies that  in the observation space, the process $y(t)$ is given by
\[ dy = \left(DF(x)\mu(x) + \frac{1}{2}\textup{trace}(H(F))\right)\, dt + DF(x) \, dW_t, \]
and the key insight of \cite{Singer2008} is that for $m\geq n$, there is a complete set of observables $\mathbb{E}[dy dy^\top] = DF(x)DF(x)^\top dt$ which allows  determination of  $DF(x)$ up to an orthogonal transformation.  Defining the correlation matrix $C_{y_i} = \mathbb{E}[dydy^\top](y_i)$, they construct the kernel
\[ K_y(\epsilon,y_i,y_j) \equiv \exp\left( -\frac{(y_j-y_i)^\top (C_{y_i}^{-1} + C_{y_j}^{-1})(y_j-y_i)}{4\epsilon} \right).\]
Note that $K_y$ is a local kernel which is closely related to a prototypical local kernel defined in Example \ref{prototype}, and it is easy to check the first moment is zero and the second moment is given by $C(y) = C_y$.  Theorem \ref{localkerneltheorem} reproves the result of \cite{Singer2008} that $L_{\epsilon}$ recovers the generator $\mathcal{L}f = \frac{1}{2}\sum_{i,j} C_{ij}\nabla_i \nabla_j f $.  Noting that ${\displaystyle \frac{\partial}{\partial y_i} = \sum_l (DF^{-1})_{il} \frac{\partial}{\partial x_l}}$, we have
\begin{align}\label{ica} \mathcal{L}f &= \frac{1}{2}\sum_{ij} C_{ij}\nabla_i \nabla_j f = \frac{1}{2}\sum_{ij} C_{ij} \frac{\partial^2 f}{\partial y_i \partial y_j} = \frac{1}{2}\sum_{i,j,l,s} C_{ij} (DF^{-1})_{il} \frac{\partial}{\partial x_l}\left( (DF^{-1})_{js} \frac{\partial \tilde f}{\partial x_s}\right) \nonumber \\  
&= \frac{1}{2}\sum_{i} \frac{\partial^2 \tilde f}{\partial x_i^2} + \frac{1}{2} \sum_{i,j,l,s}C_{ij} (DF^{-1})_{il} \frac{\partial (DF^{-1})_{js}}{\partial x_l}  \frac{\partial \tilde f}{\partial x_s} \nonumber \\
&= \frac{1}{2}\sum_{i} \frac{\partial^2 \tilde f}{\partial x_i^2} + \frac{1}{2} \sum_{j,l,s}DF_{jl} \frac{\partial (DF^{-1})_{js}}{\partial x_l}  \frac{\partial \tilde f}{\partial x_s}
\end{align}
where $\tilde f(x) = f(F^{-1}(y))$ and $\sum_i C_{ij}(DF^{-1})_{il} = DF_{jl}$.  Ignoring the second term of \eqref{ica}, which is an additional drift term, the kernel $K_y$ recovers a homogeneous isotropic diffusion as first shown in \cite{Singer2008}.  We can now extend this result to use the observables $\mathbb{E}[dy] = \left(DF(x)\mu(x) - \frac{1}{2}\textup{trace}(H(f))\right) dt$.  Setting $b(y_i) = \mathbb{E}[dy](y_i)$ and using the prototypical kernel
\[ K(\epsilon,y_i,y_j) = \exp\left(\frac{-(y_j-y_i-\epsilon b(y_i))^T C_{y_i}^{-1}(y_j-y_i-\epsilon b(y_i))}{2\epsilon}\right), \]
we find that $L_{\epsilon}$ converges to the operator
\begin{align}\label{driftica} \mathcal{L}f &= \sum_j b_j \nabla_j f + \frac{1}{2}\sum_{ij} C_{ij}\nabla_i \nabla_j f = \sum_{j,l} b_j(x)(DF^{-1})_{jl}\frac{\partial \tilde f}{\partial x_l} + \frac{1}{2}\sum_{ij} C_{ij}\nabla_i \nabla_j f \nonumber \\
&= \sum_{j,l,s} \left(DF_{jl}\mu_l(x) + \frac{1}{2}\frac{\partial^2 F^j}{\partial x_l^2}\right)(DF^{-1})_{js}\frac{\partial \tilde f}{\partial x_s} + \frac{1}{2}\sum_{ij} C_{ij}\nabla_i \nabla_j f \nonumber \\
&= \mu \cdot \nabla \tilde f +  \frac{1}{2}\sum_{j,l,s} \frac{\partial DF_{jl}}{\partial x_l}(DF^{-1})_{js}\frac{\partial \tilde f}{\partial x_s} + \frac{1}{2}\sum_{ij} C_{ij}\nabla_i \nabla_j f .
\end{align}
Since ${\displaystyle \sum_{jl}  \frac{\partial DF_{jl}}{\partial x_l}(DF^{-1})_{js} = -\sum_{jl} DF_{jl} \frac{\partial (DF^{-1})_{js}}{\partial x_l}}$, these terms cancel, yielding
\begin{align}\label{driftica1} \mathcal{L}f = \mu \cdot \nabla \tilde f +  \frac{1}{2}\Delta \tilde f , \end{align}
which is the generator of the It\^o diffusion process $x(t)$ in the latent space.  In the next section we will reinterpret this change of variables as a change in the Riemannian metric on the manifold.

\section{The intrinsic geometry of symmetric local kernels}\label{kernelgeometry}

In this section we consider local kernels that are symmetric in $x$ and $y$.  We will show that the operator $G_{\epsilon}$ associated to a symmetric local kernel is always a Laplacian with respect to a certain Riemannian metric, which depends on the second moment $C$ of the kernel and the metric $g$ inherited from the ambient space.  

\begin{Def}[Symmetric kernel] \rm A kernel function $K(\epsilon,x,y)$ is called {\it symmetric} if it is equal to its adjoint, $K(\epsilon,x,y) = K^*(\epsilon,x,y) = K(\epsilon,y,x)$.
\end{Def}

Notice that $\overline K = K + K^*$ is always symmetric, and any symmetric kernel can be trivially written in this form.  The results in this section will not assume that $K$ is symmetric but they will focus on the expansion of $\overline K$.  In order to connect symmetric kernels to the Laplacian operator, we first connect the operators $C_{ij}\nabla_i \nabla_j f$ and $\nabla_j \nabla_i (C_{ij} f)$ to the Laplacian with respect to a new Riemannian metric.

\begin{lem}[Change of metric]\label{newmetric} Let $(\mathcal{M},g)$ be a Riemannian manifold and let $\mu(x)$ be a vector field and $C(x)$ a (1,1)-tensor on $\mathcal{M}$.  Define the new metric  $\hat g(u,v) = g(C^{-1/2}u,C^{-1/2}v)$ which we denote $\hat g = C^{-1/2}gC^{-1/2}$. Then
\[ \sum_{i,j} C_{ij}\nabla_i \nabla_j f = \Delta_{\hat g}f + \kappa \cdot \nabla f \hspace{50pt}  \sum_{i,j} \nabla_j\nabla_i (C_{ij}f) = \rho^{-1}\Delta_{\hat g}(\rho f) - \textup{div}(\kappa f) \]
where $\Delta_{\hat g}$ is the Laplacian with respect to $\hat g$ and all other operators and inner products are with respect to $g$, where $\kappa$ is a vector field which depends on $g$ and $C$, and where $\rho = \sqrt{|g|/|\hat g|} = \sqrt{|C|}$ is a scalar function.
\end{lem}
\begin{proof}
Let $C = C(x)$ be the matrix with entries $C_{ij}(x)$ and define new coordinates $\hat s = C^{-1/2}s$ so that $\frac{d\hat s_l}{ds_j} = C_{lj}^{-1/2}$ and $\frac{df}{ds_j}(0) = \sum_{l} \frac{d\hat s_l}{ds_j} \frac{df}{d\hat s_l}(0) = \sum_l C_{lj}^{-1/2}  \frac{df}{d\hat s_l}(0)$ and therefore
\[ \frac{\partial^2 \tilde f}{\partial s_i \partial s_j} = \sum_{k,l}C^{-1/2}_{ki} \frac{\partial}{\partial \hat s_k} \left( C^{-1/2}_{lj} \frac{\partial \tilde f}{\partial \hat s_l} \right) = \sum_{k,l} C^{-1/2}_{ki}C_{lj}^{-1/2} \frac{\partial^2 \tilde f}{\partial \hat s_k \partial \hat s_l} + C_{ki}^{-1/2} \frac{\partial \tilde C^{-1/2}_{lj}}{\partial \hat s_k} \frac{\partial \tilde f}{\partial \hat s_l}.  \]
Substituting the above expression into the summation $\sum_{ij}C_{ij}\frac{\partial^2 \tilde f}{\partial s_i \partial s_j}(0)$ we find
\begin{align} \sum_{i,j} C_{ij} \frac{\partial^2 \tilde f}{\partial s_i \partial s_j}(0) &= \sum_{i,j,k,l} \left(C_{ij}C^{-1/2}_{ki}C_{lj}^{-1/2} \frac{\partial^2 \tilde f}{\partial \hat s_k \partial \hat s_l} + C_{ij}C_{ki}^{-1/2} \frac{\partial \tilde C^{-1/2}_{lj}}{\partial \hat s_k} \frac{\partial \tilde f}{\partial \hat s_l}\right) \nonumber \\ 
&= \sum_{k,l} \left[ \left(\sum_{i,j} C_{ij}C^{-1/2}_{ki}C_{lj}^{-1/2} \right) \frac{\partial^2 \tilde f}{\partial \hat s_k \partial \hat s_l} + \left(\sum_{i,j} C_{ij}C^{-1/2}_{ki}\frac{\partial \tilde C^{-1/2}_{lj}}{\partial \hat s_k} \right) \frac{\partial \tilde f}{\partial \hat s_l}\right] \nonumber \\
&= \sum_{k,l} \left[ \delta_{ik} \frac{\partial^2 \tilde f}{\partial \hat s_k \partial \hat s_l} + \left(\sum_{j} C^{1/2}_{jk} \frac{\partial \tilde C^{-1/2}_{lj}}{\partial \hat s_k} \right) \frac{\partial \tilde f}{\partial \hat s_l}\right] \nonumber \\
&= \sum_{k}  \frac{\partial^2 \tilde f}{\partial \hat s_k ^2 } + \sum_{k,l,j} C^{1/2}_{jk} \frac{\partial \tilde C^{-1/2}_{lj}}{\partial \hat s_k} \frac{\partial \tilde f}{\partial \hat s_l}, \nonumber
\end{align}
where all the derivatives are evaluated at $s = 0$.  Notice that the first term corresponds to the Laplacian $\Delta_{\hat g}f(x) =  \sum_{k}  \frac{\partial^2 \tilde f}{\partial \hat s_k ^2 }(0)$ and we can rewrite the second term as
\[ \sum_{k,l,j} C^{1/2}_{jk} \frac{\partial \tilde C^{-1/2}_{lj}}{\partial \hat s_k} \frac{\partial \tilde f}{\partial \hat s_l} = \sum_{i,k,l,j,r} C^{1/2}_{jk} C^{1/2}_{kr} \frac{\partial \tilde C^{-1/2}_{lj}}{\partial s_r} C^{1/2}_{li} \frac{\partial \tilde f}{\partial s_i} = \sum_{i,l,j,r} C_{jr}\frac{\partial \tilde C^{-1/2}_{lj}}{\partial s_r} C^{1/2}_{li} \frac{\partial \tilde f}{\partial s_i}.\]
We now define the vector field ${\displaystyle \kappa_i =  \sum_{l} \left(\sum_{j,r} C_{jr}\frac{\partial \tilde C^{-1/2}_{lj}}{\partial s_r} \right)C^{1/2}_{li}}$ so the previous expression can be simplified as
\[  \sum_{i,j} C_{ij} \frac{\partial^2 \tilde f}{\partial s_i \partial s_j}(0) = \Delta_{\hat g}f(x) + \kappa(x) \cdot \nabla f(x) \]
as desired.  To find $\nabla_j\nabla_i(C_{ij} f)$ note that with respect to the inner product $\left<\cdot,\cdot\right>$ on $L^2(\mathcal{M},g)$ we have
\begin{align} \left<\nabla_j\nabla_i(C_{ij} f),h\right> &= \left<f, C_ij \nabla_i\nabla_j h \right> = \left<f,\Delta_{\hat g}h\right> + \left<f,\kappa \cdot \nabla h \right> \nonumber \\
&= \int f \Delta_{\hat g}h \sqrt{|g|}\ dx + \left<-\textup{div}(f\kappa),h \right> = \int f \rho \Delta_{\hat g}h \sqrt{|\hat g|}\ dx + \left<-\textup{div}(f\kappa),h \right> \nonumber \\
&= \int \Delta_{\hat g}(f \rho) h \sqrt{|\hat g|}\ dx + \left<-\textup{div}(f\kappa),h \right> = \left<\rho^{-1}\Delta_{\hat g}(\rho f) -\textup{div}(\kappa f),h \right>,
\end{align}
for an arbitrary smooth test function $h$, therefore $ \sum_{i,j}\nabla_j\nabla_i(C_{ij} f) = \rho^{-1}\Delta_{\hat g}(\rho f) -\textup{div}(\kappa f)$.
\end{proof}
 
Applying Lemma \ref{newmetric} to the sum $\mathcal{L}+\mathcal{L}^*$ we have the following lemma.

\begin{lem} Let $\mathcal{L}$ and $\mathcal{L}^*$ denote the operators in (\ref{eqScriptL}). Under the assumptions of Lemma \ref{newmetric},
\[ \mathcal{L}f+\mathcal{L^*}f = \Delta_{\hat g}f + \nabla_{\hat g} f \cdot \frac{\nabla_{\hat g}\rho}{\rho} + f \tilde\omega, \]
where $\nabla_{\hat g}$ is the gradient with respect to $\hat g$, $\rho = \sqrt{|C|}$ and $\tilde\omega$ is a scalar function which depends on $\mu,C,$ and $g$.
\end{lem}
\begin{proof} From the previous lemma we have
\begin{align} \mathcal{L}^*f &= -\textup{div}(f\mu) + \frac{1}{2}\rho^{-1}\Delta_{\hat g}(\rho f) - \frac{1}{2}\textup{div}(\kappa f) \nonumber \\
&= \frac{1}{2}\rho^{-1}\left(\rho \Delta_{\hat g} f + f\Delta_{\hat g}\rho + 2\nabla_{\hat g}f\nabla_{\hat g}\rho\right) - \mu \cdot \nabla f - \frac{1}{2}\kappa \cdot \nabla f + f\textup{div}(\mu-\kappa/2) \nonumber \\
&= \frac{1}{2}\Delta_{\hat g}f + \nabla_{\hat g}f\nabla_{\hat g}\rho - \mathcal{L}f + \frac{1}{2}\Delta_{\hat g}f + f\left(\rho^{-1}\Delta_{\hat g}\rho +\textup{div}(\mu-\kappa/2) \right).
\end{align} 
Letting $\tilde \omega = \rho^{-1}\Delta_{\hat g}\rho +\textup{div}(\mu-\kappa/2)$ and moving $\mathcal{L}f$ to the left side yields the desired result.
\end{proof}

Combining the previous lemma with the expansion of the local kernel $K$ and its adjoint $K^*$ from Section \ref{localkernels}, we define the symmetric kernel $\overline K = K + K^*$.

\begin{thm}[Expansion of symmetric kernel]\label{symmtheorem} Let $K$ be a local kernel and define $\overline K \equiv K + K^*$. Then 
\[ \overline{G}_{\epsilon}f(x) \equiv \epsilon^{-d/2} \int_{\mathcal{M}} \overline K(\epsilon,x,y)f(y)dy = 2m(x)f(x) + \epsilon\left((2\omega(x)+\tilde \omega(x))f(x) + \Delta_{\hat g}f + \nabla_{\hat g} f \cdot \frac{\nabla_{\hat g}\rho}{\rho}\right) + \mathcal{O}(\epsilon^2) \]
and
\[ \overline L_{\epsilon} = \frac{1}{\epsilon}\left( (\overline G_{\epsilon}1)^{-1}\overline G_{\epsilon}f - \textup{Id}(f) \right) =  \Delta_{\hat g}f + \nabla_{\hat g} f \cdot \frac{\nabla_{\hat g}\rho}{\rho} + \mathcal{O}(\epsilon^2)\]
where $\hat g = C^{-1/2}g C^{-1/2}$ and $\rho = \sqrt{|C|}$.
\end{thm}

If $C(x)$ is the identity map, then $\hat g = g$ and we recover Lemma \ref{diffmaplemma} extended to all local isotropic kernels, so we no longer need to assume to specific form $h(||x-y||^2/\epsilon)$.  Moreover, for the prototypical kernel,  the following corollary holds.

\begin{cor}[Expansion of prototypical kernel]\label{symmPrototype} Let $K$ be a local kernel with zeroth moment $m(x) = m_0 \sqrt{|A(x)|}$ and second moment $C(x) = m(x)A(x)$ (such as the prototypical kernel). Then
\[ \overline{G}_{\epsilon}f(x) \equiv \epsilon^{-d/2} \int_{\mathcal{M}} \overline K(\epsilon,x,y)f(y)dy = 2m(x)f(x) + \epsilon\left((2\omega(x)+\tilde \omega(x))f(x) +m_0 q \Delta_{\tilde g}f + 2m_0 \nabla_{\tilde g} f \cdot \nabla_{\tilde g}q \right) + \mathcal{O}(\epsilon^2) \]
where $\tilde g = A^{-1/2}gA^{-1/2}$ and $q = \sqrt{|A|}$.
\end{cor}
\begin{proof}
We have $C(x) = m(x)A(x) = m_0\sqrt{|A(x)|}A(x) = m_0 q(x)A(x)$, which means that
\[ \hat g = C^{-1/2}gC^{-1/2} = m_0^{-1} |A|^{-1/2}A^{-1/2}g A^{-1/2} = m_0^{-1} q^{-1}\tilde g.\] 
Thus $\hat g$ is conformal to $\tilde g$ and we have the following standard relationship for $\Delta_{\hat g}$ and $\Delta_{\tilde g}$:
\[ \Delta_{\hat g}f = m_0 q \Delta_{\tilde g} f + (1 - d/2) m_0 \nabla_{\tilde g} f \cdot \nabla_{\tilde g}q.  \]
Moreover, since $\rho = \sqrt{|C|} = |A|^{\frac{d+2}{4}} = m_0^{d/2} q^{d/2 + 1}$ and $\hat g^{ij} = m_0 q \tilde g^{ij}$ we have
\[ \nabla_{\hat g}f \cdot \frac{\nabla_{\hat g}\rho}{\rho} = \rho^{-1}  \sum_{ij} \hat g^{ij} \partial_i f \partial_j \rho = m_0 q^{-d/2-1} \sum_{ij} q\tilde g^{ij} \partial_i f (d/2+1)q^{d/2}\partial_j q = (d/2+1)m_0\nabla_{\tilde g}f \cdot \nabla_{\tilde g} q, \]
and combining this with the above formula yields
\[ \Delta_{\hat g}f +\nabla_{\hat g}f \cdot \frac{\nabla_{\hat g}\rho}{\rho} =m_0 q \Delta_{\tilde g} f + 2m_0 \nabla_{\tilde g}f \cdot \nabla_{\tilde g} q. \]
The result then follows from Theorem \ref{symmtheorem}.
\end{proof}

Notice that we do not consider the standard normalization for the prototypical kernel.  This is because the prototypical kernels have the following unique interpretation.  Let $\mathcal{H}:\mathcal{N} \to \mathcal{M} \subset \mathbb{R}^n$ be an embedding where $(\mathcal{N},g_{\mathcal{N}})$ is an abstract Riemannian manifold and $\mathcal{M}$ is the embedded manifold that our data lies on.  Assume that the data $x_i = \mathcal{H}(\tilde x_i)$ was originally sampled uniformly on $\mathcal{N}$ and then mapped into $\mathbb{R}^n$ by $\mathcal{H}$.  Set $q(x) = |D\mathcal{H}(\mathcal{H}^{-1}(x))|$, where the determinant is computed on $T_x\mathcal{M}$.  The sampling measure on $\mathcal{M}$ will be $q(x)^{-1}$.  This is crucial because we estimate the integral operator $\overline G_{\epsilon}f(x)$ as a Monte Carlo integral,
\[  \lim_{N \to \infty} \sum_{j=1}^N \overline K(\epsilon,x_i,x_j)f(x_j) = \int_{\mathcal{M}} K(\epsilon,x_i,y)f(y)q(y)^{-1}dy = \epsilon^{d/2}\overline G_{\epsilon}(fq^{-1})(x_i). \]
The fact that the data $x_i$ have sampling density $q^{-1}$ will bias our estimate of $\overline G_{\epsilon}$ and we will use this to our advantage.  

\begin{thm}[Intrinsic geometry of local kernels]\label{uniformmaintheorem} Let $(\mathcal{N}, g_{\mathcal{N}})$ be an abstract Riemannian manifold and let $\{\tilde x_i\}_{i=1}^N \subset \mathcal{N}$ be sampled uniformly according to the volume form defined by $g_\mathcal{N}$.  Let $\mathcal{H}:\mathcal{N} \hookrightarrow \mathbb{R}^n$ be an embedding with image $\mathcal{M}=\mathcal{H}(\mathcal{N})$ and let $x_i = \mathcal{H}(\tilde x_i)$. Define $A(x_i) = D\mathcal{H}(\tilde x_i)D\mathcal{H}(\tilde x_i)^\top$.  For any local kernel $K$ with $m(x)=\sqrt{|A(x)|}$ and covariance $C(x) = \sqrt{|A(x)|}A(x)$ (such as a prototypical kernel), and any smooth function $f$ on $\mathcal{M}$,
\[ \lim_{N\to\infty} \frac{2}{\epsilon}
\left( \frac{\sum_{j} \overline K(\epsilon,x_i,x_j)f(x_j)}{\sum_j \overline K(\epsilon,x_i,x_j)}-f(x_i)\right)
 = \Delta_{\tilde g}f(x_i) + \mathcal{O}(\epsilon) = \Delta_{g_{\mathcal{N}}} (f\circ \mathcal{H})(\tilde x_i) + \mathcal{O}(\epsilon) \]
where $\overline K(\epsilon,x,y) = K(\epsilon,x,y) + K(\epsilon,y,x)$ and $\tilde g(u,v) = g_{\mathcal{N}}(D\mathcal{H}^{-1}u,D\mathcal{H}^{-1}v)$.
\end{thm}
\begin{proof} Note that since $\tilde x_i$ are uniformly sampled, the data $\{x_i\}$ have density $q(x_i)^{-1}$ where $q(x) = |D\mathcal{H}(\mathcal{H}^{-1}(x))| = \sqrt{|A|}$.  This biases the Monte Carlo integral so that
\[ \lim_{N\to\infty} \epsilon^{-d/2} \sum_{j} \overline K(\epsilon,x_i,x_j)f(x_j) = \overline G_{\epsilon}(fq^{-1}), \]
and applying the previous corollary we have,
\[  \overline G_{\epsilon}(fq^{-1}) = 2m_0 m f q^{-1} + \epsilon\left((2\omega+\tilde \omega)fq^{-1} + q m_0 \Delta_{\tilde g}(fq^{-1}) + 2 m_0\nabla_{\tilde g} (fq^{-1}) \cdot \nabla_{\tilde g}q \right) + \mathcal{O}(\epsilon^2). \]
Note that when $f = 1$,
\[ \overline G_{\epsilon}(q^{-1})  = 2m_0 mq^{-1} + \epsilon\left((2\omega+\tilde \omega)q^{-1} + q m_0\Delta_{\tilde g}q^{-1}  + 2 m_0\nabla_{\tilde g} q^{-1}  \cdot \nabla_{\tilde g}q \right) + \mathcal{O}(\epsilon^2). \]
Expanding the ratio using the general fact that $\frac{a+\epsilon b}{c+\epsilon d} = \frac{a}{c} + \epsilon \frac{bc-ad}{c^2} + \mathcal{O}(\epsilon^2)$, and noting that $m q^{-1} = 1$ yields
\[ \frac{\overline G_{\epsilon}(fq^{-1})}{\overline G_{\epsilon}(q^{-1})} = f + \frac{\epsilon}{2} \left( q \Delta_{\tilde g}(fq^{-1}) + 2 \nabla_{\tilde g} (fq^{-1}) \cdot \nabla_{\tilde g}q - q f \Delta_{\tilde g}q^{-1}  - 2 f \nabla_{\tilde g} q^{-1}  \cdot \nabla_{\tilde g}q\right) + \mathcal{O}(\epsilon^2), \]
and applying the product rule $\Delta_{\tilde g}(fq^{-1}) = q^{-1}\Delta_{\tilde g} f + f\Delta_{\tilde g}(q^{-1}) + 2 \nabla_{\tilde g} f \cdot \nabla_{\tilde g} (q^{-1})$, we have
\[ \frac{2}{\epsilon}\left(\frac{\overline G_{\epsilon}(fq^{-1})}{\overline G_{\epsilon}(q^{-1})} - f \right) = \Delta_{\tilde g} f + \mathcal{O}(\epsilon), \]
and the limit as $N\to\infty$ follows.  Note that $\mathcal{H}$ is an isometry from $(\mathcal{N},g_{\mathcal{N}})$ to $(\mathcal{M},\tilde g)$ since $\tilde g(D\mathcal{H}u,D\mathcal{H}v) = g_{\mathcal{N}}(u,v)$ and therefore $\mathcal{H}^*(f) = f\circ \mathcal{H}$ commutes with the Laplacian. It follows that
\[ \Delta_{\tilde g}f(x) = (\Delta_{\tilde g}f)(\mathcal{H}(\tilde x)) = H^*(\Delta_{\tilde g}f)(\tilde x) = \Delta_{g_{\mathcal{N}}} (H^*f)(\tilde x) = \Delta_{g_{\mathcal{N}}} (f\circ\mathcal{H})(\tilde x)) \]
which completes the proof.
\end{proof}

We will call $(\mathcal{N},g_{\mathcal{N}})$ the \emph{intrinsic geometry} of the manifold $\mathcal{M}$ with respect to the local kernel $K$.  The previous theorem shows that, in direct analogy to the results of \cite{BN} and \cite{diffusion}, a symmetric local kernel defines a Laplacian operator on the intrinsic geometry.  In fact, the Laplacian $\Delta_{g_{\mathcal{N}}}$ is equivalent to the Riemannian metric $g_{\mathcal{N}}$ in the sense that one can be uniquely recovered from the other \cite{Jost}.  Unless the embedding $\cal H$ is isometric, the Riemannian metric $\tilde g$ will not agree with the metric $g$ that $\mathcal{M}$ inherits from $\mathbb{R}^n$.  When the embedding is isometric, the second moment of the local kernel will be the identity matrix and therefore $\tilde g = g$, recovering the result of standard diffusion maps \cite{diffusion} for uniform sampling.

Theorem \ref{uniformmaintheorem} is restricted by the assumption of uniform sampling in the intrinsic geometry.  In the next section we generalize this result to allow any smooth sampling density on $\mathcal{N}$.

\subsection{Nonuniform sampling in the intrinsic geometry}\label{nonuniform}

Theorem \ref{uniformmaintheorem} assumes that the data points $x_i = \mathcal{H}(\tilde x_i)$ are generated by sampling $\tilde x_i$ uniformly on the intrinsic geometry $(\mathcal{N},g_{\mathcal{N}})$.  This means that the data are sampled according to the volume form defined by $g_{\mathcal{N}}$.  As was first noted in \cite{diffusion}, this is a restrictive assumption for applications that do not have control over the sampling.  The solution introduced in \cite{diffusion} is the right-normalization discussed in Section \ref{dmbackground}, and here we replicate this technique for local kernels.  

\begin{thm}[Intrinsic geometry of local kernels, with nonuniform sampling]\label{maintheorem} Let $(\mathcal{N}, g_{\mathcal{N}})$ be an abstract Riemannian manifold and let $\{\tilde x_i\}_{i=1}^N \subset \mathcal{N}$ be sampled according to any smooth density on $\mathcal{N}$.  Let $\mathcal{H}:\mathcal{N} \hookrightarrow \mathbb{R}^n$ be an embedding with image $\mathcal{M}=\mathcal{H}(\mathcal{N}) \subset  \mathbb{R}^n$ and let $x_i = \mathcal{H}(\tilde x_i)$ and define $A(x_i) = D\mathcal{H}(\tilde x_i)D\mathcal{H}(\tilde x_i)^\top$.  For any local kernel $K$ with $m(x)=\sqrt{|A(x)|}$ and covariance $C(x) = \sqrt{|A(x)|}A(x)$ (such as a prototypical kernel), and any smooth function $f$ on $\mathcal{M}$,
\[ \lim_{N\to\infty} \frac{2}{\epsilon}\left(  \frac{{\displaystyle\sum_{j=1}^N} { \overline K(\epsilon,x_i,x_j)f(x_j)}/{\sum_l \overline K(\epsilon,x_j,x_l)}}{{\displaystyle \sum_{j=1}^N} { \overline K(\epsilon,x_i,x_j)}/{\sum_l \overline K(\epsilon,x_j,x_l)}}-f(x_i)\right) = \Delta_{\tilde g}f(x_i) + \mathcal{O}(\epsilon)= \Delta_{g_{\mathcal{N}}} (f\circ \mathcal{H})(\tilde x_i) + \mathcal{O}(\epsilon)\]
where $\overline K(\epsilon,x,y) = K(\epsilon,x,y) + K(\epsilon,y,x)$ and $\tilde g(u,v) = g_{\mathcal{N}}(D\mathcal{H}^{-1}u,D\mathcal{H}^{-1}v)$.
\end{thm}
\begin{proof}
Assume that $\tilde x_i$ are sampled from $\mathcal{N}$ with density $q_{\mathcal{N}}$ written with respect to the volume form defined by $g_{\mathcal{N}}$.  This density biases the data $x_i$ so that their density is now $q_{\mathcal{N}}q^{-1}$, which biases the Monte Carlo integral so that
\[ \lim_{N\to\infty} \epsilon^{-d/2} \sum_{j} \overline K(\epsilon,x_i,x_j)f(x_j) = \overline G_{\epsilon}(fq_{\mathcal{N}}q^{-1}). \]
Applying the Corollary \ref{symmPrototype} and recalling that $m=q^{-1}$, we have
\[  \overline G_{\epsilon}(fq_{\mathcal{N}}q^{-1}) = 2m_0 f q_{\mathcal{N}}+ \epsilon\left((2\omega+\tilde \omega)fq_{\mathcal{N}}q^{-1} + q m_0 \Delta_{\tilde g}(fq_{\mathcal{N}}q^{-1}) + 2 m_0\nabla_{\tilde g} (fq_{\mathcal{N}}q^{-1}) \cdot \nabla_{\tilde g}q \right) + \mathcal{O}(\epsilon^2) \]
and setting $f = 1$ yields
\[ \overline G_{\epsilon}(q_{\mathcal{N}}q^{-1})  = 2m_0 q_{\mathcal{N}} + \epsilon\left((2\omega+\tilde \omega)q_{\mathcal{N}}q^{-1} + qm_0 \Delta_{\tilde g}(q_{\mathcal{N}}q^{-1})  + 2m_0 \nabla_{\tilde g} (q_{\mathcal{N}}q^{-1})  \cdot \nabla_{\tilde g}q \right) + \mathcal{O}(\epsilon^2). \]
We now introduce the right-normalization
\begin{align} \overline G_{\epsilon}\left(\frac{fq_{\mathcal{N}}q^{-1}}{\overline G_{\epsilon} q_{\mathcal{N}}q^{-1}} \right) &= \frac{2m_0 f q_{\mathcal{N}}}{2m_0 q_{\mathcal{N}} + \epsilon\left((2\omega+\tilde \omega)q_{\mathcal{N}}q^{-1} + qm_0 \Delta_{\tilde g}(q_{\mathcal{N}}q^{-1})  + 2m_0 \nabla_{\tilde g} (q_{\mathcal{N}}q^{-1})  \cdot \nabla_{\tilde g}q \right)}  \nonumber \\
&\hspace{10pt}+ \epsilon\left((2\omega+\tilde \omega)\frac{f}{2m_0q_{\mathcal{N}}}q_{\mathcal{N}}q^{-1} + q m_0 \Delta_{\tilde g}\left(\frac{f}{2m_0q_{\mathcal{N}}}q_{\mathcal{N}}q^{-1}\right) + 2 m_0\nabla_{\tilde g} \left(\frac{f}{2m_0q_{\mathcal{N}}}q_{\mathcal{N}}q^{-1}\right) \cdot \nabla_{\tilde g}q \right) + \mathcal{O}(\epsilon^2) \nonumber \\
&= \frac{f}{1 + \frac{\epsilon}{2}\left((2\omega+\tilde \omega)m_0^{-1}q^{-1} + q q_{\mathcal{N}}^{-1} \Delta_{\tilde g}(q_{\mathcal{N}}q^{-1})  + 2q_{\mathcal{N}}^{-1} \nabla_{\tilde g} (q_{\mathcal{N}}q^{-1})  \cdot \nabla_{\tilde g}q \right)}  \nonumber \\
&\hspace{10pt}+ \frac{\epsilon}{2}\left((2\omega+\tilde \omega)m_0^{-1}f q^{-1} + q \Delta_{\tilde g}\left(f q^{-1}\right) + 2 \nabla_{\tilde g} \left(fq^{-1}\right) \cdot \nabla_{\tilde g}q \right) + \mathcal{O}(\epsilon^2) \nonumber \\
&= f + \frac{\epsilon}{2}\left( q \Delta_{\tilde g}\left(f q^{-1}\right) + 2 \nabla_{\tilde g} \left(fq^{-1}\right) \cdot \nabla_{\tilde g}q - fq q_{\mathcal{N}}^{-1} \Delta_{\tilde g}(q_{\mathcal{N}}q^{-1})  -  2 f q_{\mathcal{N}}^{-1} \nabla_{\tilde g} (q_{\mathcal{N}}q^{-1})  \cdot \nabla_{\tilde g}q \right) + \mathcal{O}(\epsilon^2) \nonumber \\
&= f + \frac{\epsilon}{2}\left( q \Delta_{\tilde g}\left(f q^{-1}\right) + 2 \nabla_{\tilde g} \left(fq^{-1}\right) \cdot \nabla_{\tilde g}q + f\hat\omega \right)  + \mathcal{O}(\epsilon^2) \nonumber
\end{align}
where $\hat\omega = - q q_{\mathcal{N}}^{-1} \Delta_{\tilde g}(q_{\mathcal{N}}q^{-1})  -  2  q_{\mathcal{N}}^{-1} \nabla_{\tilde g} (q_{\mathcal{N}}q^{-1})  \cdot \nabla_{\tilde g}q$.  We note that by linearity of $\Delta$ and $\nabla$ we can neglect the order $\epsilon$ term in the denominator when plugging into these operators since they are already order $\epsilon$.  We now apply left-normalization to $\hat G_{\epsilon}(f) \equiv \overline G_{\epsilon}\left(\frac{fq_{\mathcal{N}}q^{-1}}{\overline G_{\epsilon} q_{\mathcal{N}}q^{-1}} \right)$ so that
\begin{align} \frac{\hat G_{\epsilon}(f)}{\hat G_{\epsilon}(1)} &= f + \frac{\epsilon}{2}\left( q \Delta_{\tilde g}\left(f q^{-1}\right) + 2 \nabla_{\tilde g} \left(fq^{-1}\right) \cdot \nabla_{\tilde g}q - f q \Delta_{\tilde g}\left(q^{-1}\right) - 2f \nabla_{\tilde g} \left(q^{-1}\right) \cdot \nabla_{\tilde g}q \right) + \mathcal{O}(\epsilon^2)  \nonumber \\
&= f + \frac{\epsilon}{2} \Delta_{\tilde g} f + \mathcal{O}(\epsilon^2). \nonumber
\end{align}
The conclusion follows from noting that 
%$\frac{2}{\epsilon}\left( \frac{\sum_j \frac{ \overline K(\epsilon,x_i,x_j)f(x_j)}{\sum_l \overline K(\epsilon,x_j,x_l)}}{\sum_j \frac{ \overline K(\epsilon,x_i,x_j)}{\sum_l \overline K(\epsilon,x_j,x_l)}}-f(x_i)\right) $ 
${\displaystyle \lim_{N\to\infty} \frac{2}{\epsilon}\left(  \frac{{\displaystyle\sum_{j=1}^N} { \overline K(\epsilon,x_i,x_j)f(x_j)}/{\sum_l \overline K(\epsilon,x_j,x_l)}}{{\displaystyle \sum_{j=1}^N} { \overline K(\epsilon,x_i,x_j)}/{\sum_l \overline K(\epsilon,x_j,x_l)}}-f(x_i)\right) }$
converges to \newline
${\displaystyle \frac{2}{\epsilon}\left(\frac{\hat G_{\epsilon}(f)(x_i)}{\hat G_{\epsilon}(1)(x_i)}-f(x_i)\right)}$ as $N\to \infty$.
\end{proof}

Theorem \ref{maintheorem} allows us to recover the intrinsic geometry independent of the sampling on $(\mathcal{N},g_{\mathcal{N}})$ by using the right-normalization.  The right-normalization is equivalent to the diffusion maps normalization with $\alpha=1$ in \cite{diffusion}.  This establishes our key result, which is that every local kernel defines a geometry in the limit of large data.  Of course, many local kernels could define the same intrinsic geometry, and the previous theorem reveals that it is the second moment of the kernel which determines the intrinsic geometry.  The next theorem establishes the converse: that every Riemannian geometry on a manifold can be represented by a local kernel.

\begin{thm}\label{mainconverse} Let $\mathcal{M}\subset\mathbb{R}^n$ be an embedded Riemannian manifold with $g$ the induced metric and let $(\mathcal{N},g_{\mathcal{N}})$ be any manifold diffeomorphic to $\mathcal{M}$.  There exists a local kernel $K$ such that $(\mathcal{N},g_{\mathcal{N}})$ is the intrinsic geometry of $(\mathcal{M},g)$ with respect to $K$.
\end{thm}
\begin{proof} Since $\mathcal{M}$ and $\mathcal{N}$ are diffeomorphic and since $g_{\mathcal{N}}$ and $g$ are positive definite we can always find  a diffeomorphism $\mathcal{H}:\mathcal{N}\to\mathcal{M}$ such that $g_{\mathcal{N}}(u,v) = g(D\mathcal{H}u,D\mathcal{H}v)$. Let $K$ be the prototypical local kernel with $A = D\mathcal{H}D\mathcal{H}^\top$, then $(\mathcal{N},g_{\mathcal{N}})$ is the intrinsic geometry of $\mathcal{M}$ with respect to $K$.
\end{proof}

Together, Theorems \ref{maintheorem} and \ref{mainconverse} show that Riemannian metrics are in one-to-one correspondence with equivalence classes of local kernels that have the same second moment tensor $C(x)$.  Of course, it is also possible to use diffusion maps to find the Laplacian with respect to any Riemannian metric.  By Nash's theorem \cite{laplacianBook}, every Riemannian manifold admits an isometric embedding into $\mathbb{R}^M$ for $M$ large enough.  To recover the intrinsic geometry $g_{\mathcal{N}}$ with diffusion maps we would have to find a global isometric embedding of $(\mathcal{N},g_{\mathcal{N}})$ into a Euclidean space.  Of course, in practice, finding such a global isometric embedding would be quite difficult.

Theorem \ref{maintheorem} provides an alternative which is valuable in two respects.  First, a local kernel allows one to easily change the metric using only local information without having to construct a globally consistent embedding.  This is a significant advantage when trying to form data driven techniques to modify the metric as we will see in Section \ref{mainresultex}.  Second, the theory of local kernels gives a geometric interpretation to many existing techniques which use local kernels such as $K(x,y) = e^{-||y-x||_{A(x)}^2}$ where $A(x)$ defines a special distance measure on the embedded data.  The theory of local kernels shows that these techniques are changing the geometry of the embedded data.  Understanding the geometric content of kernel based methods provides novel avenues for analyzing the data.  

Next we demonstrate the numerical application of a local kernel to modify the geometry of a data set, and in Section \ref{mainresultex} we will demonstrate new techniques for data driven geometric regularization using local kernels.

\subsection{Numerical example: Recovering the flat metric on a torus with a local kernel} \label{flattorussec}

In this section we show that a local kernel can recover the flat metric on a torus embedded in $\mathbb{R}^3$ with nonzero Riemannian curvature.  Let $\theta,\phi \in [0,2\pi)$ be the intrinsic coordinates of the torus. The flat metric is given simply by $g_{\theta,\phi} = \textup{Id}_{2\times 2}$, the product metric induced by the structure $T^2 = S^1 \times S^1$.  Now consider the embedding $\iota:T^2 \to \mathbb{R}^3$ given by
\[ \iota((\theta,\phi)) = \left[ \begin{array}{c} (2+\sin\theta)\cos\phi \\ (2+\sin\theta)\sin\phi \\ \cos\theta \end{array} \right] \hspace{50pt} D\iota((\theta,\phi)) = \left[ \begin{array}{c c} \cos\theta\cos\phi & -(2+\sin\theta)\sin\phi) \\ \cos\theta\sin\phi & (2+\sin\theta)\cos\phi \\ -\sin\theta & 0 \end{array} \right] \]
which induces a curved metric on the torus.  Our goal is to use a local kernel to undo the curvature induced by the embedding and recover the flat metric.

 \begin{figure}[h]
  \begin{center}
\subfigure[]{  \includegraphics[width=.31\linewidth]{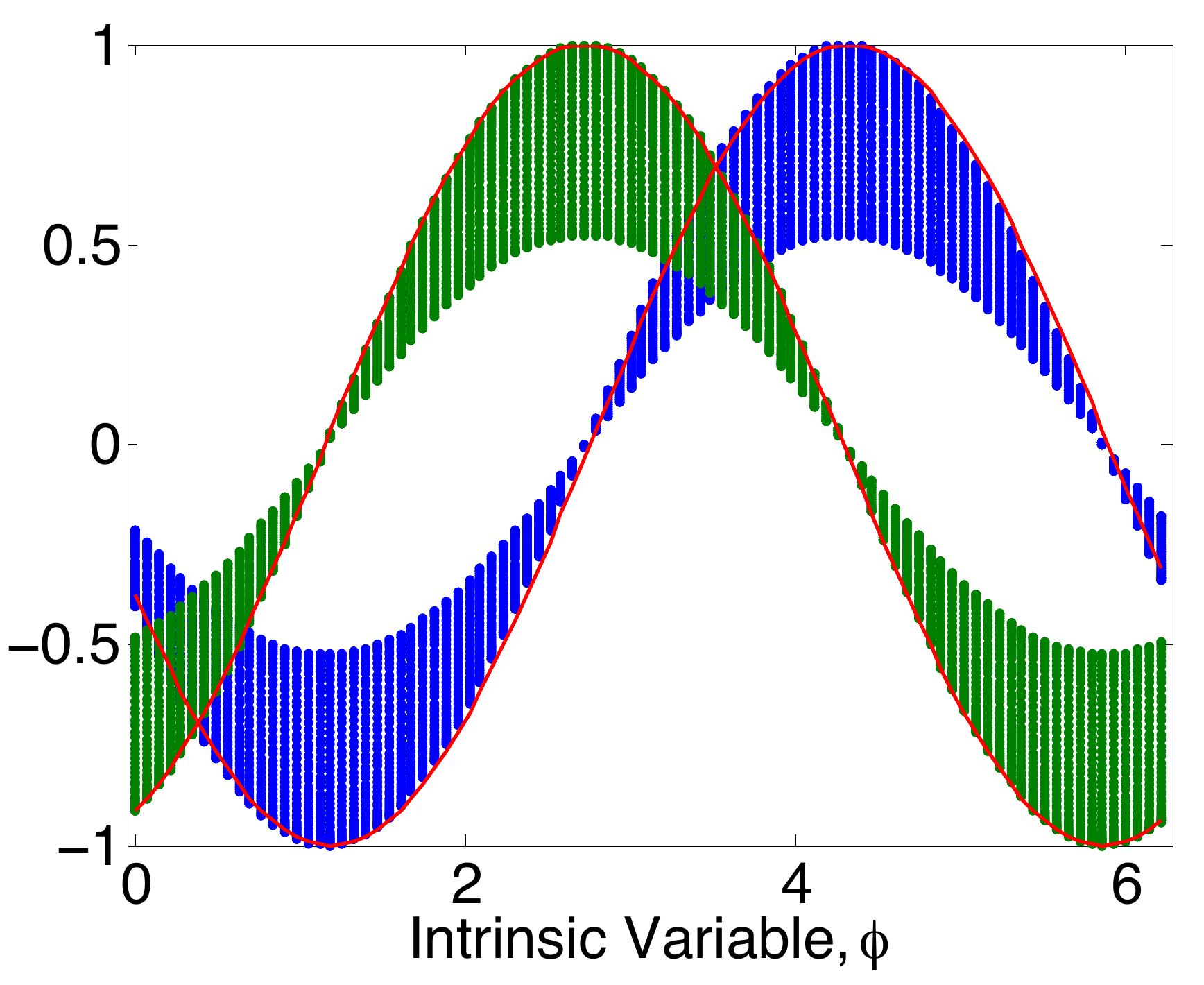}}
\subfigure[]{\includegraphics[width=.31\linewidth]{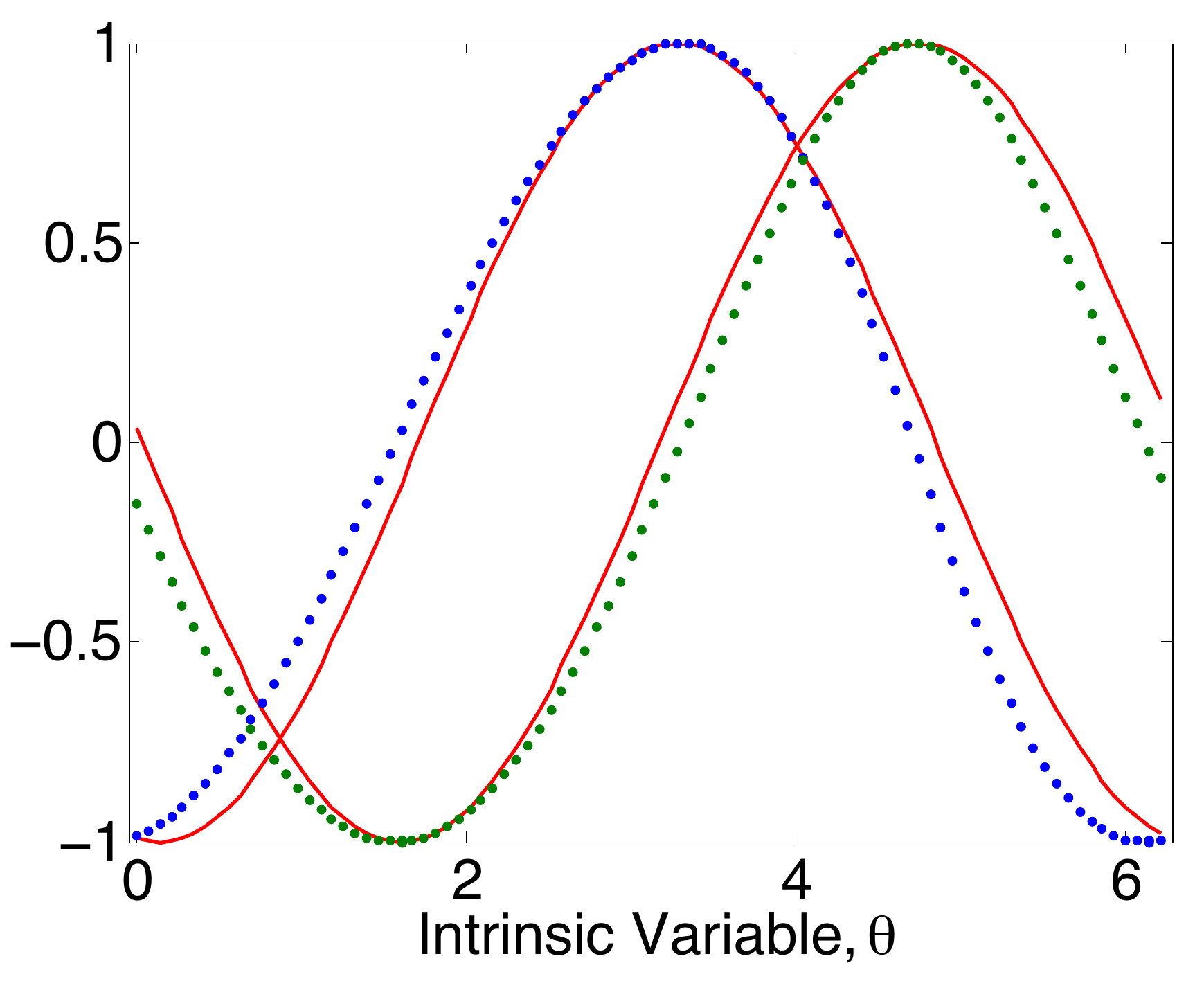}}
\subfigure[]{\includegraphics[height=.25\linewidth]{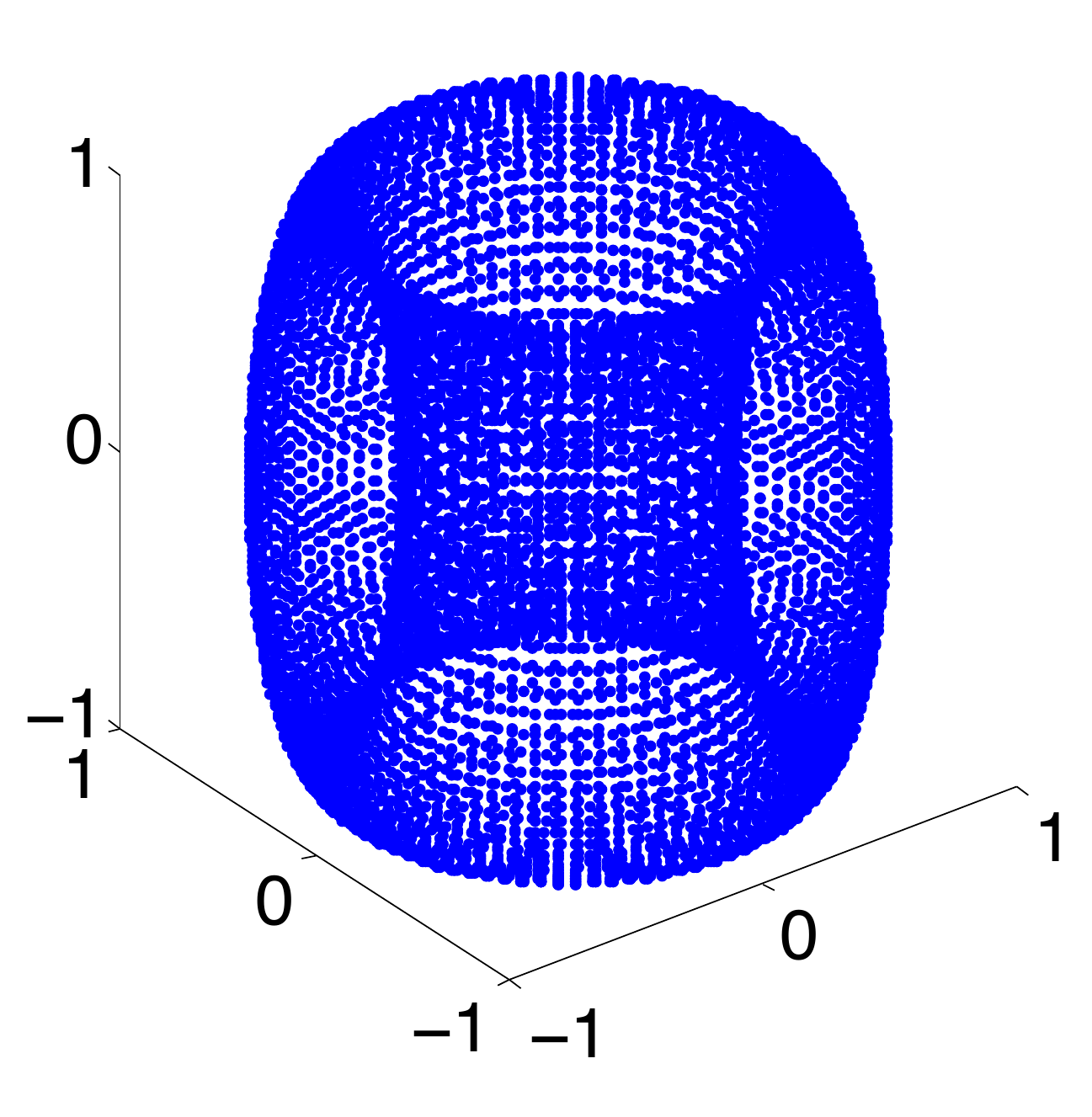}}\\
\subfigure[]{  \includegraphics[width=.31\linewidth]{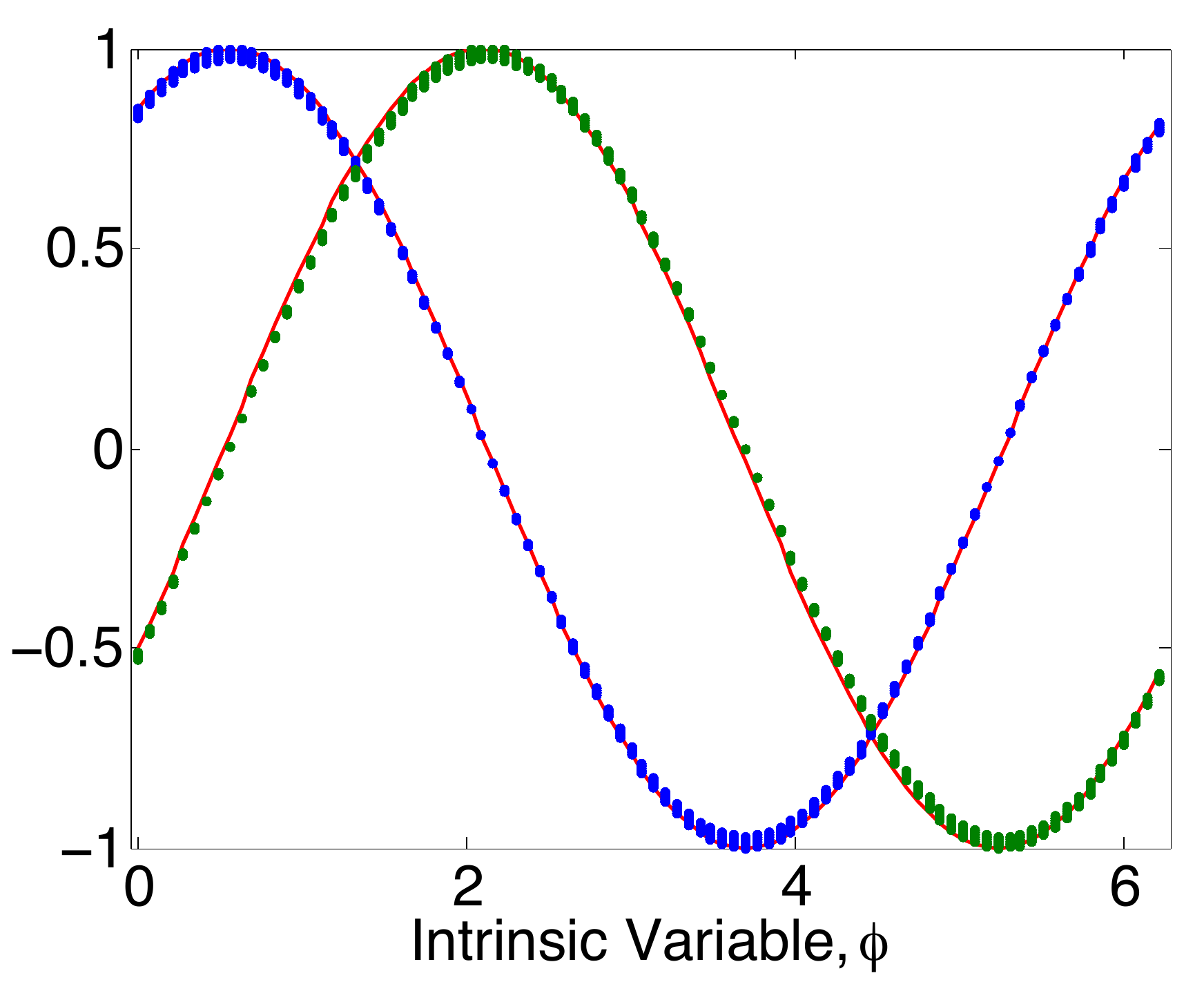}}
\subfigure[]{\includegraphics[width=.31\linewidth]{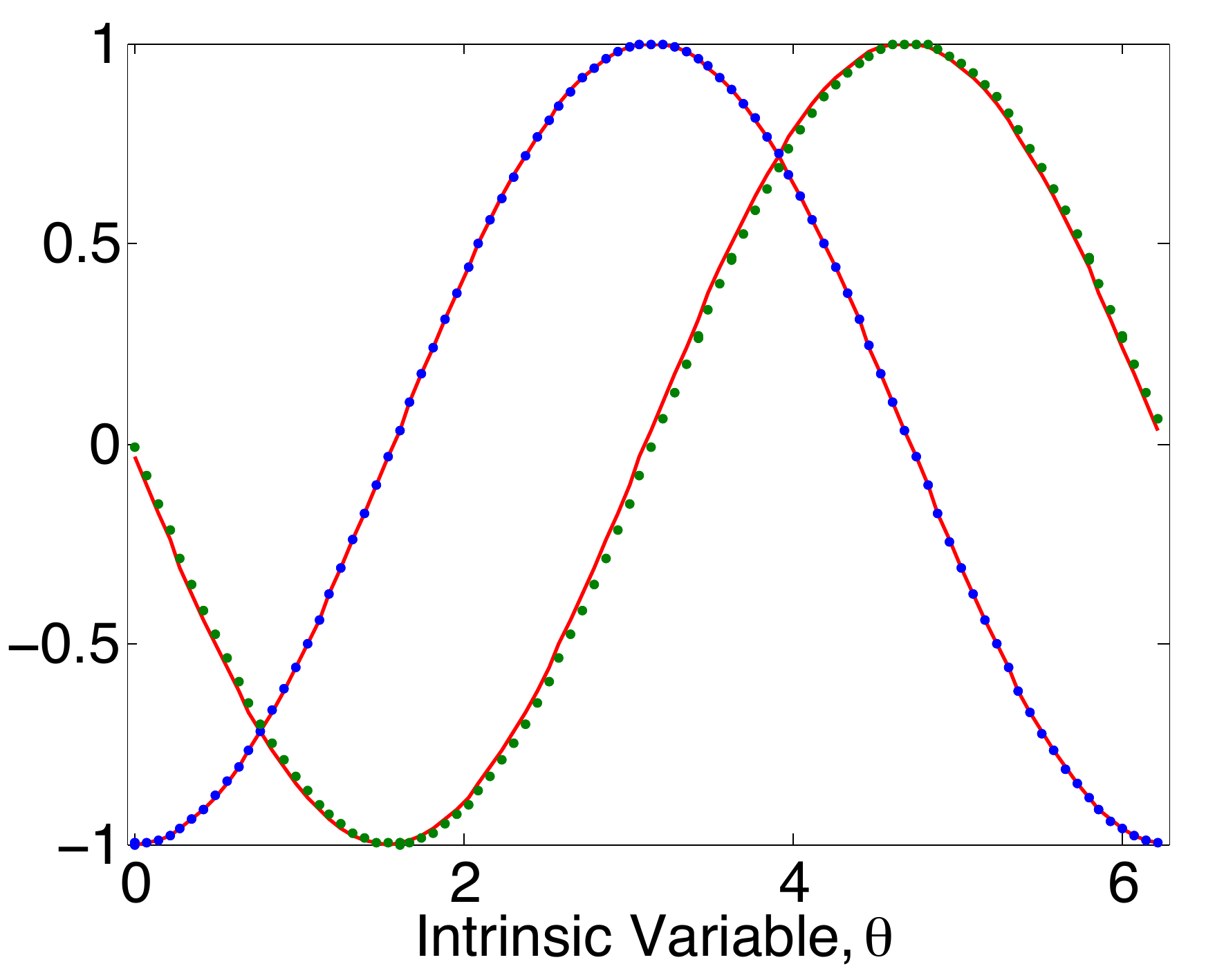}}
\subfigure[]{\includegraphics[height=.25\linewidth]{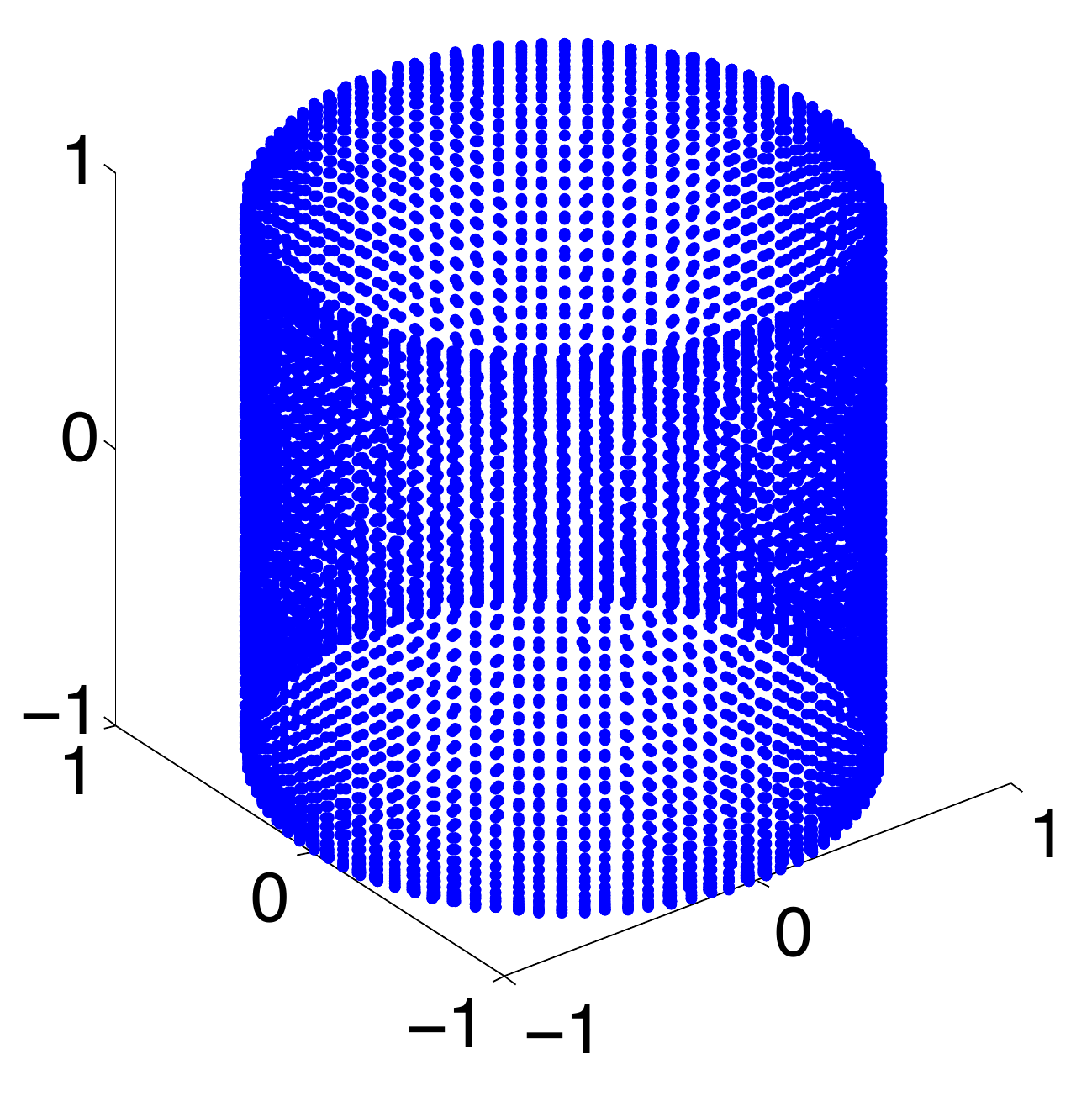}}\\
\caption{\label{flattorus} Comparison of standard diffusion maps (a-c) with local kernel approach (d-f). (a) First (blue) and second (green) eigenfunctions of the Laplacian with respect to the induced metric approximated by the diffusion maps construction; the red curves are sine functions with the same phase as the eigenfunctions. (b) Eigenfunctions five (blue) and six (green); note that all the plots contain the same number of points and the vertical spread in this plot indicates the $\theta$ dependence. (c) The diffusion maps embedding of the torus using eigenfunctions one, two, and five.  (d) Same as (a) but using eigenfunctions from the local kernel construction described in the text, Eigenfunctions one (blue) and two (green); (e) Eigenfunctions three (blue) and four (green). (f) Embedding using eigenfunctions one, two, and three.  Note that the surface shown is flat (zero Riemannian curvature) as expected but is not an embedding of the torus; this is because a smooth isometric embedding of the flat torus requires four dimensions.}
\end{center}
\vspace{.1in}
\end{figure}

We generated 8100 points on a uniform grid in $[0,2\pi]^2$ to represent the intrinsic variables and then mapped these points into $\mathbb{R}^3$ via $\iota$ to generate the observed variables.  We first applied the standard diffusion map algorithm to the observed data set with $\alpha = 1$ (since the points are not uniformly distributed on the embedded manifold) in order to approximate the first four eigenvectors of the Laplacian with respect to the curved metric from the embedding space.  In Figure \ref{flattorus} we show these eigenfunctions plotted against the intrinsic variables along with the diffusion map embedding with coordinates given by the first three eigenfunctions.  As in Section \ref{dmbackground}, the diffusion maps algorithm estimates the Laplacian with respect to the Riemannian metric induced by the embedding.  

To show that a local kernel could recover the Laplacian with respect to the flat metric, we defined the local kernel
\[ K(\epsilon,x,y) = \exp\left(-\frac{(y-x)^T A(x)(y-x)}{\epsilon} \right) \hspace{40pt} A(x) = \left(D\iota(\iota^{-1}(x))^{\dagger}\right)^TD\iota(\iota^{-1}(x))^{\dagger}. \]
With this definition, $K$ is a prototypical kernel with $\overline{K}(\epsilon,x,x+\sqrt{\epsilon}z) = e^{-z^T A(x) z}$ which implies that $C(x)^{-1/2} = D\iota(\iota^{-1}(x))^{\dagger}$ on $T_x\mathcal{M}$.  Since the metric induced by the embedding is $g = (D\iota(x))^T D\iota(x)$, Theorem \ref{maintheorem} implies that using the local kernel $K$ approximates the Laplacian with respect to the metric
\[ \hat g = C^{-1/2} g C^{-1/2} = I \]
which is the flat metric on the torus.  In order to validate Theorem \ref{maintheorem} numerically we constructed the discrete Laplacian matrix $L_{ij}$ defined by,
\[ L_{ij} \equiv \frac{2}{\epsilon}\left( \frac{\overline K(\epsilon,x_i,x_j)}{ \sum_l \overline K(\epsilon,x_j,x_l) \sum_s \frac{ \overline K(\epsilon,x_i,x_s)}{\sum_l \overline K(\epsilon,x_s,x_l)}}-\textup{Id}_{N\times N}\right). \]
Since Theorem \ref{maintheorem} says that in the limit of large data and small $\epsilon$ this matrix converges to the Laplacian with respect to $\hat g$, which is the flat metric, the eigenvectors of this matrix should approximate the eigenfunctions of $\Delta_{\hat g}$.  In Figure \ref{flattorus} we confirm this result numerically using the data set described above.  For both the diffusion maps algorithm and for the matrix $L$, we chose $\epsilon = \frac{1}{N}\sum_{i=1}^N ||x_i-x_{nn(i)}||^2$ where $nn(i)$ is the index of the nearest neighbor of $x_i$.

Of course, this kernel is not purely data driven since we have used knowledge of the embedding $\iota$ to define the local kernel.  The point of this example is simply to demonstrate numerically that a local kernel can achieve a desired change of metric without having to re-embed the data.  Note that the first four eigenfunctions of $L_{\epsilon}$ with respect to the local kernel $K(\epsilon,x,y)$, as shown in Figure \ref{flattorus}, approximate $[\sin(\theta + \theta_0),\cos(\theta+\theta_0),\sin(\phi+\phi_0),\cos(\phi+\phi_0)]$ up to phase shifts $\theta_0$ and $\phi_0$, and these are precisely the eigenfunctions of the Laplacian on the flat torus.  We note that these coordinates give an isometric embedding of the flat torus into $\mathbb{R}^4$.  

The example in this section illustrates the power of local kernels to modify the geometry of data. However, this example made use of the embedding function which is not typically known.  In the next section we will examine a data-driven approach to regularizing the geometry of data using local kernels.

\section{Data driven geometry regularization via local kernels}\label{mainresultex}

An important observation of \cite{diffusion} was that in many applications the sampling distribution is an extrinsic factor which we do not wish to influence the geometry.  However, as we have shown in Section \ref{dmbackground}, unless we know the embedding to be an isometry, not only the sampling distribution but the entire embedding geometry could be considered extrinsic.  In this section we apply a data-driven anisotropic local kernel to regularize the geometry.  

\subsection{Conformally Invariant Embedding}

The perspective of diffusion maps is that we would like to study the metric $g$ inherited from the ambient space, and thus if the data is not sampled according to the volume form of $g$, then we must remove the sampling bias.  In this section we consider an alternative explanation for the disagreement between the sampling measure and the volume form.  Using this new framework we show that it is possible to construct a kernel which is invariant to any conformal transformation of a data set.

Our new assumption will be that the data set $\{\tilde x_i\}$ was sampled uniformly on a manifold $(\mathcal{N},g_{\mathcal{N}})$ but the observed data $\{x_i\}$ is given by  a conformal isometry $\mathcal{H}:\mathcal{N} \to \mathcal{H}(\mathcal{N}) \subset \mathbb{R}^n$.  Let $\mathcal{M}=\mathcal{H}(\mathcal{N}) \subset \mathbb{R}^n$ be the observed manifold and let $g$ be the Riemannian metric that $\mathcal{M}$ inherits from the ambient space.  Since $\mathcal{H}$ is assumed to be a conformal isometry, the observed metric is given by $g = \rho g_{\mathcal{N}}$ for some positive scalar valued function $\rho$.  Moreover, considering $g_{\mathcal{N}}(x)$ and $g(x)$ as matrices, we have $g_{\mathcal{N}}(x) = D\mathcal{H}(x) g(x) D\mathcal{H}(x)$ which implies
\[ \sqrt{\textup{det}(g)} = \sqrt{\textup{det}(\rho g_{\mathcal{N}})} = \rho^{d/2} \sqrt{\textup{det}(g_{\mathcal{N}})} = \rho^{d/2} \sqrt{\textup{det}(D\mathcal{H} g D\mathcal{H} )}  = \rho^{d/2}  |D\mathcal{H}| \sqrt{\textup{det}(g)}.\]
%\[ |D\mathcal{H}| \sqrt{\textup{det}(g)} = \sqrt{\textup{det}(D\mathcal{H} g D\mathcal{H} )} = \sqrt{\textup{det}(h)} = \sqrt{\textup{det}(\rho g)} = \rho^{d/2} \sqrt{\textup{det}(g)}, \]
We conclude that $|D\mathcal{H}| = \rho^{-d/2}$.  Since we assume that the original data set was uniformly sampled on $\mathcal{N}$ with respect to $g_{\mathcal{N}}$, the $\{\tilde x_i\}$ are distributed according to the volume form $d\textup{vol}_{g_{\mathcal{N}}}(x)$.  This implies that the observed data $\{x_i\}$ are distributed according to 
\[ q(x) = \textup{det}\left(D\mathcal{H}(\mathcal{H}^{-1}(x))^{-1}\right) = \rho(\mathcal{H}^{-1}(x))^{d/2}. \]  
Using this fact, we can recover the factor $\rho$, from the conformal change of metric, as $\rho(\tilde x) = q(\mathcal{H}(\tilde x))^{2/d}$.  Finally, we can recover the original metric $g_{\mathcal{N}}$ with a local kernel $K$ such that the mean and skewness are zero and the covariance is given by $K_{ij}(y) = \rho(\mathcal{H}^{-1}(x))^{-1} = q(x)^{-2/d}$.  For example, we can use the prototypical kernel
\begin{align}\label{conformalKernel} K(x,y) = \exp\left(\frac{(x-y)^\top \rho(x)\textup{Id}_{d \times d}(x-y)}{4\epsilon} \right) = \exp\left(\frac{||x-y||^2}{4\epsilon q(x)^{-2/d}}\right) 
\end{align}
to construct the Laplacian with respect to the metric $q^{2/d}g = \rho^{-1} g = \rho^{-1}\rho g_{\mathcal{N}} = g_{\mathcal{N}}$, which is the original Riemannian metric on the unobserved manifold $\mathcal{N}$.  

\begin{examp}[Conformal isometry of the unit circle] \rm We first demonstrate the difference between the conformally invariant construction and that of standard diffusion maps.  The data is originally generated uniformly on the unit circle parameterized by $\theta \in [0,2\pi)$, in this example we choose $4000$ points $\{\theta_j = 2\pi j/4000\}_{j=1}^{4000}$.  However, the observed data lies on an ellipse $x_j = \mathcal{H}(\theta_j) = (\cos\theta_j,a\sin\theta_j)^\top$.  The volume form on the ellipse is given by
\[ d\textup{vol}(x) = \sqrt{\textup{det}(D\mathcal{H}(\theta)D\mathcal{H}(\theta)^\top)} = \sqrt{\sin^2\theta+a^2\cos^2\theta} = \sqrt{1+(a^2-1)\cos^2\theta}, \]
whereas the sampling measure $q(x)$ on the ellipse is given by
\[  q(x) = |D\mathcal{H}(\theta)|^{-1} = \frac{1}{\sqrt{1+(a^2-1)\cos^2\theta}}. \]
For $a\neq 1$, the sampling density does not agree with the volume form, and the data $\{x_j\}$ is not uniformly sampled on the ellipse. We first applied the standard diffusion map with normalization $\alpha=1$ to the data set $\{x_j\}$ to construct the Laplacian operator $\Delta_{\mathcal{M}}$ with respect to the Riemannian metric that the ellipse inherits from $\mathbb{R}^2$.  Analytically, $\Delta_{\mathcal{M}}$ can be written in $\theta$-coordinates as
\[ \Delta_{\mathcal{M}} f(\theta) = \frac{1}{\sqrt{g(\theta)}} \frac{\partial}{\partial\theta}\left(  \frac{1}{\sqrt{g(\theta)}}  \frac{\partial}{\partial\theta} f(\theta) \right) = \frac{1}{g(\theta)} \frac{\partial^2 f}{\partial\theta^2} - \frac{1}{2 g(\theta)^2}\frac{\partial g}{\partial\theta} \frac{\partial f}{\partial\theta},   \]
where $g(\theta) = 1+(a^2-1)\cos^2\theta$.  The first two nontrivial eigenfunctions are given by $\phi_1(\theta) = \sin(z(\theta))$ and $\phi_2(\theta) = \cos(z(\theta))$ where $z'(\theta) = \sqrt{g(\theta)}$.  By numerically integrating, we find $z(\theta)$ and plot the first two eigenfunctions of the ellipse in Figure \ref{conformalEllipse}, where we set $a=1/6$.  These eigenfunctions are shown to agree with the eigenfunctions produced by the diffusion maps algorithm.  By plotting these eigenfunctions as $(\phi_1(\theta_j),\phi_2(\theta_j))$ in $\mathbb{R}^2$ for $j=80l, l = 1,...,50$, we see that the diffusion maps algorithm represents the geometry by a non-uniformly sampled circle, which is isometric to the ellipse with uniform sampling. 

Next, we use the local kernel \eqref{conformalKernel}, where $q(x)$ is taken from the initial kernel density estimate produced by the standard diffusion map (see Section \ref{dmbackground}).  In Figure \ref{conformalEllipse} we show that the eigenfunctions of this kernel agree with those of the standard Laplacian on the unit circle $\Delta f = \partial^2 f / \partial \theta^2$, which are simply $\tilde\phi_1 = \sin\theta$ and $\tilde\phi_2 =\cos\theta$.  Moreover, the embedding $(\tilde\phi_1(\theta_j),\tilde\phi_2(\theta_j))$ reveals that this local kernel recovers the uniformly sampled circle.  The key difference is that the local kernel modifies the geometry in order to make it agree with the sampling measure, whereas the diffusion map ignores the sampling measure and preserves the observed geometry.  Which of these results is preferable will depend on the application.  If the sampling of the data is purely incidental then the diffusion map embedding is preferable because it preserves the geometry.  If the sampling of the data should inform the analysis, then it may be advantageous to distort the geometry in order to have a uniformly sampled manifold.

 \begin{figure}[h]
  \begin{center}
\subfigure[]{  \includegraphics[height=.3\linewidth]{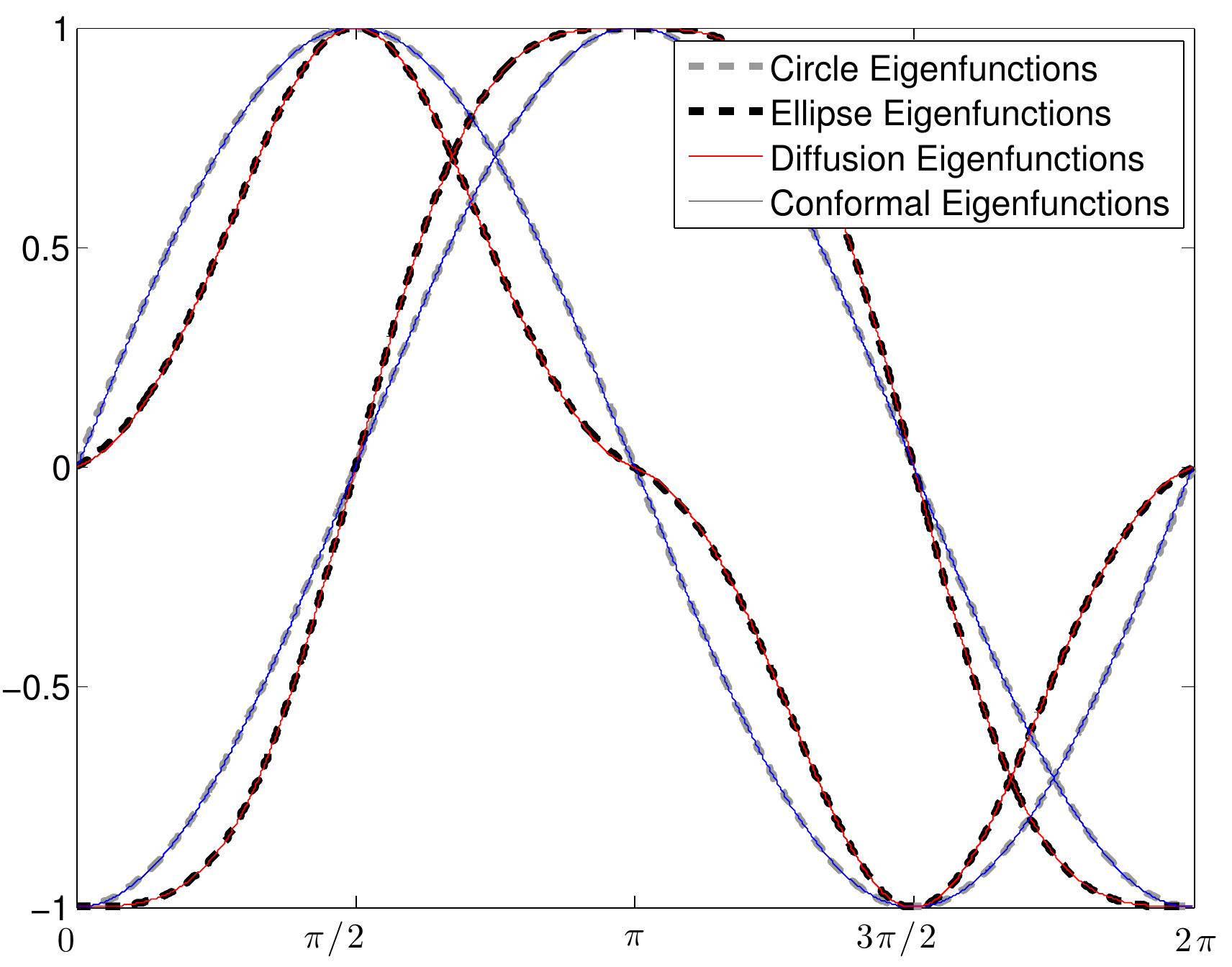}}
\subfigure[]{   \includegraphics[height=.3\linewidth]{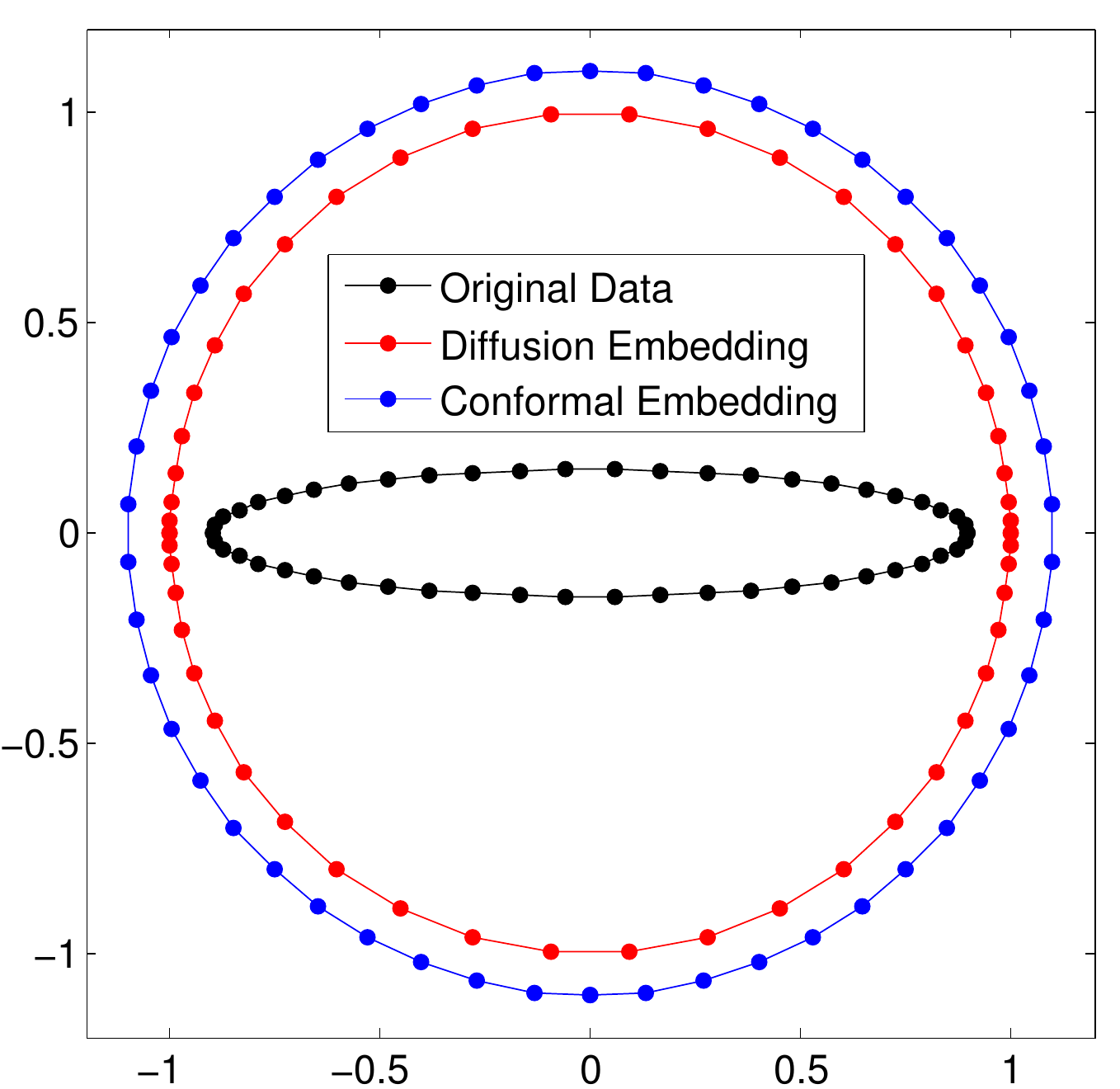}}
\caption{\label{conformalEllipse} (a) The standard diffusion map algorithm recovers the eigenfunctions of the ellipse (given by the embedding geometry), whereas the conformal map removes the latent distribution and recovers the geometry of the circle, as shown by the eigenfunctions.   The circle and ellipse eigenfunctions shown with dashed curves were computed analytically. (b) The eigenfunctions for any topological circle lie on a circle, however the conformal eigenfunctions are uniformly distributed (diameters of both embeddings were adjusted for clarity).  Plots were generated by applying the diffusion maps and conformal maps algorithms to 4000 points sampled from the ellipse with major axis length of 1 and minor axis length of $1/6$ shown in black, where every 80th point is shown to illustrate the densities.  }
%In the top two figures the ellipse is uniformly sampled, in the bottom two figures the ellipse is sampled according to the density$\frac{2-\sin\theta}{\int 2-sin\theta d\theta}$}
\end{center}
%\vspace{.1in}
\end{figure}

\end{examp}

The previous example shows that the local kernel \eqref{conformalKernel} can recover a uniformly-sampled intrinsic manifold that has been mapped by a conformal isometry before being observed.  Thus, the kernel will recover the same geometry from any two different data sets generated by conformal isometries applied to an initial data set that is uniformly sampled. This leads to an interesting application: We can use this kernel to detect when two embeddings of a data set are conformally equivalent. 

 Assume that we are given two copies of a data set, $\{y_j\} \subset \mathcal{M}_1 \subset \mathbb{R}^{m_1}$ and $\{z_j\} \subset \mathcal{M}_2 \subset \mathbb{R}^{m_2}$ where $y_j$ are sampled according to an arbitrary density $q_1(y)$.   Assume further that the second data set is actually given by a conformal isometry of the first data set, so that $z_j = \mathcal{F}(y_j)$.  In this case, applying the local kernel \eqref{conformalKernel} to $\{y_j\}$ we will find the Riemannian metric $g = q^{-2/d}g_1$ where $g_1$ is the metric $\mathcal{M}_1$ inherits from the ambient space, and the sampling of $\mathcal{M}_1$ is uniform with respect to $g$.  Moreover, since $\{z_j\}$ is given by a conformal isometry applied to $\{y_j\}$, the metric $g_2$ which $\mathcal{M}_2$ inherits from the ambient space is given by $g_2 = \rho g_1$ for some scalar function $\rho$.  This implies that $g_2 = \rho g_1 = \rho q^{2/d} g$ so that $g_2$ is conformally equivalent to $g$, and $\{z_j\}$ have sampling density $q_2 = \rho^{d/2}q_1$.  Applying the local kernel \eqref{conformalKernel} to $\{z_j\}$ we find the metric $q_2^{-2/d}g_2 = \rho^{-1}q_1^{-2/d}g_2 = g$, which is the same metric as the local kernel \eqref{conformalKernel} found on $\{y_j\}$.  This shows that the local kernel \eqref{conformalKernel} is invariant under any conformal isometry of a given data set.  In the next example we demonstrate this algorithm for two conformally equivalent tori in $\mathbb{R}^3$.
 
 \begin{examp}[Conformally equivalent tori]\rm
 In this example we consider the torus of Section \ref{flattorussec} and a conformally equivalent torus given by
 \[ \tilde \iota((\theta,\phi)) = \left( (\sqrt{2}+\sin\theta)\cos\phi, (\sqrt{2}+\sin\theta)\sin\phi, \cos\theta \right)^\top. \]
 We note that the choice of the radii $2$ and $\sqrt{2}$ is necessary to insure the tori are conformally equivalent.  To test the conformally invariant embedding, we generated $10000$ points on a uniform grid $(\theta_i,\phi_i) \in [0,2\pi)^2$ and mapped them into $\mathbb{R}^3$ via $x_i = \iota(\theta_i,\phi_i)$ and $\tilde x_i = \tilde \iota(\theta_i,\phi_i)$, as shown in Figure \ref{conformalembeddings}(a).  We first applied the conformally invariant embedding developed above to each data set, and found the first 10 eigenvectors of $L_{\epsilon}$ constructed from the local kernel $K$ in \eqref{conformalKernel}.  We used these eigenvectors to form a conformally invariant embedding with coordinates, $\Phi(x_i) = (\varphi_1(x_i),...,\varphi_{10}(x_i))^\top$ and $\tilde \Phi(\tilde x_i) = (\tilde \varphi_1(\tilde x_i),...,\tilde\varphi_{10}(\tilde x_i))^\top$.  Ordinary least squares finds the optimal linear map between these coordinate systems, which maps the coordinates $\tilde\Phi(\tilde x_i)$ into the conformally invariant embedding space for $\{x_i\}$. We then applied diffusion maps (with $\alpha=1$) to both data sets, and using the first 10 diffusion coordinates,  built a linear map from the diffusion coordinates of $\tilde x_i$ to those of $x_i$.  Figure \ref{conformalembeddings} shows pictorially that the conformally invariant embedding coordinates are the same for the two conformally equivalent data sets. Because the tori are conformal, the conformally invariant geometries are isometric, which implies that the eigenfunctions are identical up to an orthogonal linear transformation, as shown by the agreement in Figure \ref{conformalembeddings}(b).  On the other hand, the standard diffusion map represents the geometry that each data set inherits from the embedding shown in (a), and since these are not isometric, there is no linear map between their respective eigenfunctions, as shown by the disagreement in Figure  \ref{conformalembeddings}(c).

  \begin{figure}[h]
  \begin{center}
\subfigure[]{  \includegraphics[height=.17\linewidth]{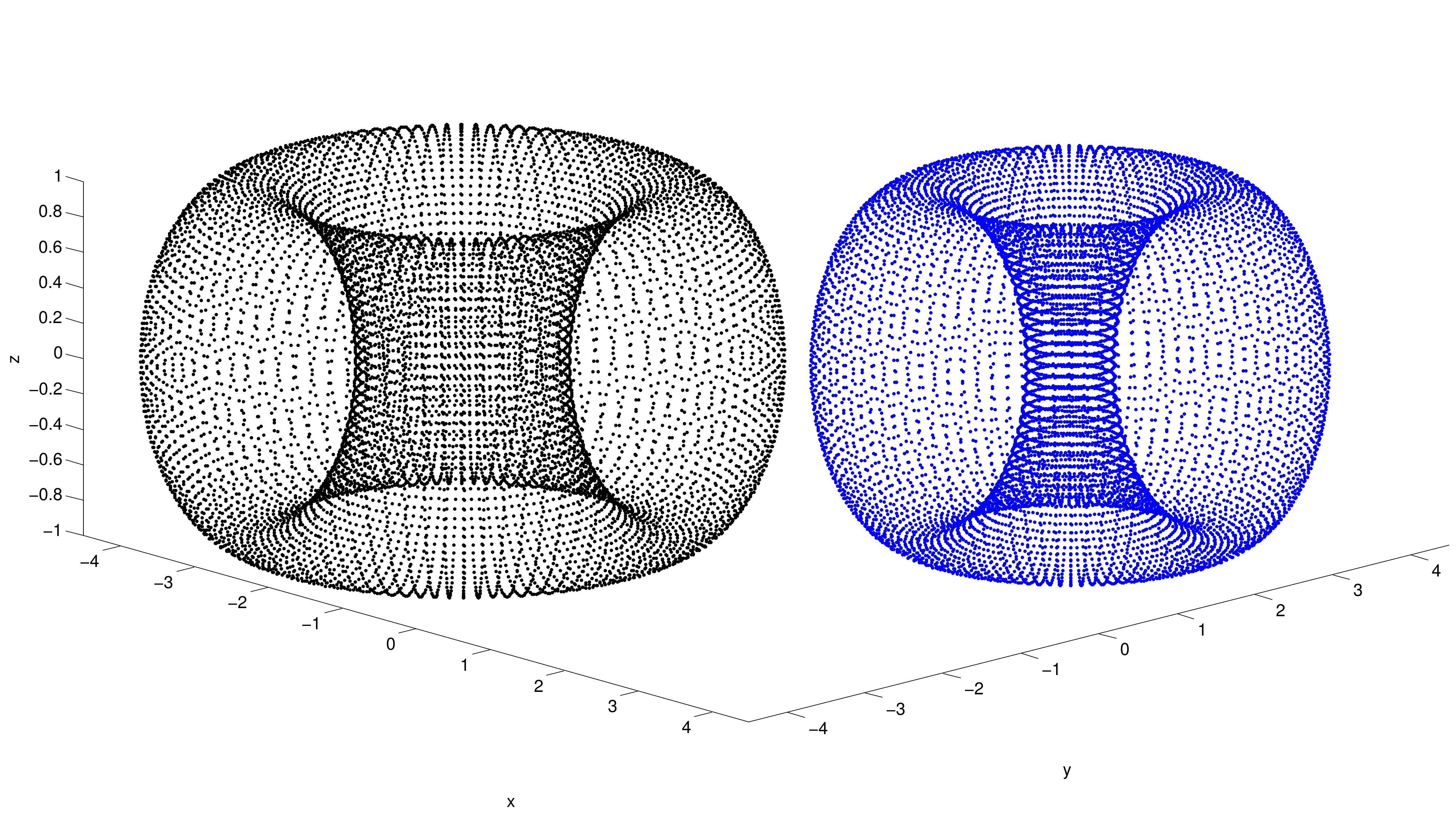}}
\subfigure[]{   \includegraphics[height=.17\linewidth]{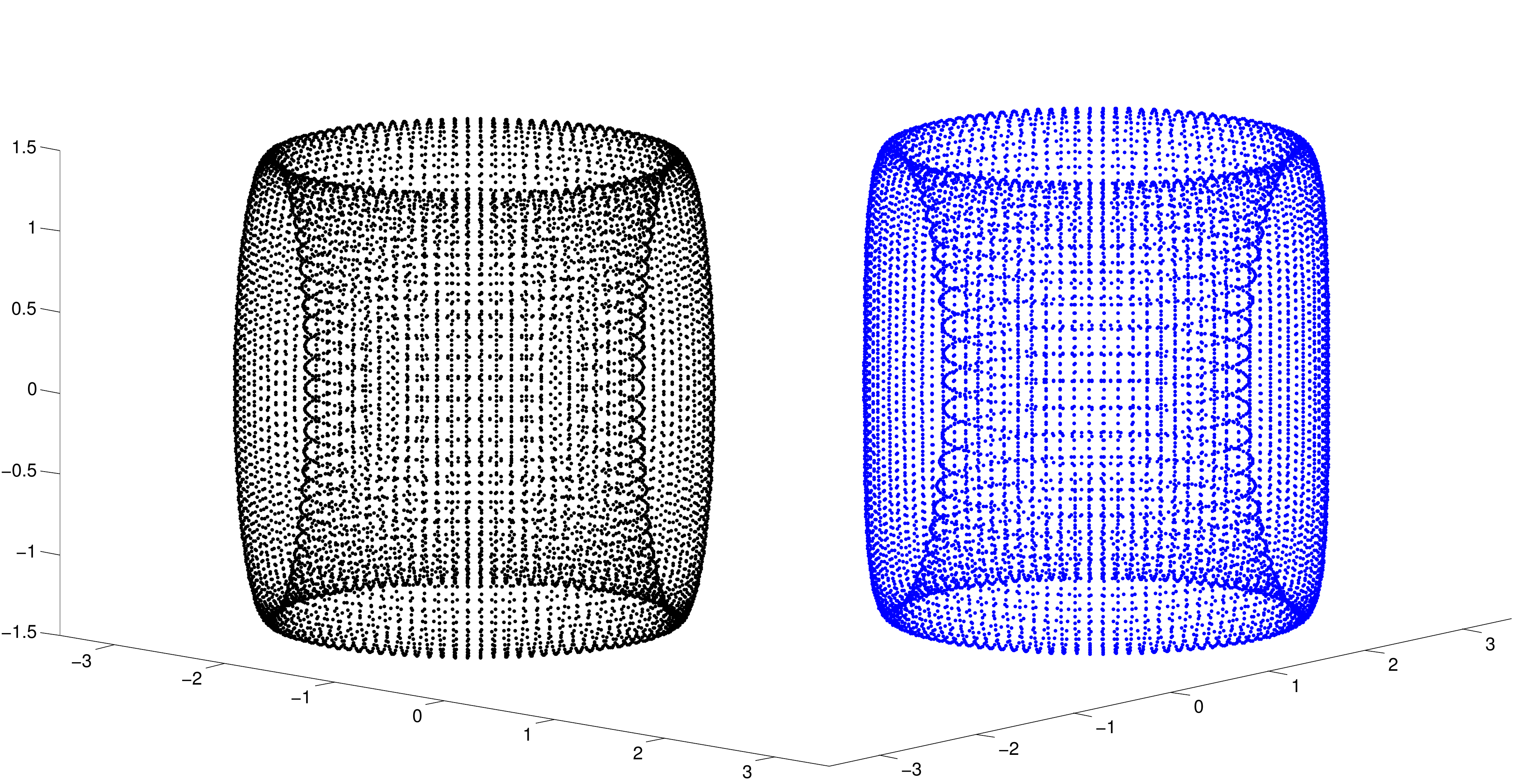}}
\subfigure[]{     \includegraphics[height=.17\linewidth]{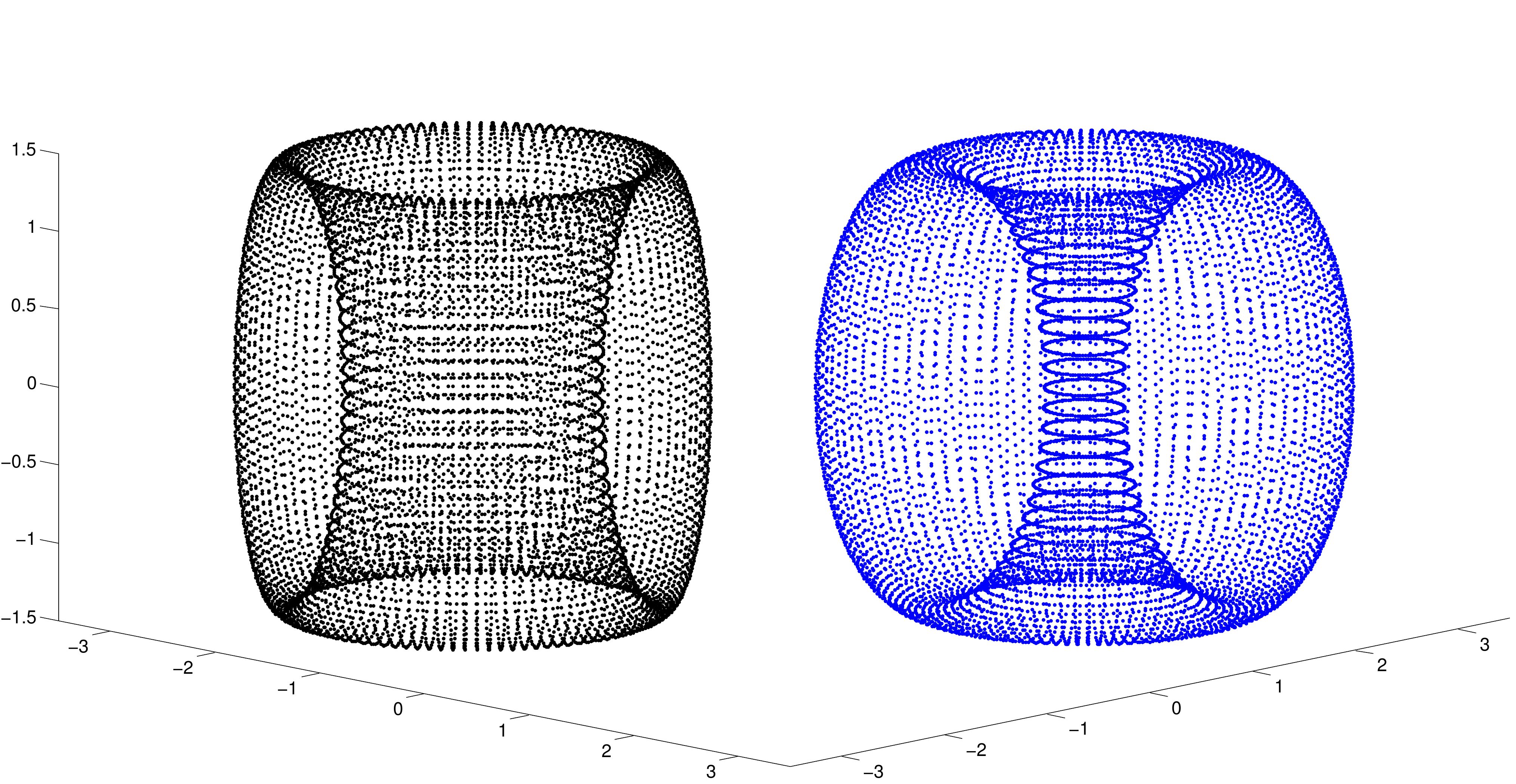}}
\caption{\label{conformalembeddings} (a) Original data sets $\{x_i\}$ (left, black) and $\{\tilde x_i\}$ (right, blue) lying on conformally equivalent tori.  (b) Conformally invariant embedding of $\{x_i\}$ (left,black) and the linearly mapped coordinates of the conformally invariant embedding of $\{\tilde x_i\}$ (right, blue).  (c) Diffusion map embedding of $\{x_i\}$ (left, black) and the linearly mapped diffusion coordinates of $\{\tilde x_i\}$ (right, blue).  }
\end{center}
\end{figure}

 \end{examp}

\subsection{Global diffeomorphism reconstruction}\label{diffeomorphismsec}

In this section we assume that we are given two datasets that are related by a global diffeomorphism, and show how to use a local kernel to reconstruct the diffeomorphism.  In particular, assume that $\tilde x_i \in \mathcal{N}\subset \mathbb{R}^m$ and $x_i = \mathcal{H}(\tilde x_i)$, where $\mathcal{H}:\mathcal{N} \hookrightarrow \mathbb{R}^{n}$ is an unknown diffeomorphism, so that $x_i$ lie on $\mathcal{M} = \mathcal{H}(\mathcal{N})$.  The key will be that we have a correspondence between individual points in the data sets.  This is often the case when we have multiple time series observations of some intrinsic state, such as assorted simultaneous observations of a dynamical system. 

To reconstruct the global diffeomorphism, we will use a local kernel to push-forward the Riemannian metric from $\mathcal{N}$ onto $\mathcal{M}$ via the correspondence between the data sets.  With this metric on $\mathcal{M}$, the two manifolds are isometric, which implies that their Laplacians have the same eigenvalues, and that the associated eigenfunctions of any eigenvalue are related by an orthogonal transformation \cite{laplacianBook}.  We can then easily estimate this orthogonal transformation using linear least squares.  A related technique was introduced in \cite{hirn} for mapping between diffusion maps embeddings; the difference here is that such a linear map provably exists since we use local kernels to change the geometry so that the manifolds are isometric. 

In order to push the metric forward from $\cal N$ onto $\cal M$, we need to estimate $D\mathcal{H}$ and then apply Theorem \ref{maintheorem} to $x_i$ on $\mathcal{M}$ with the prototypical kernel
\[ K(\epsilon,x,y) = \exp\left(-\frac{(y-x)^\top D\mathcal{H}(x)^\top D\mathcal{H}(x)(y-x)}{2\epsilon}\right). \]
To estimate the matrix $D\mathcal{H}(x_i)$, we take the nearest neighbors $\{x_j\}$ of $x_i$ and use the correspondence to find $\tilde x_i = \mathcal{H}^{-1}(x_i)$ and the neighbors $\tilde x_j = \mathcal{H}^{-1}(x_j)$.  Note that $\tilde x_j$ may not be the nearest neighbors of $\tilde x_i$ due to the diffeomorphism.  We then construct the weighted vectors
\[ v_j = \exp\left(-||x_j-x_i||^2/\epsilon\right)(x_j-x_i) \hspace{50pt} \tilde v_j = \exp\left(-||x_j-x_i||^2/\epsilon\right)(\tilde x_j-\tilde x_i), \] 
and define $D\mathcal{H}_i$ to be the $m\times n$ matrix which minimizes $\sum_j ||\tilde v_i - D\mathcal{H}_i v_j ||^2$.  We note that the exponential weight is used to localize the vectors; otherwise the linear least squares problem would try to preserve the longest vectors $x_j-x_i$, which do not represent the tangent space well.  Notice that the same exponential factor is used on both the $v_j$ and the $\tilde v_j$ so that all the distortion of distances is represented linearly.  We can now approximate $D\mathcal{H}(x_i)^\top D\mathcal{H}(x_i) \approx D\mathcal{H}_i^\top D\mathcal{H}_i$, so that numerically we evaluate the local kernel
\begin{align}\label{diffeokernel} K(\epsilon,x_i,x_j) = \exp\left(-\frac{|| D\mathcal{H}_i(x_j-x_i)||^2}{2\epsilon}\right). \end{align}
Using Theorem \ref{maintheorem} we approximate the Laplacian $\Delta_{\tilde g} = \mathcal{\mathcal{H}}^*\Delta_{g_{\mathcal{N}}}$ on $\mathcal{M}$ and use the standard diffusion maps algorithm (with $\alpha=1$) to approximate the Laplacian $\Delta_{g_{\mathcal{N}}}$ on $\mathcal{N}$.  Since $(\mathcal{M},\tilde g)$ and $(\mathcal{N},g_{\mathcal{N}})$ are isometric, the eigenvalues will be the same (up to the precision of the discrete approximation) and the eigenfunctions will be related by orthogonal transformations.  Thus, we can build a linear map $H$ between the eigenfunctions by ordinary least squares.  

Using this linear map between the eigenfunctions we can represent the global diffeomorphism.  By taking sufficiently many eigenfunctions $\varphi_l$ and $\tilde \varphi_l$ on the respective manifolds, the eigenfunctions can be considered coordinates of an embeddings $\Phi(x) = (\varphi_1(x),...,\varphi_{\hat n}(x))^\top$ and $\tilde \Phi(\tilde x) = (\tilde \varphi_1(\tilde x),...,\tilde \varphi_{\hat m}(\tilde x))^\top$.  We thus have the commutative diagram

  \begin{center}
$$  \begin{array}[c]{ccc}
\mathcal{N}&\xrightarrow{\ \ \ \ \ \mathcal{H}\ \ \ \ \ }&\mathcal{M}\\ \\
\left\downarrow\rule{0cm}{.5cm}\right.   \scriptstyle{\tilde\Phi}&&\left\downarrow\rule{0cm}{.5cm}\right. \scriptstyle{\Phi}\\ \\
L^2(\mathcal{N},g_\mathcal{N}) \approx \mathbb{R}^{\hat n}&\xrightarrow{\ \ \ \ \ H\ \ \ \ \ }&L^2
(\mathcal{M},{\tilde g}) \approx \mathbb{R}^{\hat m}
\end{array}$$
\end{center}
where $H = \Phi \circ \mathcal{H} \circ \tilde \Phi^{-1}$ is linear.  Using various standard techniques we can extend the maps $\tilde \Phi$ and $\Phi$ and their inverses to new data points and so the map $H$ represents the global diffeomorphism $\mathcal{H}$ in the eigenfunction coordinates.  In the following example we demonstrate this technique on a torus in $\mathbb{R}^3$ and compare to constructing a linear map in diffusion coordinates.

 \begin{figure}[h]
  \begin{center}
\subfigure[]{  \includegraphics[height=.2\linewidth]{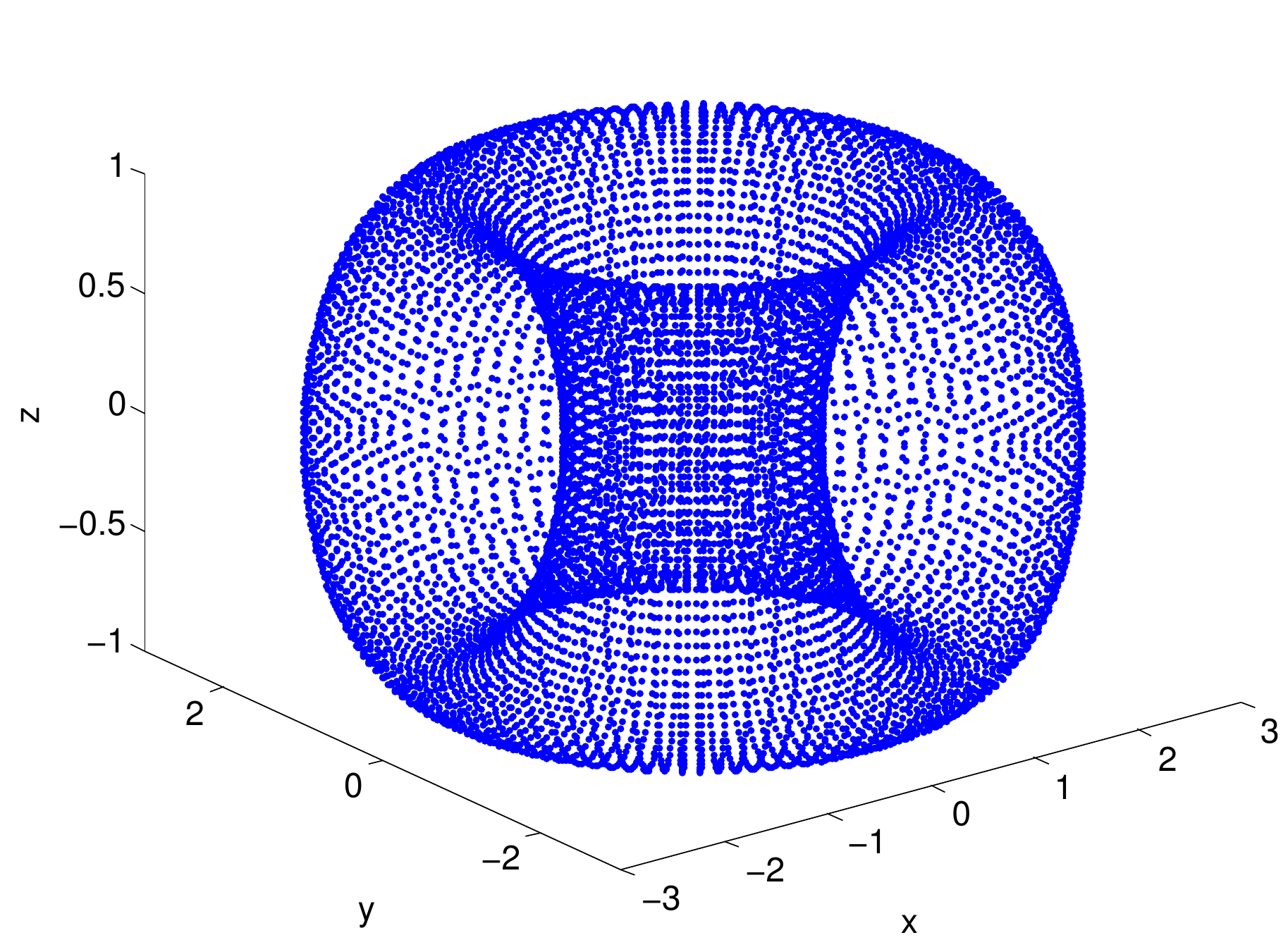}}
\subfigure[]{   \includegraphics[height=.2\linewidth]{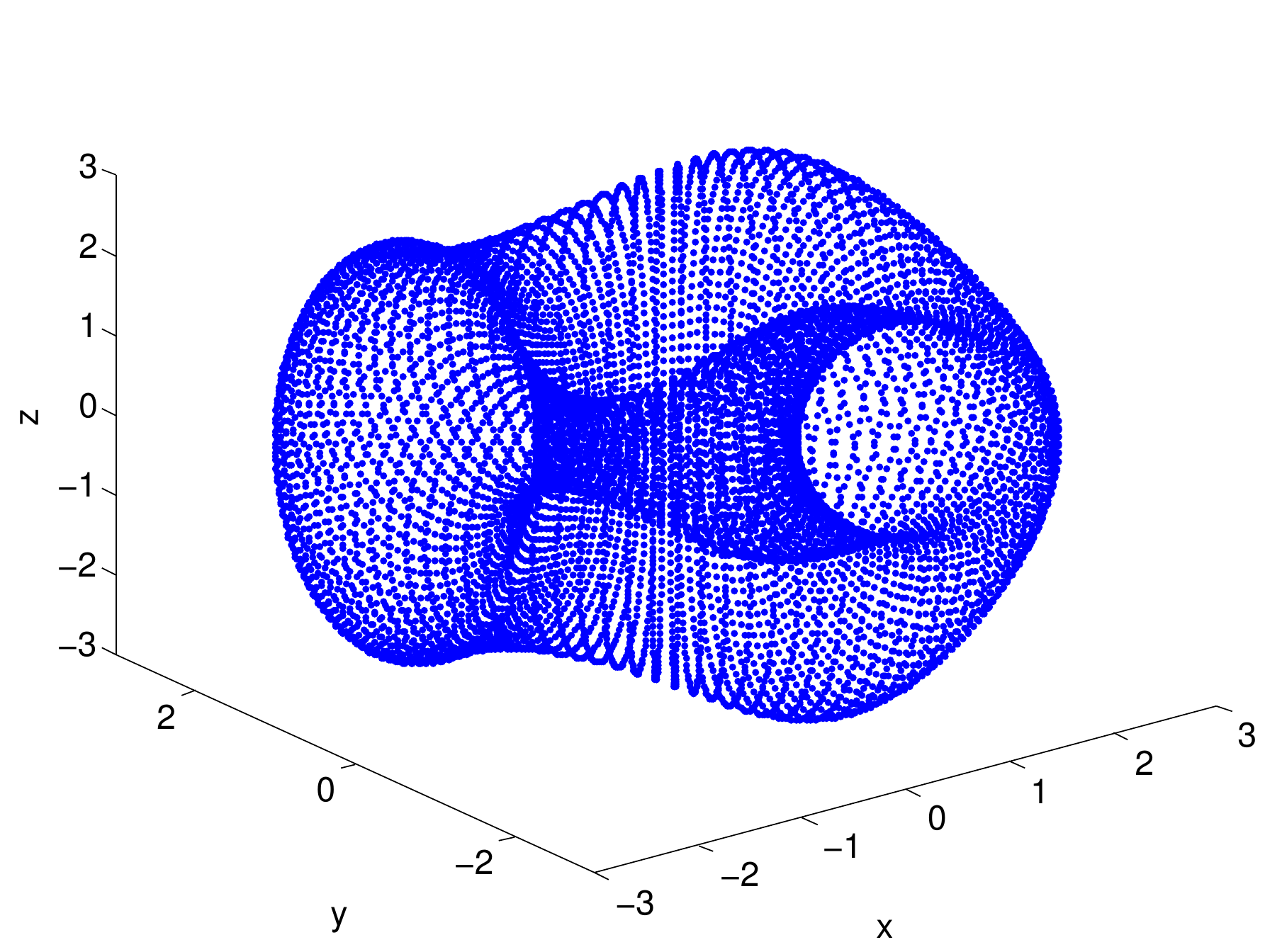}}\\
\subfigure[]{     \includegraphics[height=.3\linewidth]{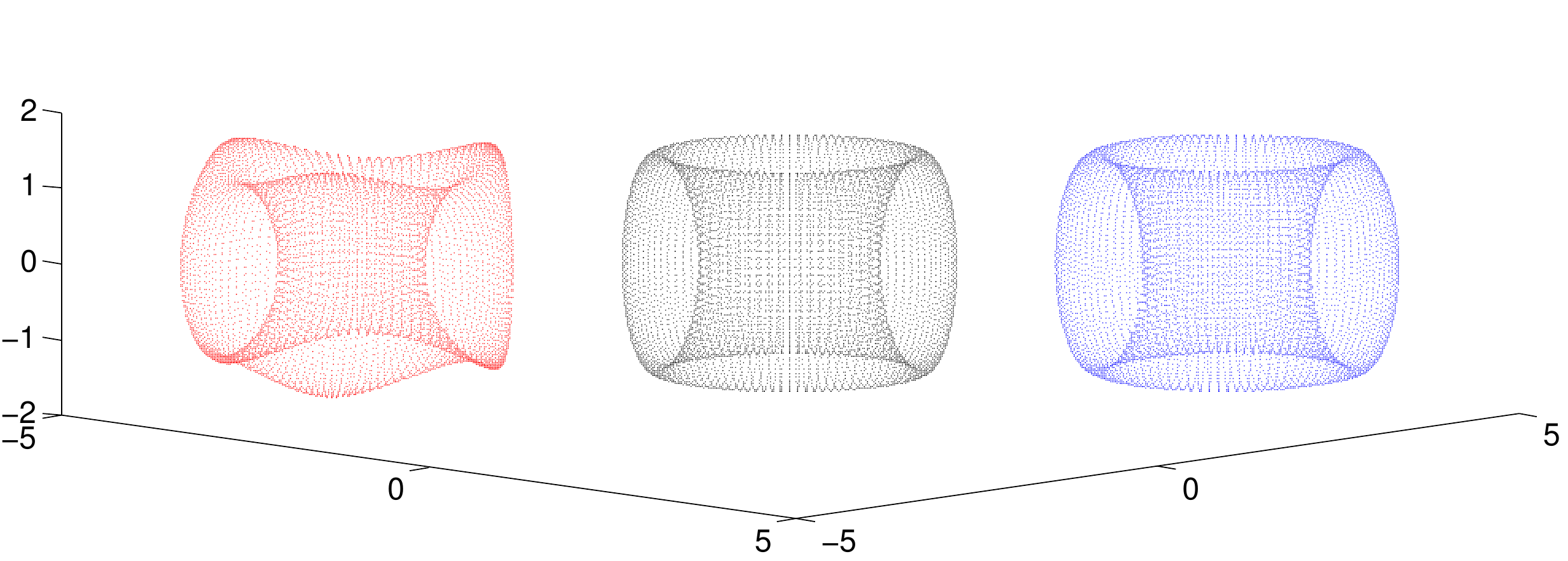}}
\caption{\label{diffeomorphism} (a) Original data set $\{\tilde x_i\}$ on $\mathcal{N}$ and (b) the diffeomorphic images $\{x_i  = \mathcal{H}(\tilde x_i)\}$ on $\mathcal{M}$. (c) Diffusion map coordinates for $\{\tilde x_i\}$ (center, black) compared to the linearly mapped diffusion coordinates for $\{x_i\}$ (red, left) and the linearly mapped eigenfunction coordinates $H\Phi(x_i)$ (blue, right).  Since the geometries which (a) and (b) inherit from their embeddings are only diffeomorphic and not isometric, the eigenfunctions produced by the diffusion map cannot be linearly mapped as shown by the disagreement between the red and black embeddings in (c).  By using the local kernel \eqref{diffeokernel} we push the geometry of (a) onto the data set (b) using the known correspondence between the points as shown by the agreement of the diffusion map embedding of (a), shown in (c, black), with the linearly mapped eigenfunctions of the local kernel applied to (b), shown in (c, blue).}
\end{center}
\end{figure}

\begin{examp}[Reconstructing a global diffeomorphism of the torus]\rm In this example we let $\mathcal{N}$ be the torus of Section \ref{flattorussec} with Euclidean coordinates $(x,y,z) = \iota((\theta,\phi))$ in $\mathbb{R}^3$, and we let 
\[ \mathcal{H}(x,y,z) = [x,y,(2+\sin(3\tan^{-1}(y/x))/2)z]^\top \]
be the unknown diffeomorphism.  The two tori are shown in Figure \ref{diffeomorphism}(a) and (b), respectively, where $10000$ points were generated on a uniform grid $(\theta_i,\phi_i)\in [0,2\pi]^2$ and where $\tilde x_i = \iota(\theta_i,\phi_i)$, $x_i = \mathcal{H}(\tilde x_i)$.

We applied the standard diffusion map to $\tilde x_i$ to estimate $\Delta_{g_{\mathcal{N}}}$ and the first 10 eigenfunctions, $\tilde \Phi(\tilde x_i)$, which represent the geometry which the data set $\tilde x_i$ inherits from the ambient space shown in Figure \ref{diffeomorphism}(a).  We then applied the above algorithm to $x_i = \mathcal{H}(\tilde x_i)$ (note that the algorithm also requires $\tilde x_i$) to estimate $\Delta_{\tilde g}$ and the first 10 eigenfunctions, $\Phi(x_i)$, which represents the geometry $\tilde g$ on the data set shown in Figure \ref{diffeomorphism}(b).  Note that the geometry $\tilde g$ is not the same as the geometry which $\{x_i\}$ inherits from the ambient space.  Instead, we have use the local kernel \eqref{diffeokernel} to push the geometry of the data set $\{\tilde x_i\}$ onto the data set $\{x_i\}$ which means that $(\mathcal{M},\tilde g)$ and $(\mathcal{N},g_{\mathcal{N}})$ are isometric as shown above.  Since the manifolds with these geometries are isometric, the eigenfunctions of their respective Laplacians are identical up to an orthogonal transformation.  To verify this numerically, we used least squares optimization to estimate the linear transformation $H$ from the eigenfunctions $\Phi(x_i)$ to the eigenfunctions $\tilde \Phi(\tilde x_i)$.  We then use $H$ to map the eigenfunction coordinates $\Phi(x_i)$ into the diffusion map coordinate space for $\mathcal{N}$.  In Figure \ref{diffeomorphism}(c), we compare the diffusion maps coordinates $\tilde \Phi_i(\tilde x)$ (black, middle of figure) with $H\Phi(x_i)$ (blue, right side of figure).  We also attempted to linearly map the diffusion map coordinates for $\{x_i\}$ into those for $\{\tilde x_i\}$, and we show the result in Figure \ref{diffeomorphism}(c)  (right side) for comparison.  Note that because the local kernel puts an isometric geometry onto $\mathcal{M}$ the eigenfunctions of $\Delta_{\tilde g}$ can be linearly mapped onto the diffusion map eigenfunctions for $\mathcal{N}$.  However, because $\mathcal{M}$ and $\mathcal{N}$ are not isometric with respect to the geometries inherited from their respective embeddings (shown in Figures \ref{diffeomorphism}(a) and \ref{diffeomorphism}(b) respectively), there is no linear map between the diffusion eigenfunctions of these data sets.

\end{examp}

\section{Conclusion}

In this article, we have extended the geometric perspective of the original diffusion map construction to a class of kernels that is large as feasible. In fact, we show that any kernel with exponential decay leads naturally to a Laplacian with respect to some Riemannian geometry.  The exponential decay is crucial, to constrict all information to flow through local interactions.

Theorems \ref{maintheorem} and \ref{mainconverse} show that every symmetric local kernel corresponds to a Riemannian geometry and conversely, any Riemannian geometry can be represented with an appropriate local kernel. This opens up all kernels with exponential decay to exploitation by the whole range of geometric tools.  On the other hand, local kernels can be classified by their intrinsic geometry, and every intrinsic geometry will be accessible by a prototypical kernel.  This shows that, in the limit of large data, one can always use a prototypical kernel; indeed this will typically be advantageous since the prototypical kernels are skew-free, which leads to fast convergence to the limiting operators.  

In Section \ref{mainresultex} we showed how to construct an embedding which is invariant under conformal transformations.  We then showed how to use a local kernel to reconstruct a global diffeomorphism between two data sets.  One potential application of this result is to dynamical systems, since there are often various observable physical aspects of the system.  A theorem of Takens \cite{Takens,SYC} states that the method of time-delay coordinates can be used to reconstruct a state space which is equivalent to the full dynamical system up to a diffeomorphism.  This means that each observed time series can be used to create a diffeomorphic copy of the dynamical system.  In the case where the dynamical system lies on an attractor, we can use this method to map each data set into any other coordinates, or given new data in some observation, this data can be mapped into other observations spaces.  We should caution that the method of Section \ref{diffeomorphismsec} relies on approximating a sufficient number of eigenfunctions of the Laplacian to represent the entire manifold, and for a high-dimensional manifold (such as the attractor of a complex dynamical system) this would require such a large amount of data that it would typically be computationally infeasible.  However, even in this case, the technique in Section \ref{diffeomorphismsec} may still be valuable as a coarse map between observation spaces.

Several further applications of this generalization are apparent.  In \cite{DMDC}, it was found that the traditional attractor reconstruction methods using delay coordinates biases the manifold toward stable components. A natural candidate for intrinsic geometry on a dynamical system is the Lyapunov geometry \cite{RDSbook}, because it is invariant to diffeomorphic observations, such as delay-coordinates. If the Lyapunov geometry is the goal, then the embedding geometry is largely extrinsic, and needs to be removed. Using an appropriate local kernel, it should be possible to recover this intrinsic geometry.  Beyond building diffeomorphisms between data sets, it may also be desirable to isolate differences in data sets.  One possibility would be identifying subsets of each data set which are diffeomorphic, however it is unclear how to identify these subsets.  Alternatively, if the difference is represented in certain components of the data it may be possible to identify these components as those which are not captured in the global diffeomorphism reconstruction (which in this case would only be an approximate diffeomorphism). Moreover, in many applications certain `features' of interest have already been identified and this should inform the geometry in the local kernel.  In this paper we have shown how to design a local kernel which recovers a conformally invariant geometry; if this approach could be generalized to recover geometries invariant to the known features, this geometry could be used to find the most important aspects of the data beyond those already represented.

Image and video analysis provide another example. Each image, or video frame, can be considered a vector of pixels in a high-dimensional data space. Such an embedding treats pixels on the opposite side of a frame the same as nearby pixels, which is often a poor assumption. There is a need to apply a more informative geometric prior. In fact, this idea is crucial for any data of interest that is accompanied by metadata. By allowing the metric to depend on the metadata, local kernels enable a large array of options  to make use of {\it a priori} connections.

\section{Acknowledgements} 

We thank two anonymous reviewers for suggestions that significantly improved the manuscript.
This research was partially supported by NSF grants DMS-1216568, DMS-1250936, and CMMI-130007.

\bibliographystyle{plain}
\bibliography{localkernelsbib}

\begin{thebibliography}{10}

\bibitem{RDSbook}
L.~Arnold.
\newblock {\em Random Dynamical Systems}.
\newblock Springer-Verlag New York, Inc., 1998.

\bibitem{BN}
M.~Belkin and P.~Niyogi.
\newblock Laplacian eigenmaps for dimensionality reduction and data
  representation.
\newblock {\em Neural Computation}, 15(6):1373--1396, 2003.

\bibitem{DMDC}
T.~Berry, J.~R. Cressman, Z.~Greguri\'{c} Feren\v{c}ek, and T.~Sauer.
\newblock Time-scale separation from diffusion-mapped delay coordinates.
\newblock {\em SIAM J. Appl. Dyn. Sys}, 12:618--649, 2013.

\bibitem{BH14}
Tyrus Berry and John Harlim.
\newblock Variable bandwidth diffusion kernels.
\newblock {\em submitted to Applied and Computational Harmonic Analysis}, 2014.

\bibitem{diffusionslowmanifold}
R.~Coifman, R.~Erban, A.~Singer, and I.~Kevrekidis.
\newblock Detecting intrinsic slow variables in stochastic dynamical systems by
  anisotropic diffusion maps.
\newblock {\em PNAS}, 106:10000000--3, 2009.

\bibitem{diffusion}
R.~Coifman and S.~Lafon.
\newblock Diffusion maps.
\newblock {\em Appl. Comp. Harmonic Anal.}, 21:5--30, 2006.

\bibitem{diffusionreductioncoords}
R.~Coifman, S.~Lafon, M.~Maggioni, B.~Nadler, and I.~Kevrekidis.
\newblock Diffusion maps, reduction coordinates, and low dimensional
  representation of stochastic systems.
\newblock {\em SIAM Journal for Multiscale Modeling \& Simulation}, 7:842--864,
  2008.

\bibitem{diffcoords}
R.~Coifman, S.~Lafon, B.~Nadler, and I.~Kevrekidis.
\newblock Diffusion maps, spectral clustering and reaction coordinates of
  dynamical systems.
\newblock {\em Appl. Comp. Harmonic Anal.}, 21:113--127, 2006.

\bibitem{hirn}
Ronald~R Coifman and Matthew~J Hirn.
\newblock Diffusion maps for changing data.
\newblock {\em Applied and Computational Harmonic Analysis}, 36(1):79--107,
  2014.

\bibitem{dsilva}
Carmeline~J. Dsilva, Ronen Talmon, Neta Rabin, Ronald~R. Coifman, and
  Ioannis~G. Kevrekidis.
\newblock Nonlinear intrinsic variables and state reconstruction in multiscale
  simulations.
\newblock {\em The Journal of Chemical Physics}, 139(18):--, 2013.

\bibitem{kpca1}
Jihun Ham, Daniel~D. Lee, Sebastian Mika, and Bernhard Sch\"{o}lkopf.
\newblock A kernel view of the dimensionality reduction of manifolds.
\newblock In {\em Proceedings of the Twenty-first International Conference on
  Machine Learning}, ICML '04, pages 47--, New York, NY, USA, 2004. ACM.

\bibitem{Hein05}
Matthias Hein, Jean yves Audibert, and Ulrike~Von Luxburg.
\newblock From graphs to manifolds - weak and strong pointwise consistency of
  graph laplacians.
\newblock In {\em Proceedings of the 18th Conference on Learning Theory
  (COLT)}, pages 470--485. Springer, 2005.

\bibitem{Jost}
J.~Jost.
\newblock {\em Riemannian Geometry and Geometric Analysis}.
\newblock Springer-Verlag Berlin, 2002.

\bibitem{Kushnir2012}
Dan Kushnir, Ali Haddad, and Ronald~R. Coifman.
\newblock Anisotropic diffusion on sub-manifolds with application to earth
  structure classification.
\newblock {\em Applied and Computational Harmonic Analysis}, 32(2):280 -- 294,
  2012.

\bibitem{laplacianBook}
S.~Rosenberg.
\newblock {\em The Laplacian on a Riemannian manifold}.
\newblock Cambridge University Press, 1997.

\bibitem{diffdist}
M.~Saerens, F.~Fouss, L.~Yen, and P.~Dupont.
\newblock The principal components analysis of a graph, and its relationships
  to spectral clustering.
\newblock {\em Lecture Notes in Artificial Intelligence No. 3201, 15th European
  Conference on Machine Learning (ECML)}, pages 371--383, 2004.

\bibitem{SYC}
T.~Sauer, J.A. Yorke, and M.~Casdagli.
\newblock Embedology.
\newblock {\em Journal of Statistical Physics}, 65(3):579--616, 1991.

\bibitem{kpca2}
Bernhard Sch\"olkopf, Alexander Smola, and Klaus-Robert M\"oller.
\newblock Kernel principal component analysis.
\newblock In Wulfram Gerstner, Alain Germond, Martin Hasler, and Jean-Daniel
  Nicoud, editors, {\em Artificial Neural Networks - ICANN'97}, volume 1327 of
  {\em Lecture Notes in Computer Science}, pages 583--588. Springer Berlin
  Heidelberg, 1997.

\bibitem{SingerEstimate}
A.~Singer.
\newblock From graph to manifold laplacian: The convergence rate.
\newblock {\em Applied and Computational Harmonic Analysis}, 21:128--134, 2006.

\bibitem{Singer2009}
A.~Singer, R.~Erban, I.~G. Kevrekidis, and R.~Coifman.
\newblock Detecting intrinsic slow variables in stochastic dynamical systems by
  anisotropic diffusion maps.
\newblock {\em PNAS}, 106(38):16090--16095, 2009.

\bibitem{singerWu}
A.~Singer and H.-T. Wu.
\newblock Vector diffusion maps and the connection laplacian.
\newblock {\em Communications on Pure and Applied Mathematics},
  65(8):1067--1144, 2012.

\bibitem{Singer2008}
Amit Singer and Ronald~R. Coifman.
\newblock Non-linear independent component analysis with diffusion maps.
\newblock {\em Applied and Computational Harmonic Analysis}, 25(2):226 -- 239,
  2008.

\bibitem{Singer22092009}
Amit Singer, Radek Erban, Ioannis~G. Kevrekidis, and Ronald~R. Coifman.
\newblock Detecting intrinsic slow variables in stochastic dynamical systems by
  anisotropic diffusion maps.
\newblock {\em Proceedings of the National Academy of Sciences},
  106(38):16090--16095, 2009.

\bibitem{adaptedDiffusion}
A.~Szlam, M.~Maggioni, and R.~Coifman.
\newblock Regularization on graphs with function-adapted diffusion processes.
\newblock {\em J. Mach. Learn. Res.}, 9:1711--1739, 2008.

\bibitem{Takens}
F.~Takens.
\newblock Detecting strange attractors in turbulence.
\newblock In David Rand and Lai-Sang Young, editors, {\em {I}n: {D}ynamical
  Systems and Turbulence, {W}arwick, {E}ds. {R}and, {D}. and {Y}oung,
  {L}.-{S}.}, volume 898 of {\em Lecture Notes in Mathematics}, pages 366--381.
  Springer Berlin / Heidelberg, 1981.

\bibitem{talmon2012}
Ronen Talmon, Dan Kushnir, Ronald~R Coifman, Israel Cohen, and Sharon Gannot.
\newblock Parametrization of linear systems using diffusion kernels.
\newblock {\em Signal Processing, IEEE Transactions on}, 60(3):1159--1173,
  2012.

\bibitem{probabilisticEstimates}
Ronen Talmon, St\'ephane Mallat, Hitten Zaveri, and Ronald~R. Coifman.
\newblock Manifold learning for latent variable inference in dynamical systems.
\newblock {\em Research Report YALEU/DCS/TR-1491, available at
  http://cpsc.yale.edu/sites/default/files/files/tr1491.pdf}, 2014.

\bibitem{Ting2010}
Daniel Ting, Ling Huang, and Michael~I. Jordan.
\newblock An analysis of the convergence of graph laplacians.
\newblock In {\em Proceedings of the 27th International Conference on Machine
  Learning (ICML)}, 2010.

\end{thebibliography}

\end{document}